\documentclass[11pt]{article}
\usepackage{amsmath}
\usepackage{graphicx}
\usepackage{psfrag}
\usepackage{epsf}
\usepackage{enumerate}

\usepackage[sort,nocompress]{cite}

\usepackage{mathtools}
\mathtoolsset{showonlyrefs}

\usepackage{url} 
\usepackage{amsthm}
\usepackage{xpatch}
\usepackage{subcaption}
\usepackage{bm}

\usepackage{longtable}
\usepackage{mathrsfs}
\usepackage{tikz}
\usepackage{color} 
\usepackage{amssymb}
\usepackage{latexsym}
\usepackage{xr}
\usepackage{float}
\usepackage{multirow}
\usepackage{dsfont}
\usepackage{enumitem}
\usepackage{booktabs}
\usepackage{dsfont}

\usepackage[ruled]{algorithm2e}
\usepackage{lscape}
\usepackage{rotating}
\numberwithin{equation}{section}

\usepackage{xcolor}
\definecolor{lightblue}{HTML}{044E9E}
\usepackage[breaklinks=true, hyperfootnotes = true,colorlinks,citecolor=lightblue,urlcolor=lightblue, linkcolor=lightblue]{hyperref}

\newcommand{\veps}{\varepsilon}
\newcommand{\1}{\boldsymbol{1}}
\newcommand{\ind}{\mathds{1}}
\newcommand{\R}{\mathbb{R}}

\newcommand{\EE}{\operatorname{E}}
\newcommand{\RR}{\mathbb{R}}

\newcommand{\ZZ}{\mathbb{Z}}
\newcommand{\NN}{\mathbb{N}}

\newcommand{\clc}{\mathcal{C}}
\newcommand{\clr}{\mathcal{R}}

\newcommand{\clb}{\mathcal{B}}
\newcommand{\clh}{\mathcal{H}}
\newcommand{\cln}{\mathcal{N}}
\newcommand{\cle}{\mathcal{E}}

\newcommand{\clf}{\mathcal{F}}

\newcommand{\clt}{\mathcal{T}}
\newcommand{\cll}{\mathcal{L}}
\newcommand{\mfh}{\mathfrak{H}}
\newcommand{\mfc}{\mathfrak{C}}

\newcommand{\bl}{\boldsymbol{l}}

\newcommand{\clp}{\mathcal{P}}

\newcommand{\bfl}{{\boldsymbol{l}}}
\newcommand{\bfs}{{\boldsymbol{s}}}
\newcommand{\bfr}{{\boldsymbol{r}}}

\newcommand{\E}{\operatorname{E}}

\newcommand{\Cov}{\operatorname{Cov}}

\newcommand{\cli}{\mathcal{I}}
\newcommand{\clx}{\mathcal{X}}
\newcommand{\covop}{\mathcal{T}}

\newcommand{\Om}{\Omega}
\newcommand {\op}{\operatorname{op}}
\newcommand {\rk}{\operatorname{rank}}
\newcommand {\tr}{\operatorname{Tr}}
\newcommand {\Id}{\operatorname{Id}}

\definecolor{expcol}{rgb}{1.0,0.0,0.5}
\definecolor{ecol}{rgb}{0.0, 0.0, 1.0}

\makeatletter
\xpatchcmd{\proof}{\@addpunct{.}}{\@addpunct{:}}{}{}
\makeatother

\newtheorem{theorem}{Theorem}[section]

\newtheorem{lemma}{Lemma}[section]

\theoremstyle{definition}
\newtheorem{remark}{Remark}[section]
\newtheorem{example}{Example}[section]
\newtheorem{defn}{Definition}[section]
\newtheorem{assumption}{Assumption}[section]

\DeclareFontFamily{U}{mathx}{\hyphenchar\font45}
\DeclareFontShape{U}{mathx}{m}{n}{<-> mathx10}{}
\DeclareSymbolFont{mathx}{U}{mathx}{m}{n}
\DeclareMathAccent{\widebar}{0}{mathx}{"73}

\newcommand{\mockalph}[1]{}

\long\def\MD#1{{\color{black}#1}}

\allowdisplaybreaks

\begin{document}

\def\spacingset#1{\renewcommand{\baselinestretch}%
{#1}\small\normalsize} \spacingset{1}

\newtheorem*{assumptionBIC*}{\assumptionnumber}
\providecommand{\assumptionnumber}{}
\makeatletter
\newenvironment{assumptionBIC}[2]
 {%
  \renewcommand{\assumptionnumber}{Assumption #1#2}%
  \begin{assumptionBIC*}%
  \protected@edef\@currentlabel{#1#2}%
 }
 {%
  \end{assumptionBIC*}
 }
\makeatother


\title{Breuer-Major Theorems for Hilbert Space-Valued Random Variables
}

\author{
Marie-Christine D\"uker \\ Technical University of Munich                             \and
Pavlos Zoubouloglou      \\ University of Münster}
\date{\today}

\maketitle

\bigskip

\begin{abstract}
\noindent
Let $\{X_k\}_{k\in\mathbb Z}$ be a stationary Gaussian process with values in a separable Hilbert space $\mathcal H_1$, and let $G:\mathcal H_1\to\mathcal H_2$ be a measurable map into another separable Hilbert space $\mathcal H_2$. We derive a central limit theorem for the centered normalized partial sums of the Hilbert space-valued subordinated process $\{G[X_k]\}_{k\in\mathbb Z}$. Our result holds under either of two sets of sufficient conditions, formulated in terms of the transformation $G$ and the temporal and cross-sectional dependence structure of $\{X_k\}_{k\in\mathbb Z}$. These conditions coincide in finite dimensions but lead to genuinely different phenomena in the infinite-dimensional setting. The proof relies on the recently developed Fourth Moment Theorem on Hilbert spaces, leveraging tools from the infinite-dimensional Malliavin--Stein framework. We also provide continuous-time and quantitative versions of the central limit theorem. In a series of examples, we recover and strengthen limit theorems for a wide array of statistics relevant in functional data analysis, and present, as an application of our result, a novel limit theorem in the framework of neural operators.
\end{abstract}

\section{Introduction} \label{sec:intro}
Consider a stationary Gaussian process $\{ X_{k}\}_{k \in \ZZ}$, defined on a complete probability space $(\Omega, \mathcal{F}, P)$ and taking values in a separable Hilbert space $\mathcal{H}_1$. Let $G: \clh_1 \to \clh_2$ be a measurable function mapping $\clh_1$ into a (possibly different) separable Hilbert space $\clh_2$, where $\clh_1, \clh_2$ are equipped with inner products $\langle \cdot, \cdot \rangle_{\clh_i}$ and induced norms $\| \cdot \|_{\clh_i}$, $i=1,2$. 
This work aims to find conditions on the mapping $G$ and the temporal and cross-sectional correlation structure of the underlying process $\{ X_{k} \}_{k \in \ZZ}$ that ensure a Central Limit Theorem (CLT), i.e., the weak convergence of the normalized partial sums 
\begin{equation} \label{eq:partial_sum_G}
    S_n = \frac{1}{\sqrt{n}} \sum_{k=1}^n \left( G[X_k] - \E G[X_k] \right), \hspace{0.2cm}
    n \in \NN,
\end{equation}
to a Gaussian random variable in the topology of $\clh_2$, as $n \to \infty$. Here, $\{G[X_k]\}_{k \in \ZZ}$ is known as the class of \textit{subordinated Gaussian processes}. Without loss of generality, we assume in Sections \ref{sec:intro}--\ref{app:quant_version} that $\E G[X_1]=0$.

For $\{X_k\}$ being independent and identically distributed (i.i.d.) random variables and taking values in a Hilbert space $\clh$, Varadhan \cite{Var62} was the first to prove a CLT for the normalized partial sums \eqref{eq:partial_sum_G}.
Following \cite{Var62}, a new line of work developed, aiming to understand CLTs in infinite dimensions under more general dependence structures. The literature on CLTs for Hilbert space-valued processes typically assumes that the sequence of random variables $\{X_k\}$ either admits a linear representation (e.g., \cite{duk18,Mer03,Mer97,jirak2018,Mer96}), or meets suitable mixing conditions (e.g., \cite{Mer96, MalOst83,Ded03,jirak2018}). 

For $\clh_1,\clh_2$ being finite-dimensional, the behavior of $\{S_n\}_{n \in \NN}$ is by now well understood for Gaussian processes $\{X_k\}$ and under general assumptions on $G$, avoiding linearity or mixing conditions.
The seminal work of Breuer and Major \cite{breuer1983central} ($\clh_1 = \clh_2 = \RR$) showed the convergence of \eqref{eq:partial_sum_G} to a Gaussian law whenever $\E |G(X_1)|^2 < \infty$ and $\{G(X_k)\}_{k \in \ZZ}$ exhibits short-range dependence, in the sense that its autocovariance (equivalently, autocorrelation) functions are absolutely summable. Such theorems are customarily referred to as Breuer-Major theorems. Later, \cite{arcones1994limit} considered the multivariate setting $\clh_1 = \RR^d , \clh_2 = \RR$ with $\E |G(X_1)|^2 < \infty$ and $\E X_1=0$, $\operatorname{Cov}(X_1) = \operatorname{Id}_{d}$ (with $\operatorname{Id}_{d}$ the $d\times d$ identity matrix); assuming further that the autocorrelation (equivalently, autocovariance) functions satisfy
\begin{equation} \label{eq:ass-finite-dim}
\sum_{v \in \ZZ} \sup_{r,s=1,\dots,d} | 
\phi_{rs}(v)|^q < \infty, 
\hspace{0.2cm}
\text{ with }
\hspace{0.2cm}
\phi_{rs}(v) \doteq \E X^{(r)}_1 X^{(s)}_{1+v},
\end{equation} 
where $X^{(r)}_k, r = 1,\dots, d$, denotes the $r$-th component of $X_k$ and $q$ the Hermite rank of the function $G$ to be defined (in a more general setting) in Definition \ref{def:Hermiterank}
below. Since then, several other cases for finite-dimensional $\clh_1,\clh_2$ have been investigated, including $\clh_2 = \RR^m, m \ge 1$; see \cite{ho1990limiting, bai2013multivariate,giraitis1985clt}.

The proof techniques used in the seminal work \cite{breuer1983central} and its respective follow-up articles are based on moment and cumulant computations using diagram formulae. An alternative, more modern approach is based on a pairing of Malliavin calculus and Stein’s lemma to derive quantitative CLTs under different distances implying weak convergence for subordinated Gaussian processes. We refer to \cite{nourdin2012normal} for a detailed outline of the tools that are used. This machinery has since then been leveraged to prove (quantitative) Breuer-Major theorems in numerous settings, including
$\clh_1 = \RR^d, \clh_2 = \RR^m$ \cite{nourdin2011quantitative,KuzNua19,NouRos14,NouPecRev10,NouPecYan19,nourdin2015optimal,Nou21}; random fields, i.e., $n \in \RR^p$ in \eqref{eq:partial_sum_G} \cite{NouPec09}, and functional settings \cite{nourdin2020functional,CamNouNua20}.
Note that, for the functional setting, the former work refers to continuous-time processes, and the latter work \cite{nourdin2020functional} to the functional convergence of processes $\frac{1}{\sqrt{n}} \sum_{k=1}^{\lfloor n \cdot \rfloor} G(X_k)$, with $G : \RR \to \RR$. In particular, they are both different from our setting  \eqref{eq:partial_sum_G}. For future reference we use the term \textit{continuous (time) CLT} to refer to the convergence of $\frac{1}{\sqrt{n}} \sum_{k=1}^{\lfloor n \cdot \rfloor} G[X_k]$, to avoid confusion with CLTs for random variables in function spaces.

The work \cite{Bou20} was the first to adapt the Malliavin-Stein approach to the context of random variables taking values in separable Hilbert spaces. Their motivation stemmed from the problem of proving continuous CLTs to a Gaussian random variable in a function space in the case $\clh_1 = \clh_2 = \RR$, thus bypassing the usual procedure of proving convergence of the finite-dimensional distributions and tightness of the family of random variables $\{S_n\}_{n \in \NN}$ in the respective space. Unfortunately, as was pointed out in \cite{BasBurCamPec25}, the work \cite{Bou20} employed a metric that does not metrize convergence in distribution in the topology of the respective Hilbert space. In the recent article \cite{duker2025fourth}, we were able to obtain a correct fourth-moment theorem in the so-called $d_2$ distance; see Section \ref{subsec:prelim-infinite} below. 

The first version of the present work also relied on \cite{Bou20}, hence yielding inconclusive results. In the current version, we leverage the tools developed in \cite{duker2025fourth} to prove a Breuer-Major theorem for \eqref{eq:partial_sum_G} allowing for general Hilbert spaces $\clh_1,\clh_2$. There are three main technical difficulties associated with our proofs. 

The first difficulty is in finding a chaotic decomposition for the process $\{S_n\}_{n \in \NN}$. To be more precise, we decompose the Hilbert space $L^2(\clh_1,\gamma:\clh_2)$, where $\gamma$ is a suitable Gaussian measure on $\clh_1$ and $L^2$ is the usual space of square integrable functions from $\clh_1$ to $\clh_2$, into its chaotic components; see Section \ref{sec:Wiener-Ito-chaos}. The arguments in Section \ref{sec:Wiener-Ito-chaos} extend the classical Wiener-It\^o decomposition (Section 2 in \cite{nourdin2012normal}) for elements in $L^2(\clh_1,\gamma:\clh_2)$. The decomposition allows us to obtain tractable forms for the Hermite coefficients, which can then be used for applications; see Section \ref{sec:applications}.

The second difficulty is to impose suitable conditions on $G$ and on the temporal and cross-sectional correlation structure of the process $\{X_k\}$. Throughout, for elements \(x,y\) of a Hilbert space $\clh$, we use the tensor convention
$(x\otimes y)h \doteq \langle h,y\rangle_{\clh} x $.
Define the autocovariance operator of $X_k$ by
\begin{equation} \label{eq:Gamma-op}
    \Gamma(v):\clh_1 \to \clh_1, 
    \hspace{0.2cm} 
    \Gamma(v)[h] \doteq \E \left( \langle X_1, h \rangle_{\clh_1} X_{1+v}\right)
    =
    \E(X_{1+v}\otimes X_1)[h], \quad v \in \ZZ.
\end{equation}
The proofs in finite dimension, e.g., \cite{arcones1994limit} and \cite{nourdin2011quantitative}, crucially exploit the fact that $\operatorname{dim}(\clh_1) < \infty$, which leads to condition \eqref{eq:ass-finite-dim}. For $\{e_r \}_{r \in \NN}$ an orthonormal basis in $\clh_1$, consider the marginal process $\{\langle X_k, e_r \rangle_{\clh_1}\}_{k \in \ZZ}$, $r \geq 1$. Attempting to extend the finite-dimensional theory directly, one might conjecture that the natural infinite-dimensional analogue of \eqref{eq:ass-finite-dim} is
\begin{equation} \label{eq:conj-cond}
\sum_{v \in \ZZ} \sup_{r,s \in \NN} | \langle \Gamma(v)e_r, e_s \rangle_{\clh_1} |^q < \infty.
\end{equation}
A perhaps surprising feature of the infinite-dimensional theory is that the sufficient assumptions emerging from our analysis take a rather different form. We obtain our Breuer-Major theorem under either of two alternative sets of assumptions: one formulated in terms of summability of the autocorrelation operators $R(v)$ (see \eqref{eq:R(v)-op} below for a rigorous definition) in operator norm, and another formulated in terms of summability of the autocovariance operators $\Gamma(v)$ in trace norm together with a regularity condition on $G$; see Assumptions \ref{thm:ass-2.2} and \ref{thm:ass-2.1}, respectively. In finite dimensions, these conditions collapse to equivalent requirements and are consistent with the classical Breuer-Major setting; see \eqref{eq:A1implesA2part2} and the corresponding discussion below. In infinite dimensions, however, they are genuinely different from the entrywise condition \eqref{eq:conj-cond}chdjd suggested by the finite-dimensional theorem. This distinction is one of the main ways in which the Hilbert space-valued Breuer-Major theorem departs from its finite-dimensional counterpart. Moreover, under the second route, trace-norm summability of $\Gamma(v)$ alone is not sufficient in general: an additional condition on $G$ is needed in order to ensure that the limiting covariance operator is trace-class; see the discussion following \eqref{eq:A1implesA2part2} below.

The third major difficulty stems from the necessity to control the trace-class norm of the difference of covariance operators, which are in general hard to compute. The need to do so arises from the conditions of the fourth-moment theorem in infinite dimensions developed in \cite{duker2025fourth}, which provides bounds for the correct $d_2$ distance between two probability laws in terms of such trace-class norms. To address this, our proofs employ operator decompositions that control the difference of operators in the trace-class norm even for operators that are not necessarily positive semidefinite or self-adjoint; see \eqref{al:operator-notation-separate-sum}. In particular, this part deviates from the classical proof strategy employed for the real and, more generally, finite-dimensional Breuer-Major theorem.

Besides developing a CLT for operators of Gaussian Hilbert space-valued random variables $\{S_n\}_{n \in \NN}$ in \eqref{eq:partial_sum_G}, we provide a continuous-time as well as a quantitative version of the Breuer-Major theorem for such random variables. To the best of our knowledge, our results are the first to generalize the notion of subordination to Hilbert space-valued random processes, allowing for a general class of transformations $G$. 

Our investigation is closely related to the modeling and analysis of functional data (e.g., data in $L^2([0,1])$, where CLTs can be used for downstream statistical tasks, such as the design of hypothesis tests, or the construction of confidence bands (for some expositions, see \cite{RamsaySilberman,Bosq2000:Linear,HsiEub15}). 
Standard models are Hilbert space-valued generalizations of autoregressive models
\cite{aue2015prediction,liu2016convolutional}.
Recently, the study of functional data has resurged, due to Machine Learning applications. Examples include data taking values in reproducing kernel Hilbert spaces \cite{MolKluSchKol22,wang2020functional}, as well as neural networks for learning maps between function spaces; see \cite{kovachki2023neural} who introduced the framework of neural operators for this task.

In Section \ref{sec:applications}, we illustrate the generality of our main result by applying it to statistics of functional data such as the sample covariance operator (also studied in \cite{MAS2002117,mas2006sufficient,duk20}) and estimators for eigenvalues of the covariance operator (e.g., \cite{Bosq2000:Linear}). Such results already exist for linear processes, but we extend them to a different dependence structure 
determined through the process's autocorrelation structure.
The neural network literature has been interested in proving CLTs for networks with randomly initialized weights
as the width of the layers diverges; see
\cite{favaro2023quantitative,kovachki2023neural,MolKluSchKol22}.
We provide a novel result along these lines for neural operators giving a quantitative CLT in terms of the network's increasing width.

The rest of the paper is organized as follows. Section \ref{subsec:prelim-infinite} focuses on notation, terminology, and some standard facts. In Section \ref{se:mainresults}, we present our main results. Subsequently, we introduce a chaotic decomposition in Section \ref{sec:Wiener-Ito-chaos}.
The proofs of our main results are presented in Section \ref{se:proofs}.
A quantitative version of our main result and its proof are given in Section \ref{app:quant_version}. Section \ref{sec:applications} states a variety of applications of our results. Section \ref{app:aux} presents auxiliary results needed for our proofs.

\section{Preliminaries} \label{subsec:prelim-infinite}

\textit{Notation and terminology:} Let $(\clx,\clf,\mu)$ be a measure space and $\mathcal{H}$ a separable Hilbert space equipped with inner product $\langle \cdot, \cdot  \rangle_{\mathcal{H}}$ and norm $\| \cdot \|_{\mathcal{H}}$. We denote by $L^2(\clx,\clf,\mu;\clh)$ the space of strongly measurable, square-integrable functions from $\clx$ to $\clh$ with respect to $\mu$. Whenever it is clear from the context, we omit the measure $\mu$ and write $L^2(\clx:\mathcal{H})$. Moreover, if $\clh = \RR$, we may write $L^2(\clx)$.

We recall some probabilistic results on general Hilbert spaces. 
Let $\{F_{n}\}_{n\in \NN}$ be a sequence of $\clh$-valued random variables. Then, $F_n$ converges in distribution to an $\clh$-valued random variable $F$ if 
\begin{equation}
    \E(h(F_n)) \to \E(h(F)),
\end{equation}
as $n \to \infty$, for all bounded, real valued and continuous functions $h$ on $\clh$. We say that an $\clh$-valued random variable $F$ admits second moments if $\E \| F \|^2_{\clh} < \infty$.
For further details, we refer to Section 2 in \cite{Bosq2000:Linear}.

\textit{Operators and norms on Hilbert spaces:}
For two Banach spaces $(V,\|\cdot \|_V)$ and $(W,\|\cdot \|_W)$, we denote as customary by $\cll(V,W)$ the space of all bounded linear operators $T : V \to W$. The space $\cll(V,W)$ is equipped with the norm
\begin{equation} \label{def:op-norm}
\|T\|_{\op} \doteq \sup_{\|x\|_V \le 1} \|Tx\|_W,
\end{equation}
so that $(\cll(V,W), \| \cdot \|_{\op})$ is a Banach space. When $V=W$, we write simply $\cll(V)$ for the space of bounded linear operators on $V$. Moreover, we write $\cll(V^{\otimes k},W)$ for the space of bounded linear operators from $V^{\otimes k}$ to $W$. When needed, we emphasize the spaces by writing $\|T\|_{\op(V,W)}$.

Let $\clh_1$ and $\clh_2$ be separable Hilbert spaces, and let $\{e_i\}_{i \in \NN}$ be an orthonormal basis of $\clh_1$. We denote by $\mathcal{S}_2(\clh_1,\clh_2)$ the space of Hilbert-Schmidt operators from $\clh_1$ to $\clh_2$, namely
\[
\mathcal{S}_2(\clh_1,\clh_2)
\doteq
\left\{
T \in \cll(\clh_1,\clh_2)
:
\sum_{i=1}^{\infty} \|T e_i\|_{\clh_2}^2 < \infty
\right\}.
\]
This space is a Hilbert space when equipped with the inner product
\begin{equation} \label{def:HS_norm}
     \langle T, S \rangle_{\mathcal{S}_2(\clh_1,\clh_2)}
     \doteq
     \sum_{i=1}^{\infty}
     \langle T(e_i), S(e_i) \rangle_{\clh_2},
     \hspace{0.2cm}
     \|T\|_{\mathcal{S}_2(\clh_1,\clh_2)}^2
     \doteq
     \sum_{i=1}^{\infty} \|T(e_i)\|_{\clh_2}^2 .
\end{equation}
The definition is independent of the chosen orthonormal basis of $\clh_1$. In the special case $\clh_1=\clh_2=\clh$, we also use the Schatten notation $\mathcal{S}_2(\clh) \doteq \mathcal{S}_2(\clh,\clh)$.

Closely related is the Banach space of trace-class operators, denoted by $\mathcal{S}_1(\clh)$ and equipped with the norm 
\begin{equation} \label{def:trace_norm}
     \| T \|_{\mathcal{S}_1(\clh)}
     =
     \tr(|T|)
     \doteq
     \sum_{i=1}^{\infty}
     \langle |T| e_i, e_i \rangle_{\clh},
\end{equation}
where $|T| = \sqrt{T^{*}T}$. If $T$ is a non-negative, self-adjoint operator, for instance a covariance operator, then
\[
\| T \|_{\mathcal{S}_1(\clh)}
=
\tr(T)
=
\sum_{i=1}^{\infty}
\langle T e_i, e_i \rangle_{\clh}.
\]
The three norms \eqref{def:op-norm}, \eqref{def:HS_norm}, and \eqref{def:trace_norm} satisfy
\begin{equation} \label{eq:bounded_HS_trace_norm}
   \|\cdot\|_{\op(\clh)}
   \le
   \| \cdot \|_{\mathcal{S}_2(\clh)}
   \leq
   \| \cdot \|_{\mathcal{S}_1(\clh)} .
\end{equation}

Crucial to our analysis are the isomorphisms $L^2(\Om: \clh) \cong L^2(\Om:\RR) \otimes \clh$ and $\mathcal{S}_2(\clh) \cong \clh \otimes \clh$.
Furthermore, for $\clh_i, i = 1,2$, Hilbert spaces, we can define an inner product on $\clh_1 \otimes \clh_2$, such that
\begin{equation} \label{eq:tensor_inner_product}
    \langle x_1 \otimes y_1, x_2 \otimes y_2  \rangle_{\clh_1 \otimes \clh_2} = \langle x_1, x_2 \rangle_{\clh_1} \langle y_1, y_2 \rangle_{\clh_2}, \quad x_i \in \clh_1, \  y_i \in \clh_2,  \ i = 1,2.
\end{equation}

\textit{Gaussian measures on Hilbert spaces:}
We provide here some elementary facts about Gaussian measures on Hilbert spaces; for more details we refer to \cite{DaPrato2014IntroductionStochasticAnalysis,BigFerForZan24}. Let \(\clh\) be a real separable Hilbert space, let \(\clb(\clh)\) be its Borel \(\sigma\)-field, and let \(Q \in \cll(\clh)\) be a self-adjoint, positive, trace-class operator. A Gaussian measure \(\gamma_{a,Q}\) on \((\clh,\clb(\clh))\), with mean \(a \in \clh\) and covariance operator \(Q\), is the Borel probability measure whose characteristic functional is given by
\begin{equation*}
    \int_{\clh} e^{i \langle x,h\rangle_\clh} \gamma_{a,Q}(dx)
    =
    \exp\left\{
        i \langle a, h\rangle_\clh
        -
        \frac{1}{2}\langle Qh,h\rangle_\clh
    \right\},
    \qquad h \in \clh.
\end{equation*}
Throughout the paper we write \(\gamma_{Q} \doteq \gamma_{0_\clh,Q}\). We say that the Gaussian measure \(\gamma_Q\) (resp. centered Gaussian random variable $Z$ with covariance operator $Q$) is nondegenerate if \(\ker(Q)=\{0\}\), equivalently, if \(\langle Qh,h\rangle_\clh>0\) for every \(h\in \clh\setminus\{0\}\). 

Suppose that $X_1$ is a Gaussian random variable with law $\gamma_Q$, where $Q$ is a nondegenerate covariance operator. Let $\{e_j\}_{j \in \NN}$ be the \textit{eigenvectors} of $Q$ and $\{\lambda_j\}_{j \in \NN}$ its corresponding sequence of positive \textit{eigenvalues}, such that 
$Q e_j = \lambda_j e_j$. Without loss of generality, we can re-enumerate $\{e_j\}_{j \in \NN}$ such that $\lambda_1 \ge \lambda_2 \ge \lambda_3 \dots$.
Then, the covariance operator $Q$ satisfies
\begin{equation} \label{eq:trace_variance}
     \tr(Q) = \sum_{j=1}^{\infty} \lambda_j = \E \| X_1 \|^2_{\clh}, \quad \lambda_j = \E \langle X_1, e_j \rangle^2_{\clh}, 
\end{equation}
by Parseval’s identity.

The space
\begin{equation} \label{eqCameronRange}
Q^{1/2}(\clh)
\doteq 
\operatorname{Ran}(Q^{1/2})
=
\{ y \in \clh : y = Q^{1/2}x \text{ for some } x \in \clh\}
\end{equation}
is called the \textit{Cameron-Martin space}. It is equipped with the inner product
\[
\langle u,v\rangle_{Q^{1/2}(\clh)}
\doteq
\langle Q^{-1/2}u, Q^{-1/2}v\rangle_\clh,
\qquad u,v \in Q^{1/2}(\clh),
\]
where \(Q^{-1/2}\) denotes the inverse of \(Q^{1/2}\) on \(\operatorname{Ran}(Q^{1/2})\). Then, consider the map $W: Q^{1/2}(\clh) \to L^2(\clh,\gamma_Q)$ given by $W_u (x) = \langle x , Q^{-1/2} u \rangle_\clh$ for $u \in Q^{1/2}(\clh), \ x \in \clh$. The map \( W\) is an isometry from \(Q^{1/2}(\clh)\), endowed with the inner product of the ambient space $\clh$, into \(L^2(\clh,\gamma_Q)\), i.e.,
\begin{equation} \label{eq:white-noise-orthog}
    \int_{\clh}  W_{u_1}(x)  W_{u_2}(x) \gamma_Q(dx)
    = \langle Q Q^{-1/2} u_1, Q^{-1/2} u_2\rangle_{\clh} = 
    \langle u_1,u_2\rangle_{\clh},
    \qquad u_1,u_2 \in Q^{1/2}(\clh).
\end{equation}

If, in addition, $Q$ is nondegenerate, then \(\operatorname{Ran}(Q^{1/2})\) is dense in \(\clh\), and the map $W$ extends uniquely to an isometry $W:\clh \to L^2(\clh,\gamma_Q)$, still denoted by $W$. This map is often referred to as the \textit{white-noise mapping} and, for this extension, one has
\begin{equation} \label{eq:white-noise-def-0}
    \int_{\clh} W_{h_1}(x)W_{h_2}(x) \gamma_Q(dx)
=
\langle h_1,h_2\rangle_\clh,
\quad h_1,h_2 \in \clh.
\end{equation}
It turns out that the map $W_h(x)$ admits the following important representation (Exercise 2.17 in \cite{DaPrato2014IntroductionStochasticAnalysis})
\begin{equation} \label{eq:white-noise-def}
    W_h(x) = \sum_{j=1}^\infty \lambda_j^{-1/2} \langle x , e_j \rangle_\clh \langle h , e_j \rangle_\clh, \quad h \in \clh,
\end{equation}
where the series converges in $L^2(\clh,\gamma_{Q})$. 
Throughout this work, we use both the white-noise notation \(W_h(x)\) as well as the explicit representations of $W_h(x)$ given in \eqref{eq:white-noise-orthog}--\eqref{eq:white-noise-def} according to convenience.

For the proofs below, let \( \{e_r\}_{r \in \NN}\) be an orthonormal basis of \(\clh_1\). We introduce the induced autocorrelation function
\begin{equation} \label{eq:cor_scores}
    \rho_{rs}(v)
    \doteq
    \E\big[ W_{e_r}(X_1) W_{e_s}(X_{1+v}) \big],
    \qquad r,s \in \NN,\ v \in \ZZ,
\end{equation}
giving rise to the autocorrelation operator defined as
\begin{equation} \label{eq:R(v)-op}
    R(v):\clh_1 \to \clh_1, 
    \hspace{0.2cm} 
\langle R(v)h,g\rangle_{\mathcal H_1}
\doteq
\E\!\left[W_h(X_1)W_g(X_{1+v})\right],
\qquad g,h\in\mathcal H_1,\ v\in\mathbb Z.
\end{equation}
Throughout the paper, we need the $p$-th tensorization of some of the objects introduced here. In particular, if  \(Qe_j=\lambda_j e_j\), 
\begin{equation} \label{eq:def-Qp}
    Q_p^{-1/2}
    \doteq
    (Q^{-1/2})^{\otimes p}, \quad \text{implying that } Q_p^{-1/2}
    (e_{r_1}\otimes\cdots\otimes e_{r_p})
    =
    (\lambda_{r_1}\cdots\lambda_{r_p})^{-1/2}
    e_{r_1}\otimes\cdots\otimes e_{r_p}.
\end{equation}

Moreover, whenever Assumption \ref{thm:ass-2.1} holds below, the operator $Q_p^{-1/2} C_p$ (with $C_p$ defined in \eqref{eq:defCp}) will be bounded due to \eqref{thm:ass-2-2:different}. Then, we can decompose  
    \begin{equation} \label{eq:Cp-Rv-Gammav}
        C_p^* R(v)^{\otimes p} C_p = (Q_p^{-1/2} C_p)^* \Gamma(v)^{\otimes p} Q_p^{-1/2} C_p.
    \end{equation}

\begin{remark} \label{re:remark-on-operator-identity-R-Gamma}
    One would like to write the equation
    \begin{equation} \label{eq:Rv-Gammav-formal}
        R(v)^{\otimes p}
    =
    Q_p^{-1/2}\Gamma(v)^{\otimes p}Q_p^{-1/2}.
\end{equation}
    However, \eqref{eq:Rv-Gammav-formal} is only formal since the left hand side has been extended to an operator on $\clh_1^{\otimes p}$ through \eqref{eq:R(v)-op}, while the right hand side can only be defined on a dense subset of $\clh_1^{\otimes p}$, i.e., the Cameron Martin space $\left( Q^{1/2}(\clh_1) \right)^{\otimes p}$.
\end{remark}

\textit{Hermite expansions:} 
 As a technical tool necessary for our paper, we introduce an orthogonal basis of the space $L^2(\clh_1,\gamma_Q:\clh_2)$. Simpler forms of this result have appeared elsewhere, e.g., for $L^2(\clh_1,\gamma_Q:\RR)$ in \cite{DaPr06}. For completeness, we start with rephrasing some basic properties of Hermite polynomials in the real valued case. The 
 Hermite polynomials $\{ H_{n} \}_{n \in \NN_{0}}$ build an orthogonal basis of $L^2(\RR,\phi(x) dx: \RR)$, where $\phi$ is the standard normal density; see Proposition 5.1.3 in \cite{pipiras2017long}.
 As a result, every function $f \in L^2(\RR,\phi(x) dx:\RR)$ has an expansion in Hermite polynomials.
 The Hermite polynomial of order $n$ is defined as $H_n(x) = (-1)^n e^{x^2/2} \frac{d^n}{dx^n} e^{-x^2/2},$ for $x \in \RR,  n \geq 1$ and $H_0(x) =1$ for all $x \in \RR$. Moreover, recall their crucial orthogonality property, for $(X,Y)$ jointly Gaussian with standard marginals,
 \begin{equation} \label{eq:hermite-orthogonality}
     \E \left( H_n(X) H_m(Y) \right) = n! \left( \E XY \right)^n \delta_{nm};
\end{equation}
see Proposition 5.1.1 in \cite{pipiras2017long}.

We continue with some more notation. Let $\cll = \ell^1(\NN_0) = \{\boldsymbol{l} \in \NN_0^\infty: \sum_{k=1}^\infty l_k < \infty \}$ denote the space of summable sequences, with each coordinate taking values in $\NN_0$. Note that, for $\bfl \in \cll$, there are only finitely many non-zero elements of $\bfl$. For fixed $\bfl \in \cll$, we define
\begin{equation} \label{eq:Hermite_product}
 H_\bfl[x] \doteq \prod_{m=1}^\infty H_{l_m} ( W_{e_m}(x)) 
 = \prod_{m=1}^{M_\bfl} H_{l_m} ( W_{e_m}(x)), \quad x \in \clh_1,
 \end{equation}
 where $W_{e_m}$ is the white noise mapping defined in \eqref{eq:white-noise-def}, $\{e_m\}_{m \in \NN}$ is a basis for $\clh_1$, and 
 \begin{equation} \label{eq:defn-Ml}
     M_0 \doteq 0, \quad M_\bfl \doteq \max\{ m \in \NN : l_m \ge 1\}
 \end{equation}
 is the order of the highest non-zero element of the sequence $\bfl$. For $m \in \NN_0$, we also write $H_m [ W_u]$ to denote a map from $\clh_1$ to $\RR$ evaluated by $H_{m}[W_u](x) \doteq H_{m}(W_u(x))$. 

 The following result is crucial for our paper. Since we have not identified it somewhere else in the literature, we prove it in Section \ref{sec:Wiener-Ito-chaos}.

\begin{lemma} \label{thm:hermite-expansion-general}
Consider the operator $f \in L^2(\clh_1,\gamma_Q :\clh_2) \cong L^2(\clh_1,\gamma_Q) \otimes \clh_2$, where $\gamma_Q$ is a Gaussian measure with covariance operator $Q$. Then, $f$ admits the generalized Hermite expansion
\begin{equation} \label{eq:Hermite_expansion}
\begin{split}
    f[x] = \sum_{i=1}^\infty \sum_{ \bfr \in \cll} c_{i,\bfr}  H_{\bfr}[x] \otimes v_i, \quad x \in \clh_1,
\end{split}
\end{equation}
with $H_\bfr$ as in \eqref{eq:Hermite_product} and $\{c_{i,\bfr}\}_{i \in \NN,\bfr \in \cll}$ is given by (with some abuse of notation) 
\begin{equation} \label{eq:coeff-hermite-general}
c_{i,\bfr} \doteq \frac{1}{\prod_{m=1}^\infty r_m!} \left\langle f, H_\bfr \otimes v_i \right\rangle_{L^{2}(\clh_1, \gamma_Q) \otimes \clh_{2}}.
\end{equation}
\end{lemma}

\textit{Weak convergence in infinite dimensions:} In this section, we present a metric on $\clp(\clh)$, the space of probability measures on $\clh$, that metrizes weak convergence. This metric was introduced by Gin\'{e} and Le\'{o}n in \cite{GinLeo80}. The recent write-up \cite{BasBurCamPec25} further clarifies the distinction between this distance and the metric $\rho_2$ used in \cite{Bou20}.

We denote by $\clc_b^k(\clh)$ the space of bounded, $\RR$-valued functions on $\clh$ admitting $k$ Fréchet derivatives, that is, $h \in \clc_b^k(\clh)$ if
\begin{equation} \label{eq:norm-def-cb2}
\| h \|_{\clc_b^k(\clh)}
\doteq
\sup_{x \in \clh} \left( \sum_{j=0}^k \| \operatorname{D}_F^j h(x) \|_{\op(\clh^{\otimes j},\RR)} \right)
< \infty,
\end{equation}
under the convention that $\| \operatorname{D}_F^0 h(x) \|_{\cll(\clh^0,\RR)} = |h(x)|$.

Then, for $\mu,\nu \in \clp(\clh)$ and $j \ge 1$, the $d_j$ metric on $\clp(\clh)$ is defined by
\begin{equation} \label{eq:d2metric}
d_j(\mu,\nu)
\doteq
\sup_{h \in \clc_b^j(\clh),\ \|h\|_{\clc_b^j(\clh)} \le 1}
\left|
\int_{\clh} h(x) (\mu(dx)-\nu(dx))
\right|.
\end{equation}
From \cite{GinLeo80} (Theorem 2.4), it is known that for every $j \ge 1$, the metric $d_j$ (and in particular $d_2$) metrizes weak convergence on $\clp(\clh)$. For $\clh$-valued random variables $X$ and $Z$, we shall write
\[
d_2(X,Z) \doteq d_2(\cll(X),\cll(Z)),
\]
where $\cll(\cdot)$ denotes the law of a random variable. Equivalently,
\[
d_2(X,Z)
=
\sup_{h \in \clc_b^2(\clh),\ \|h\|_{\clc_b^2(\clh)} \le 1}
\left|
\EE[h(X)]-\EE[h(Z)]
\right|.
\]

\textit{Isonormal Gaussian processes and contractions:}
Denote by $\mathfrak{H}$ an underlying separable Hilbert space, and define an isonormal Gaussian process, i.e., a centered family of Gaussian random variables $\left\{ W(h)
\colon h \in \mathfrak{H} \right\}$, defined on a complete probability space $(\Omega,\mathcal{F},P)$, such that
\begin{equation*}
\E{W(h_1)W(h_2)} = \left\langle h_1, h_2 \right\rangle_{\mathfrak{H}}, \qquad
h_1, h_2 \in \mathfrak{H},
\end{equation*}
where the $\sigma$-algebra $\mathcal{F}$ is generated by $W$. We show in Lemma \ref{pro:isonormal} in Section \ref{app:aux} that it is possible to construct an isonormal Gaussian process on $\mathfrak{H}$ with the same autocorrelation function as an $\clh$-valued Gaussian process. Denote by $\mfh^{\otimes n}$ the $n$-th tensor power of $\mfh$, and by $\mathfrak{H}^{\odot n}$ the $n$-fold symmetrized
tensor product of $\mathfrak{H}$, equipped with the norm $\sqrt{n !}\| \cdot \|_{\mfh^{\otimes n}}$. For $n \ge 0$ and a kernel $f \in \mathfrak{H}^{\odot n}$, we write $I_n(f)$ to denote the multiple Wiener-It\^{o} integral of order $n$ of $f$; see \cite{Nua06book}.

Let $\left\{ \phi_k \colon k \geq 1 \right\}$ be an
orthonormal basis of $\mathfrak{H}$. Fix $f \in \mathfrak{H}^{\odot n}$ and $g
\in \mathfrak{H}^{\odot m}$, then for every $l = 0,\ldots, n \wedge m$, the
$l$-th contraction of $f$ and $g$ is defined as
\begin{equation*}
f \otimes_l g = \sum_{i_1, \ldots, i_l =1}^{\infty} \left\langle
  f,\phi_{i_1}\otimes \cdots \otimes \phi_{i_l} \right\rangle_{\mathfrak{H}^{\otimes l}}
\otimes \left\langle
  g,\phi_{i_1}\otimes \cdots \otimes \phi_{i_l} \right\rangle_{\mathfrak{H}^{\otimes l}} \in \mathfrak{H}^{\otimes
  (n+m-2l)}.
\end{equation*}
See also Appendix B in \cite{nourdin2012normal} for more details on contractions.

Since several Hilbert spaces appear in the present work, to aid with presentation we fix the notations for their respective bases. These are summarized in the table below.
\begin{table}[h!]
\centering
\begin{tabular}{ll}
\toprule
Space & Orthonormal basis \\
\midrule
$\mathcal{H}_1$ (with $X_k$ random variables in $\clh_1$) & $\{e_k\}_{k\in\mathbb{N}}$ \\
$\mathcal{H}_2$ (with $G: \clh_1 \to \clh_2$) & $\{v_k\}_{k\in\mathbb{N}}$ \\
$Q^{1/2}(\mathcal{H}_1)$ (Cameron-Martin space associated to $\clh_1, \gamma_Q$) & $\{u_k\}_{k\in\mathbb{N}}$ \\
$\mathfrak{H}$ (underlying Hilbert space for the isonormal Gaussian process)  & $\{\phi_k\}_{k\in\mathbb{N}}$ \\
\bottomrule
\end{tabular}
\end{table}

\section{Main results} \label{se:mainresults}
We state in Section \ref{subsec:statement-main} our main results, providing a central limit theorem for the quantity in \eqref{eq:partial_sum_G} and a continuous-time version of the same result. The statements are followed by a discussion on our assumptions in Section \ref{subse:examples-hilbert}. Sections \ref{subseexpautocorop} and \ref{subseexpautocovop} provide examples for our assumptions.

\subsection{Statements} \label{subsec:statement-main}

Recall that
\begin{equation} \label{cov_op_Q}
    Q:\clh_1 \to \clh_1, 
    \hspace{0.2cm} 
    Q[\cdot] = \E \left( \langle X_1, \cdot \rangle_{\clh_1} X_1 \right)
\end{equation}
is the nondegenerate covariance operator of $X_1$. Central to our statements is the condition $\E \| G[X_1] \|^2_{\clh_2} < \infty$ which can be recast as $G \in L^2(\clh_1, \gamma_Q:\clh_2)$, where $\gamma_Q$ is the (unique) Gaussian measure on $\clh_1$ associated with the covariance operator $Q$ and mean-zero. 
We state our theorems for the important case $\dim(\clh_1) = \infty$, but the results hold even with $\dim(\clh_1) < \infty$. When $\operatorname{dim}(\clh_1) < \infty$, the results here recover many existing theorems, e.g., Theorem 7.2.4 of \cite{nourdin2012normal} and Theorem 1.1 of \cite{nourdin2011quantitative} that goes back to \cite{breuer1983central}.


Prior to stating our main theorem, we introduce two sets of assumptions. Either of those assumption sets leads to a Breuer-Major type theorem on infinite-dimensional Hilbert spaces. The assumptions are stated in terms of the autocovariance operator $\Gamma(v)$ in \eqref{eq:Gamma-op}, the autocorrelation operator $R(v)$ in \eqref{eq:R(v)-op} and the Hermite rank of $G$, formally defined in Definition \ref{def:Hermiterank}. 
The assumptions are followed by a discussion.

\begin{assumption}\label{thm:ass-2.2}
Suppose $G \in L^2(\clh_1,\gamma_{Q}:\clh_2)$ with Hermite rank $q \ge 1$ and that 
    \begin{equation} \label{thm:ass-2.0-eq-Gamma}
    \sum_{v \in \ZZ} \| R(v) \|_{\op(\clh_1)}^{q} < \infty,
    \end{equation}
where $R(v)$ is defined in \eqref{eq:R(v)-op}.
\end{assumption}

We introduce the second sufficient assumption. For \(p\ge q\) and \(\mathbf j=(j_1,\ldots,j_p)\in \mathbb N^p\), define
\[
\ell_m(\mathbf j) \doteq \#\{r\in\{1,\ldots,p\}:j_r=m\},\qquad m\ge 1,
\]
and let \(\ell(\mathbf j):=(\ell_m(\mathbf j))_{m\ge1}\). Moreover, for each \(p\ge q\), define the linear map \(C_p\) by, acting on the basis $\{v_i\}$, 
\begin{equation}\label{eq:defCp} 
C_p v_i \doteq \sum_{\mathbf j=(j_1,\ldots,j_p)\in\mathbb N^p} \frac{\prod_{m\ge1}\ell_m(\mathbf j)!}{p!}\, c_{i,\ell(\mathbf j)} \,e_{j_1}\otimes\cdots\otimes e_{j_p},
\end{equation}
where the coefficients $c_{i,\ell(\mathbf j)}$ are the Hermite coefficients induced by Lemma \ref{thm:hermite-expansion-general}.
\begin{assumption} \label{thm:ass-2.1}
Suppose \(G\in L^2(\clh_1,\gamma_Q:\clh_2)\) has Hermite rank \(q\ge1\).
We assume that, for every \(p\ge q\), the map \(C_p\) defined in \eqref{eq:defCp} extends to a bounded linear operator $C_p\in \mathcal{L}(\mathcal H_2,\mathcal H_1^{\otimes p})$ such that $C_p(\mathcal H_2)\subseteq \operatorname{Ran}(Q_p^{1/2})$, and that
\begin{align}
    \sum_{v \in \ZZ} \| \Gamma(v) \|_{\mathcal{S}_1(\clh_1)}^{q} < \infty, \label{thm:ass-2.1-eq-Gamma} \\
    \sum_{p=q}^{\infty}
p!\,
\bigl\|Q_p^{-1/2}C_p\bigr\|_{\operatorname{op}(\mathcal{H}_2,\mathcal{H}_1^{\otimes p})}^2
<\infty, \label{thm:ass-2-2:different}
\end{align}
where $\Gamma(v)$ is defined in \eqref{eq:Gamma-op} and $Q_p^{-1/2}$ in \eqref{eq:def-Qp}.
\end{assumption}

\begin{theorem}[Breuer-Major for Hilbert space-valued random variables] \label{thm:main}
Let $\{X_{k}\}_{k\in \ZZ}$ be a zero-mean, stationary Gaussian process with values in $\clh_1$ and covariance operator $Q$. Suppose either Assumption \ref{thm:ass-2.2} or \ref{thm:ass-2.1}. Denote by $Z$ the centered Gaussian random variable on $\clh_2$ with covariance operator 
\begin{equation} \label{eq:def-T}
    \covop_{Z} \doteq
    \sum_{v=1}^{\infty} \bigg( \E G[X_1] \otimes G[X_{v+1}] + 
    \E G[X_{v+1}] \otimes G[X_1] \bigg) + 
    \E G[X_1] \otimes G[X_1],
\end{equation}
\MD{where the series converges in trace-class norm and defines a nonnegative, self-adjoint covariance operator. Then, $S_n \xrightarrow{d} Z$, as $n \to \infty$.}
\end{theorem}

In Theorem \ref{th:quantitave} in Section \ref{app:quant_version}, we also give a quantitative version of Theorem \ref{thm:main} and its proof. 

Theorem \ref{thm:main} and the used proof techniques allow us also to generalize the result to a continuous-time version. Define
\begin{equation} \label{eq:Vn-defn-cont}
    V_n (t) \doteq \frac{1}{\sqrt{n}} \sum_{k=1}^{\lfloor nt \rfloor} G[X_k], \quad t \in [0,1],
\end{equation}
where $\{ X_k \}_{k \in \ZZ}$, $X_k \in \clh_1$, and $G : \clh_1 \to \clh_2$. Here, $V_n$ is an element of $L^2([0,1]:\clh_2) \cong L^2([0,1]) \otimes \clh_2$, which is again a Hilbert space. 
Recall that the covariance operator $\covop_B$ of a Brownian motion $B$ in $L^2([0,1])$ is given by the integral operator
\begin{equation} \label{eq:BrownianMotion_int}
    \covop_B: L^2([0,1]) \to L^2([0,1]), \hspace{0.2cm}
    \covop_B[f] = \int_0^1 f(t)  \kappa(s,t) dt
    \hspace{0.2cm} \text{ with } \hspace{0.2cm}
    \kappa(s,t) \doteq s \wedge t.
\end{equation}
For tensorization properties of Gaussian measures, we refer the reader to \cite{CarChe79}.

\begin{theorem} \label{th:cont_case}
    Let $\{V_n(t)\}_{t \in [0,1]}$ be as in \eqref{eq:Vn-defn-cont} with $\{X_{k}\}_{k\in \ZZ}$ a zero-mean, stationary Gaussian process with values in $\clh_1$ and covariance operator $Q$. Suppose either Assumption \ref{thm:ass-2.2} or \ref{thm:ass-2.1}. 
    \MD{Let $\covop_Z$ be the covariance operator in \eqref{eq:def-T}. Then,} 
    \begin{equation*}
    V_n \xrightarrow{d} W,  \hspace{0.2cm} \text{ as } n \to \infty,
    \end{equation*}
    in $L^2([0,1]) \otimes \clh_2$, where $W$ is the centered Gaussian element in $L^2([0,1]) \otimes \clh_2$ with covariance operator $\covop_W: L^2([0,1]) \otimes \clh_2 \to L^2([0,1]) \otimes \clh_2$ given by
    \begin{equation*} \label{eq:limiting_cov_op_cont_case}
    \covop_{W} = \covop_{B} \otimes \covop_{Z},
    \end{equation*}
    where $B = \{B_t ~|~ t \in [0,1]\}$ denotes a standard Brownian motion in $L^2([0,1])$ with covariance operator $\covop_{B}$.
\end{theorem}

\begin{remark}
    The limit $W$ in Theorem \ref{th:cont_case} may equivalently be
    regarded as an $\clh_2$-valued Brownian motion with covariance operator
    $\clt_Z$. As a centered Gaussian element of $L^2([0,1]:\clh_2)$, its
    covariance kernel is given by
    \begin{equation*}
        \E \big[W(s) \otimes W(t)\big] = (s \wedge t)\clt_Z,
        \qquad s,t\in[0,1].
    \end{equation*}
\end{remark}

\subsection{Discussion on assumptions} \label{subse:examples-hilbert}

The two assumptions in Theorem \ref{thm:main} provide different sufficient conditions for a Breuer-Major theorem in infinite dimensions. 
Assumption \ref{thm:ass-2.2} is formulated in terms of the autocorrelation operators $R(v)$ in \eqref{eq:R(v)-op}, whereas Assumption \ref{thm:ass-2.1} is formulated in terms of the autocovariance operators $\Gamma(v)$ together with an additional regularity condition on the interplay between $G$ and the covariance operator $Q$ of the underlying Gaussian stationary sequence (cf. \eqref{thm:ass-2-2:different}). 

This distinction is specific to the infinite-dimensional setting. Indeed, when $\dim(\clh_1)<\infty$, then the matrix Schatten norms are all equivalent. Moreover, when $Q$ is positive definite, then Condition \eqref{thm:ass-2.1-eq-Gamma} is equivalent to Condition \eqref{thm:ass-2.0-eq-Gamma}. In particular, conjectured conditions \eqref{eq:conj-cond} and \eqref{thm:ass-2.1-eq-Gamma} are equivalent, and so whenever Assumption \ref{thm:ass-2.1} holds, so does Assumption \ref{thm:ass-2.2}, allowing for a more general class of operators $G$.

In infinite dimensions, however, the two Conditions \eqref{thm:ass-2.0-eq-Gamma} and \eqref{thm:ass-2.1-eq-Gamma} are no longer equivalent. In this case, we only have that
\begin{equation} \label{eq:A1implesA2part2}
    \| \Gamma(v) \|_{\mathcal{S}_1(\clh_1)}
    =
    \| Q^{\frac{1}{2}} R(v) Q^{\frac{1}{2}} \|_{\mathcal{S}_1(\clh_1)}
    \leq
    \| Q^{\frac{1}{2}} \|_{\mathcal{S}_2(\clh_1)}^2 
    \| R(v) \|_{\op(\clh_1)}
    =
    \tr(Q) 
    \| R(v) \|_{\op(\clh_1)} .
\end{equation}
Moreover, under Assumption \ref{thm:ass-2.1}, the summability condition \eqref{thm:ass-2.1-eq-Gamma} alone is not sufficient in general: one also needs the condition \eqref{thm:ass-2-2:different} on $G$ in order to ensure that the limiting covariance operator in \eqref{eq:def-T} is trace-class.

The following example shows that \MD{\eqref{thm:ass-2.1-eq-Gamma} and square integrability of $G$ alone do not suffice}. More precisely, it gives a process $\{X_k\}_{k\in\ZZ}$ satisfying
\[
    \sum_{v\in\ZZ} \|\Gamma(v)\|_{\mathcal S_1(\clh_1)} < \infty
\]
and a map $G\in L^2(\clh_1,\gamma_Q:\clh_2)$ for which the conclusion of Theorem \ref{thm:main} cannot hold for the normalized sums of $\{G[X_k]\}_{k\in\ZZ}$. This happens because, in the absence of condition \eqref{thm:ass-2-2:different}, the candidate limiting covariance operator in \eqref{eq:def-T} need not be in $\mathcal S_1(\clh_2)$.

\begin{example} Take $\clh_1 = \clh_2 = \ell^2(\NN)$ with the canonical orthonormal bases $\{e_r\}_{r\ge1}$ in $\clh_1$ and $\{v_i\}_{i\ge1}$ in $\clh_2$. Let, for $r \ge 1$
\[
Qe_r \doteq \lambda_r e_r,
\qquad
Ae_r \doteq \alpha_r e_r, \qquad \text{where }\lambda_r \doteq 2^{-r},
\qquad
\alpha_r \doteq 1-\frac{1}{(r+1)^2},
\]
Then $Q$ is self-adjoint, positive, and trace-class, while $A$ is bounded, self-adjoint, and $\|A\|_{\op}=\sup_{r\ge1}\alpha_r=1$. Define
\[
\Sigma_\varepsilon \doteq Q-AQA, \quad \text{with }\Sigma_\varepsilon e_r
=
\lambda_r(1-\alpha_r^2)e_r.
\]
Hence $\Sigma_\varepsilon$ is also self-adjoint, positive, and trace-class. Let
$\{\varepsilon_k\}_{k\in\ZZ}$ be an i.i.d.\ centered Gaussian sequence in $\clh_1$ with covariance
operator $\Sigma_\varepsilon$, and define
\[
X_k\doteq \sum_{m=0}^\infty A^m\varepsilon_{k-m},
\qquad k\in\ZZ.
\]
This series converges in $L^2(\Omega:\clh_1)$, because
\begin{equation*}
    \sum_{m=0}^\infty \EE\|A^m\varepsilon_0\|_{\clh_1}^2
=
\sum_{m=0}^\infty \tr\!\bigl(A^m\Sigma_\varepsilon A^m\bigr) =
\sum_{r=1}^\infty \lambda_r(1-\alpha_r^2)\sum_{m=0}^\infty \alpha_r^{2m} =
\sum_{r=1}^\infty \lambda_r
<\infty.
\end{equation*}
Therefore $\{X_k\}_{k\in\ZZ}$ is a centered stationary Gaussian sequence in $\clh_1$, and it satisfies
\[
X_k=AX_{k-1}+\varepsilon_k
\qquad\text{in }L^2(\Omega:\clh_1),
\]
with covariance operator
\[
\Gamma(v) = \EE[X_v\otimes X_0]=QA^{|v|},
\qquad v\in\ZZ,
\]
In particular, $\Gamma$ is summable in trace norm. Since $\Gamma(v)$ is positive and diagonal,
\[
\sum_{v \in \ZZ} \|\Gamma(v)\|_{\mathcal S_1 (\ell^2(\NN))}
= \sum_{v \in \ZZ}
\tr(\Gamma(v))
=
\sum_{v \in \ZZ} \sum_{r=1}^\infty 2^{-r}\alpha_r^{|v|} 
= \sum_{r=1}^\infty 2^{-r}\frac{1+\alpha_r}{1-\alpha_r}  = \sum_{r=1}^\infty 2^{-r}\bigl(2(r+1)^2-1\bigr)<\infty,
\]
where we used that $\frac{1+\alpha_r}{1-\alpha_r}=2(r+1)^2-1$. On the other hand, one can also see that $\|R(v)\|_{\op(\clh_1)}
=
\sup_{r\ge1}\alpha_r^{|v|}
=
1$, implying that $\sum_{v\in\ZZ}\|R(v)\|_{\op(\clh_1)}=\infty$

We show that Theorem \ref{thm:main} cannot hold, by illustrating that the candidate limiting operator cannot be a covariance operator, i.e., trace-class.
Let $\gamma_Q$ denote the centered Gaussian measure on $\clh_1$ with covariance operator $Q$.
We define
\[
G(x) \doteq \sum_{i=1}^\infty \frac{1}{i\sqrt{\lambda_i}} 
\langle x,e_i\rangle_{\ell^2(\NN)} v_i,
\]
for $\gamma_Q$-almost every $x\in\clh_1$. This defines an element of
$L^2(\clh_1,\gamma_Q:\clh_2)$, because
\[
\EE\|G(x)\|_{\ell^2(\NN)}^2
=
\sum_{i=1}^\infty \frac{1}{i^2\lambda_i} 
\EE|\langle x,e_i\rangle_{\ell^2(\NN)}|^2 =
\sum_{i=1}^\infty \frac{1}{i^2\lambda_i} \lambda_i =
\sum_{i=1}^\infty \frac{1}{i^2}
<\infty.
\]
Since each $X_k$ has law $\gamma_Q$, the random variables
\[
Y_k\doteq G(X_k)
=
\sum_{i=1}^\infty \frac{1}{i\sqrt{\lambda_i}} 
\langle X_k,e_i\rangle_{\ell^2(\NN)} v_i,
\qquad k\in\ZZ,
\]
are well defined in $L^2(\Omega:\clh_2)$. The sequence $\{Y_k\}_{k\in\ZZ}$ is centered,
stationary, and Gaussian in $\clh_2$. Its covariance operators are diagonal since
\begin{equation}
    \EE[Y_0\otimes Y_v]
=
\sum_{i=1}^\infty \frac{1}{i^2\lambda_i}
\langle \Gamma(v)e_i,e_i\rangle_{\ell^2(\NN)} (v_i\otimes v_i) =
\sum_{i=1}^\infty \frac{1}{i^2}\alpha_i^{|v|} (v_i\otimes v_i).
\end{equation}

Formally, the covariance operator of the limit of the renormalized sums $\frac{1}{\sqrt n}\sum_{k=1}^n Y_k$ would be given by \eqref{eq:def-T}, i.e., 
\[
\clt
=
\sum_{i=1}^\infty \frac{1}{i^2}\sum_{v\in\ZZ}\alpha_i^{|v|} (v_i\otimes v_i),
\]
Since
\[
\frac{1}{i^2} \sum_{v\in\ZZ}\alpha_i^{|v|}
= \frac{1}{i^2}
\frac{1+\alpha_i}{1-\alpha_i}
= \frac{1}{i^2}
2(i+1)^2-1
=
\frac{1}{i^2}\bigl(2(i+1)^2-1\bigr)
\sim 2
\]
it follows that $\clt\notin \mathcal S_1(\clh_2)$, and so no
$\clh_2$-valued centered Gaussian random variable having covariance operator
$\clt$ exists.
\end{example}

\subsection{Examples on the autocorrelation operator}
\label{subseexpautocorop}

We present several examples of processes $\{X_k\}_{k \in \ZZ}$ and identify sufficient conditions such that they satisfy condition \eqref{thm:ass-2.2}.
As a first example, we show that our results trivially extend (in a quantitative way) the standard CLT for $\clh_1$-valued subordinated Gaussian processes. In particular, we recover Theorem 2.7 in \cite{Bosq2000:Linear}.

\begin{example}[i.i.d. case]
    \MD{Let $\{X_k\}_{k \in \ZZ}$ be a sequence of zero-mean, Gaussian, i.i.d. random variables with covariance operator $Q$, and let $G \in L^2(\clh_1,\gamma_Q:\clh_2)$ be centered.} Then, since $\Gamma(v)=0$ for all $v \neq 0$ and $\Gamma(0)=Q$, we have
    \begin{equation*}
        R(v)=Q^{-1/2}\Gamma(v)Q^{-1/2}=0, \quad v \neq 0,
        \qquad
        R(0)=Q^{-1/2}QQ^{-1/2}=\Id_{\clh_1},
    \end{equation*}
    and therefore
    \begin{equation*}
        \sum_{v \in \ZZ} \| R(v) \|_{\op(\clh_1)}^{q}
        =
        \|R(0)\|_{\op(\clh_1)}^{q}
        =
        \|\Id_{\clh_1}\|_{\op(\clh_1)}^{q}
        =
        1
        <
        \infty.
    \end{equation*}
    Hence, by \eqref{eq:A1implesA2part2}, both Assumption \ref{thm:ass-2.2} and condition \eqref{thm:ass-2.1-eq-Gamma} trivially hold for all $q \ge 1$. By Theorem \ref{thm:main}, it follows that, as $n \to \infty$,
    \begin{equation*}
        \frac{1}{\sqrt{n}} \sum_{k=1}^n G[X_k] \xrightarrow{d} \cln, \quad \clt_\cln \doteq \E G[X_1] \otimes G[X_1].
    \end{equation*}
\end{example}

Example \ref{ex:separable-spacetime} studies models whose spatiotemporal autocovariance structure decomposes into the product of two functions, one encoding the decay in time and one in space. Due to their simplicity in modeling and estimation problems, these models are popular in applications; see \cite{constantinou2017testing} and the references therein.

\begin{example}[Separable autocovariance]
\label{ex:separable-spacetime}
Assume that there exists a sequence \(r:\mathbb Z\to \mathbb R\), with
\(r(0)=1\), such that the lag-\(v\) covariance operator satisfies
\[
\Gamma(v)=r(v)Q,\qquad v\in\mathbb Z,
\]
where \(Q\) is the covariance operator of \(X_1\). Assume further that
\(Q\) is nondegenerate.

For \(h,g\in \mathcal{H}_1\), by the definition of
\(R(v)\) and \eqref{eq:white-noise-def},
\[
\begin{aligned}
\langle R(v)h,g\rangle_{\mathcal{H}_1}
&=
\mathbb E\big[W_h(X_1)W_g(X_{1+v})\big] \\
&=
\sum_{i,j = 1}^\infty
\lambda_i^{-1/2}\lambda_j^{-1/2}
\langle h,e_i\rangle_{\mathcal{H}_1}
\langle g,e_j\rangle_{\mathcal{H}_1}
\mathbb E\big[
\langle X_1,e_i\rangle_{\mathcal{H}_1}
\langle X_{1+v},e_j\rangle_{\mathcal{H}_1}
\big] \\
&= \sum_{i,j = 1}^\infty
\lambda_i^{-1/2}\lambda_j^{-1/2}
\langle h,e_i\rangle_{\mathcal{H}_1}
\langle g,e_j\rangle_{\mathcal{H}_1}  \langle \Gamma(v)e_i,e_j\rangle_{\mathcal{H}_1} \\
&= \sum_{i,j = 1}^\infty
\lambda_i^{-1/2}\lambda_j^{-1/2}
\langle h,e_i\rangle_{\mathcal{H}_1}
\langle g,e_j\rangle_{\mathcal{H}_1}
r(v)\lambda_i\delta_{ij} \\
&=
r(v)\sum_{i = 1}^\infty
\langle h,e_i\rangle_{\mathcal{H}_1}
\langle g,e_i\rangle_{\mathcal{H}_1} \\
&=
r(v)\langle h,g\rangle_{\mathcal{H}_1}.
\end{aligned}
\]
Since this identity holds for all \(h,g\in {\mathcal{H}_1}\), it follows that $R(v)=r(v)I_{\mathcal{H}_1}$ and so $\|R(v)\|_{\mathrm{op}({\mathcal{H}_1})}=|r(v)|$, hence Assumption~3.1 is satisfied whenever
\[
\sum_{v\in\mathbb Z}|r(v)|^q<\infty.
\]
\end{example}

We present a concrete application of Example \ref{ex:separable-spacetime}. Let $\clh_1=L^2(D)$, where $D\subset\R^d$ is a bounded domain, and let
$\{X_k\}_{k\in\ZZ}$ be a centered, stationary Gaussian process in $\clh_1$ with
separable space--time covariance, i.e., there exists a sequence $r$ as in Example \ref{ex:separable-spacetime} and
a symmetric positive definite kernel $C:D\times D\to\R$
such that for all $s,t\in D$ and $v\in\ZZ$,
\begin{equation}
\label{eq:sep-cov-kernel}
\Cov\big(X_0(s),X_v(t)\big) = r(v) C(s,t).
\end{equation}
Let $Q:L^2(D)\to L^2(D)$ be the integral operator induced by $C$,
\begin{equation}
(Qf)(t) \doteq \int_D C(s,t) f(s) ds,\qquad f\in L^2(D),
\end{equation}
and assume that $Q$ is positive, trace-class, and self-adjoint.
Then, the process $\{X_k\}$ can be modeled as in Example \ref{ex:separable-spacetime}.

The following example identifies an equivalent condition to \eqref{thm:ass-2.1-eq-Gamma} for \textit{m-dependent processes}. While $m$-dependence is a simple relaxation of independence, most time series models are not actually $m$-dependent, but can be approximated by an $m$-dependent sequence; see Chapter 16.1 in \cite{horvath2012inference} for some examples. In particular, $m$-dependence does not require the underlying model to admit a linear representation.

\begin{example}[\(m\)-dependent case]
\label{ex:m-dependent}
Let \(\{X_k\}_{k\in\mathbb Z}\) be an \(\mathcal{H}_1\)-valued, stationary, centered
Gaussian, \(m\)-dependent sequence. That is, for every \(\ell\in\mathbb Z\), the
\(\sigma\)-fields generated by \(\{X_k:k\le \ell\}\) and by
\(\{X_k:k\ge \ell+m+1\}\) are independent.

We verify Assumption~3.1. 
For all \(h,g\in \mathcal{H}_1\), the Cauchy--Schwarz inequality and stationarity imply
\[
\begin{aligned}
\big|\langle R(v)h,g\rangle_{H_1}\big|
&=
\big|
\mathbb E\big[W_h(X_1)W_g(X_{1+v})\big]
\big| \\
&\le
\Big(\mathbb E\big[W_h(X_1)^2\big]\Big)^{1/2}
\Big(\mathbb E\big[W_g(X_{1+v})^2\big]\Big)^{1/2} \\
&=
\|h\|_{\mathcal{H}_1}\|g\|_{\mathcal{H}_1}.
\end{aligned}
\]
Hence $\|R(v)\|_{\mathrm{op}(\mathcal{H}_1)}\le 1$ for $v\in\mathbb Z$.

Moreover, if \(|v|>m\), then \(X_1\) and \(X_{1+v}\) are independent by
\(m\)-dependence. Since both \(W_h(X_1)\) and \(W_g(X_{1+v})\) are centered
Gaussian random variables, it follows that
\[
\mathbb E\big[W_h(X_1)W_g(X_{1+v})\big]
=
\mathbb E[W_h(X_1)]\,\mathbb E[W_g(X_{1+v})]
=
0.
\]
Therefore, $R(v)=0$ whenever $|v|>m$. 

Combining the two observations, for every \(q\ge1\),
\[
\sum_{v\in\mathbb Z}\|R(v)\|_{\mathrm{op}(\mathcal{H}_1)}^q
=
\sum_{|v|\le m}\|R(v)\|_{\mathrm{op}(\mathcal{H}_1)}^q
\le
2m+1
<\infty.
\]
\MD{Thus Assumption~\ref{thm:ass-2.2} holds. Consequently, for every
$G\in L^2(\clh_1,\gamma_Q:\clh_2)$ with Hermite rank $q\ge1$, Theorem~\ref{thm:main} applies.}
\end{example}

A stationary sequence $\{X_k\}_{k \in \ZZ}$ of $\clh_1$-valued random variables is called an AutoRegressive Hilbertian process of order one (or functional autoregressive process) (ARH(1)) associated with $(\mu, \veps,\psi)$ if
\begin{equation}\label{eq:ARH1-def}
    X_k - \mu = \psi(X_{k-1} - \mu) + \veps_k, \quad k \in \ZZ,
\end{equation}
where $\veps = \{\veps_k: k \in \ZZ\}$ is an $\clh_1$-valued white noise, $\mu \in \clh_1$, and $\psi \in \mathcal{S}_2(\clh_1)$. Without loss of generality, we set $\mu = 0_{\clh_1}$. Let, for simplicity, $\psi$ be a compact symmetric operator that is diagonalized by the same basis as $Q$, that is, 
\begin{equation}\label{def:rho-compact-symmetric}
    \psi = \sum_{j=0}^{\infty} \alpha_j  e_j \otimes e_j, \quad \text{or} \quad \psi(e_j) = \alpha_j e_j, \quad \psi^{|v|} = \sum_{j=1}^{\infty} \alpha_j^{|v|}  e_j \otimes e_j, \quad v \in \ZZ,
\end{equation}
where $\alpha_j$ is a decreasing, positive sequence with $\alpha_1< 1$ and $\lim_{j \to \infty} \alpha_j = 0$. 
The case of a general ARH($p$) series is  presented in the Section \ref{subseexpautocovop}.  

Several options for the noise $\{\veps_k\}$ are possible. Here, we let $\{W_t: t \in \RR\}$ be a Brownian motion in $L^2([0,1])$ and
\begin{equation} \label{eq:vareps:WnWn}
    \veps_k(\cdot) \doteq W_{k+\cdot} - W_{k} \in L^2([0,1]), \quad k \in \ZZ.
\end{equation}
In particular, $\{\veps_k\}$ is a Gaussian sequence.

\begin{example}[ARH(1) process] \label{example1_cond_ARH1}
Suppose $\{X_{k}\}_{k \in \ZZ}$ is an ARH(1) process associated with $(0_{\clh_1}, \veps, \psi)$  defined in \eqref{eq:ARH1-def}--\eqref{eq:vareps:WnWn}. Then, from Theorem 3.2 in \cite{Bosq2000:Linear},
\begin{equation*} \label{eq:correlation_ARH1}
    R(v) = \psi^{|v|}, \quad  v \in \ZZ,
\end{equation*}
where the equality follows and crucially since $Q$ and $\psi$ are assumed to be diagonalized by the same eigenbasis, hence the two operators commute. For $q \ge 1$,
\begin{equation} \label{eq:S_X_ARH1}
    \sum_{v \in \ZZ} \| R(v) \|_{\op(\clh_1)}^q = \sum_{v \in \ZZ} \|\psi\|_{\op(\clh_1)}^{q |v|} 
    < \infty.
\end{equation}
\MD{Thus Assumption \ref{thm:ass-2.2} holds, and Theorem \ref{thm:main} applies to every $G\in L^2(\clh_1,\gamma_Q:\clh_2)$ with Hermite rank $q\ge1$.} 
\end{example}

The following family of examples shows how Assumption \ref{thm:ass-2.2} can be ensured by some simplified conditions in the case of diagonal Karhunen--Lo\`eve type expansions.

\begin{example}[Karhunen--Lo\`eve Gaussian processes]
\label{ex:KL-diagonal}
Fix a positive summable sequence $\{\lambda_m\}_{m\ge 1} \in \ell^1(\NN)$, and let $\{\xi_{k,m}\}_{k\in\ZZ, m\ge 1}$ be a centered Gaussian family such that for each $m$, the $\RR$-valued process $\{\xi_{k,m}\}_{k\in\ZZ}$ is stationary with
\begin{equation}
\E[\xi_{0,m}\xi_{v,m'}] = \delta_{m,m'} \rho_m(v),\qquad v\in\ZZ,
\end{equation}
and $\rho_m(0) = 1$ for all $m \in \NN$.
Define the $\clh_1$-valued stationary Gaussian process
\begin{equation}
\label{eq:KL-process}
X_k \doteq \sum_{m=1}^\infty \sqrt{\lambda_m} \xi_{k,m} e_m,\qquad k\in\ZZ.
\end{equation}
Then $X_k$ is well-defined in $L^2(\Omega:\clh_1)$ and has covariance operator
$Q=\sum_{m\ge 1}\lambda_m  (e_m\otimes e_m) \in \mathcal{S}_1(\clh_1)$.
Moreover, its lag-$v$ autocovariance operator $\Gamma(v)\doteq \E[X_v\otimes X_0]$
is diagonal in the basis $\{e_m\}$ and given by
\begin{equation}
\label{eq:Gamma-diagonal}
\Gamma(v) = \sum_{m=1}^\infty \lambda_m \rho_m(v) (e_m\otimes e_m), \quad \text{so that } R(v)  = \sum_{\lambda_m>0}\rho_m(v)(e_m\otimes e_m),
\end{equation}
where $Q^{-1/2}$ is understood on $(\ker Q)^\perp$. 
In particular, if the sequence $\{\left(\sup_{m>0}|\rho_m(v)|\right)^q\}_{v \in \ZZ}$ is summable, then 
\[
\sum_{v\in\ZZ}\|R(v)\|_{\op(\clh_1)}^q
\le \sum_{v \in \ZZ} \left(\sup_{m>0}|\rho_m(v)|\right)^q  
<\infty.
\]
\end{example}

\subsection{Examples on the autocovariance operator} \label{subseexpautocovop}

We continue with an example for processes $\{X_n\}_{n \in \ZZ}$ admitting a linear representation. As pointed out in Section 7 of \cite{Bosq2000:Linear}, such linear representations can be interpreted as Wold decompositions. In finite dimensions, it is well known that any weakly stationary sequence admits a Wold decomposition. We consider here the purely stochastic case. 

Let $\{\varepsilon_n\}_{n\in\ZZ}$ be a (second-order) Gaussian white noise in $\clh_1$ with
$\E[\varepsilon_n]=0$ and covariance operator
\begin{equation}
C_{\varepsilon} \doteq \E(\varepsilon_0\otimes \varepsilon_0)\in \mathcal{S}_1(\clh_1),
\qquad \|C_\varepsilon\|_{\mathcal S_1(\clh_1)}<\infty.
\end{equation}
Assume that $\{X_n\}_{n\in\ZZ}$ admits the linear representation
\begin{equation} \label{eq:linear-series}
X_n=\sum_{j=0}^\infty A_j[\varepsilon_{n-j}],\qquad n\in\ZZ,
\end{equation}
where $A_j\in\mathcal L(\clh_1)$ are absolutely summable in operator norm when raised to the power $2q/ (q+1)$, where $q$ is the Hermite rank of the operator $G$, i.e.,
\begin{equation}\label{eq:abs-summ-Aj}
\sum_{j=0}^{\infty}
\|A_j\|_{\op(\mathcal H_1)}^{2q/(q+1)}
<\infty .
\end{equation}

\begin{example}[Linear process] \label{example2_cond}
Assume that \(q\in\mathbb N\). For \(v\ge 0\), the lag-\(v\)
covariance operator satisfies
\[
\Gamma(v)
=
\E(X_{1+v}\otimes X_1)
=
\sum_{j=0}^{\infty} A_{j+v}C_\varepsilon A_j^*,
\]
and \(\Gamma(v)=\Gamma(-v)^*\) for \(v<0\). We use the convention that
\(A_m=0\) whenever \(m<0\). Then, for every \(v\in\mathbb Z\),
\begin{align}
\sum_{v\in\mathbb Z}
\|\Gamma(v)\|_{\mathcal S_1(\mathcal H_1)}^q
&\le   \sum_{v\in\ZZ} \left( \sum_{j=0}^\infty \|A_{j+v}C_\varepsilon A_j^*\|_{\mathcal S_1(\clh_1)} \right)^q \notag \\
&\le
\|C_\varepsilon\|_{\mathcal S_1(\mathcal H_1)}^q
\sum_{v\in\mathbb Z}
\left(
\sum_{j=0}^{\infty}
\|A_{j+v}\|_{\op(\mathcal H_1)}
\|A_j\|_{\op(\mathcal H_1)}
\right)^q, \label{eq:3737}
\end{align}
where \eqref{eq:3737} follows by Theorem 18.11 (g) in \cite{conway2000course}. For each fixed \(v\in\mathbb Z\), with explanations given below, we get
\begin{align}
&
\left(
\sum_{j=0}^{\infty}
\|A_{j+v}\|_{\op(\mathcal H_1)}
\|A_j\|_{\op(\mathcal H_1)}
\right)^q \notag
\\
&=
\left(
\sum_{j=0}^{\infty}
\Big(
\|A_{j+v}\|_{\op(\mathcal H_1)}
\|A_j\|_{\op(\mathcal H_1)}
\Big)^{\frac{2}{q+1}}
\Big(
\|A_{j+v}\|_{\op(\mathcal H_1)}
\|A_j\|_{\op(\mathcal H_1)}
\Big)^{1-\frac{2}{q+1}}
\right)^q \notag
\\
&\le
\left(
\sum_{j=0}^{\infty}
\|A_{j+v}\|_{\op(\mathcal H_1)}^\frac{2q}{q+1}
\|A_j\|_{\op(\mathcal H_1)}^\frac{2q}{q+1}
\right)
\left(
\sum_{j=0}^{\infty}
\|A_{j+v}\|_{\op(\mathcal H_1)}^{\frac{2q}{2(q+1)}}
\|A_j\|_{\op(\mathcal H_1)}^{\frac{2q}{2(q+1)}}
\right)^{q-1} \notag\\
&\le \left(
\sum_{j=0}^{\infty}
\|A_{j+v}\|_{\op(\mathcal H_1)}^\frac{2q}{q+1}
\|A_j\|_{\op(\mathcal H_1)}^\frac{2q}{q+1}
\right) 
\left(
\left(
\sum_{j=0}^{\infty}
\|A_{j+v}\|_{\op(\mathcal H_1)}^{\frac{2q}{q+1}}
\right)^{1/2}
\left(
\sum_{j=0}^{\infty}
\|A_j\|_{\op(\mathcal H_1)}^{\frac{2q}{q+1}}
\right)^{1/2}
\right)^{q-1} \notag\\
&\le \left(
\sum_{j=0}^{\infty}
\|A_{j+v}\|_{\op(\mathcal H_1)}^\frac{2q}{q+1}
\|A_j\|_{\op(\mathcal H_1)}^\frac{2q}{q+1}
\right) 
\left( 
\sum_{j=0}^{\infty}
\|A_j\|_{\op(\mathcal H_1)}^{\frac{2q}{q+1}}
\right)^{q-1}, \label{eq:3838}
\end{align}
where Hölder's and the Cauchy--Schwarz inequalities were applied.

Therefore, combining \eqref{eq:3737} with \eqref{eq:3838}
\begin{align}
\sum_{v\in\mathbb Z}
\|\Gamma(v)\|_{\mathcal S_1(\mathcal H_1)}^q
&\le 
\|C_\varepsilon\|_{\mathcal S_1(\mathcal H_1)}^q  \sum_{v\in\mathbb Z}
\left(
\sum_{j=0}^{\infty}
\|A_{j+v}\|_{\op(\mathcal H_1)}
\|A_j\|_{\op(\mathcal H_1)}
\right)^q \notag
\\
&\le 
\|C_\varepsilon\|_{\mathcal S_1(\mathcal H_1)}^q 
\left(
\sum_{j=0}^{\infty}
\|A_j\|_{\op(\mathcal H_1)}^{\frac{2q}{q+1}}
\right)^{q-1}
\sum_{v\in\mathbb Z}
\sum_{j=0}^{\infty}
\|A_{j+v}\|_{\op(\mathcal H_1)}^{\frac{2q}{q+1}}
\|A_j\|_{\op(\mathcal H_1)}^{\frac{2q}{q+1}} \notag
\\
&= 
\|C_\varepsilon\|_{\mathcal S_1(\mathcal H_1)}^q 
\left(
\sum_{j=0}^{\infty}
\|A_j\|_{\op(\mathcal H_1)}^{\frac{2q}{q+1}}
\right)^{q-1}
\sum_{j=0}^{\infty}
\|A_j\|_{\op(\mathcal H_1)}^{\frac{2q}{q+1}}
\sum_{v\in\mathbb Z}
\|A_{j+v}\|_{\op(\mathcal H_1)}^{\frac{2q}{q+1}} \notag
\\
&= 
\|C_\varepsilon\|_{\mathcal S_1(\mathcal H_1)}^q 
\left(
\sum_{j=0}^{\infty}
\|A_j\|_{\op(\mathcal H_1)}^{\frac{2q}{q+1}}
\right)^{q+1}
\end{align}
and so condition \eqref{eq:abs-summ-Aj} implies condition \eqref{thm:ass-2.1-eq-Gamma}.
\end{example}

Arguably one of the most popular classes of models in functional time series analysis is that of autoregressive Hilbertian processes of order \(p\) (ARH\((p)\)). We  present this class and give sufficient conditions ensuring that its unique stationary solution admits a linear representation as in Example~\ref{example2_cond}.

Let \(\{\varepsilon_n\}_{n\in\ZZ}\) be as in Example~\ref{example2_cond}. Consider the recursion
\begin{equation}\label{eq:arhp}
X_n = \sum_{k=1}^p \Phi_k[X_{n-k}] + \varepsilon_n,\qquad n\in\ZZ,
\end{equation}
where \(\Phi_k\in\mathcal L(\clh_1)\) are bounded linear operators. Associated with the coefficients \(\Phi_k\), \(k=1,\dots,p\), is the polynomial
\begin{equation}
Q(z) = z^p I - z^{p-1}\Phi_1 - \cdots - z\Phi_{p-1} - \Phi_p,
\qquad z\in\mathbb C.
\end{equation}
For every \(z\in\mathbb C\), \(Q(z)\) is a bounded linear operator on the complex extension \(\widetilde{\clh}_1\) of \(\clh_1\). If the following statement holds
\begin{equation}\label{eq:arhp-condition}
Q(z)\ \text{is not invertible implies that } |z|<1,
\end{equation}
then the recursion \eqref{eq:arhp} admits a unique centered stationary solution, and this solution has the linear representation \eqref{eq:linear-series}.

To see that \eqref{eq:abs-summ-Aj} holds, the proof of Theorem 5.2 of \cite{Bosq2000:Linear} gives 
\[
\mathcal A \doteq
\begin{pmatrix}
\Phi_1 & \Phi_2 & \cdots & \Phi_{p-1} & \Phi_p \\
\Id_{\clh_1} & 0 & \cdots & 0 & 0 \\
0 & \Id_{\clh_1} & \cdots & 0 & 0 \\
\vdots & \vdots & \ddots & \vdots & \vdots \\
0 & 0 & \cdots & \Id_{\clh_1} & 0
\end{pmatrix}
\in \mathcal L(\clh_1^p), \quad \text{then }\|\mathcal A^j\|_{\op(\clh_1)}\le C\eta^j,\qquad j\in\NN_0,
\]
for some $\eta \in (0,1)$.
Moreover, letting $P$ denote the projection operator $P(x_1,\dots,x_p)=x_1$ and $J$ the canonical embedding $Jx=(x,0,\dots,0)$, then we can write $A_j = P \mathcal A^j J$ for all $j \in \NN_0$ and hence
\[
\|A_j\|_{\op(\clh_1)}
\le
\|P\|_{\op(\clh_1)} \|\mathcal A^j\|_{\op(\clh_1)} \|J\|_{\op(\clh_1)}
\le
\|P\|_{\op(\clh_1)}  C \eta^j  \|J\|_{\op(\clh_1)}.
\]
and so \eqref{eq:abs-summ-Aj} follows for all $q \ge 1$.

A simple, easily verifiable sufficient condition for~\eqref{eq:arhp-condition} is
\begin{equation}
\sum_{k=1}^p \|\Phi_k\|_{\op(\clh_1)}<1,
\end{equation}
which implies stability and yields \eqref{eq:linear-series}--\eqref{eq:abs-summ-Aj}; see Corollary 5.3 in \cite{Bosq2000:Linear}.

\begin{example}[ARH\((p)\) model]\label{example1_cond}
Suppose that \(\{X_k\}\) is the unique centered Gaussian stationary solution of \eqref{eq:arhp}, and assume that condition~\eqref{eq:arhp-condition} holds. Then there exist coefficients \(\{A_j\}_{j\ge0}\subset\mathcal L(\clh_1)\) such that
\begin{equation}\label{eq:arhp-ma}
X_n=\sum_{j=0}^\infty A_j[\varepsilon_{n-j}],\qquad n\in\ZZ,
\qquad\text{and}\qquad
\sum_{j=0}^\infty \|A_j\|_{\op(\clh_1)}<\infty.
\end{equation}
Consequently, Example~\ref{example2_cond} applies and so Condition \eqref{thm:ass-2.1-eq-Gamma} is true for all $q \ge 1$.
\end{example}

\section{A general Wiener-Itô chaos decomposition} \label{sec:Wiener-Ito-chaos}

\subsection{Hermite expansion} \label{se:Hermite}

In Section \ref{subsec:prelim-infinite}, we briefly introduced the Hermite expansion for functions in $L^2(\RR,\phi(x) dx:\RR)$ and its generalization to $L^2(\clh_1,\gamma_Q:\clh_2)$.
Here, we formally state and prove the generalized Hermite expansion. 

Recall the orthogonality property for real Hermite polynomials from \eqref{eq:hermite-orthogonality} and the generalized Hermite polynomials from \eqref{eq:Hermite_product}. We prove here the technical Lemma \ref{thm:hermite-expansion-general} stated above.

\begin{proof}[Proof of Lemma \ref{thm:hermite-expansion-general}]
For shortness' sake we write 
\begin{equation*}
\{ \Gamma_{\bfr i}[\cdot]\}_{i \in \NN, \bfr \in \cll}, 
\hspace{0.2cm}
\text{ with }
\hspace{0.2cm}
\Gamma_{\bfr i}[x] \doteq H_\bfr[x] \otimes v_i, \quad x \in \clh_1.
\end{equation*}
We prove that the family $ \{ \Gamma_{\bfr i}  \}_{i \in \NN, \bfr \in \cll}$ is orthogonal and complete in $L^2(\clh_1,\gamma_Q:\RR) \otimes \clh_2$.

\textit{Orthogonality:}
Fix $i,j \in \NN$ and $\bfr,\bfs \in \cll$, and let $M \doteq \max\{M_\bfr,M_\bfs\} < \infty$ (see \eqref{eq:defn-Ml}). Denote $\delta_{\bfr \bfs} = \prod_{k=1}^d \delta_{r_k s_k}$ whenever $\bfr,\bfs \in \NN_0^d$. Then, with explanations given below,
\begin{align}
    &\langle \Gamma_{\bfr i} ,\Gamma_{\bfs j}  \rangle_{L^2(\clh_1,\gamma_Q) \otimes \clh_2} 
    \nonumber \\
    &=  \left\langle 
    \prod_{m=1}^M H_{r_m}[W_{e_m}] \otimes v_i, 
    \prod_{m=1}^M H_{s_m}[W_{e_m}] \otimes v_j \right\rangle_{L^2(\clh_1,\gamma_Q) \otimes \clh_2} 
    \nonumber \\
    &=  \left\langle 
    \prod_{m=1}^M H_{r_m}[W_{e_m}],  
    \prod_{m=1}^M H_{s_m}[W_{e_m}]   \right\rangle_{L^2(\clh_1,\gamma_Q)} 
    \langle v_i,v_j \rangle_{\clh_2} 
    \label{eq:orthogonal3} \\
    &=  \int_{\clh_1} \prod_{m=1}^M H_{r_m}(W_{e_m}(x)) H_{s_m}(W_{e_m}(x)) \gamma_Q(dx) \delta_{ij} 
    \nonumber \\
    &= \prod_{m=1}^M \int_{\clh_1}
     H_{r_m}(W_{e_m}(x)) 
    H_{s_m}(W_{e_m}(x)) \gamma_Q(dx) \delta_{ij} 
    \label{eq:orthogonal5} \\
    &= \prod_{m=1}^M 
    \langle Q^{-1/2} e_m, QQ^{-1/2} e_m \rangle^{r_m}_{\clh_1} r_m! \delta_{i j} \delta_{\bfr \bfs}
    = \prod_{m=1}^M r_m! \delta_{r_m s_m} \delta_{ij},
    \label{eq:orthogonal6}
\end{align}
where \eqref{eq:orthogonal3} follows from \eqref{eq:tensor_inner_product}, \eqref{eq:orthogonal5} follows since the random variables $W_{u}, W_{v}$ are independent for $u \neq v$, and \eqref{eq:orthogonal6} follows from \eqref{eq:hermite-orthogonality} and \eqref{eq:white-noise-orthog}.

\textit{Completeness:}
Let $\widetilde \psi \in L^2(\clh_1,\gamma_Q:\clh_2)$ and assume that
\begin{equation}\label{eq:completeness_condition}
\big\langle \widetilde \psi , \Gamma_{\bfl i} \big\rangle_{L^2(\clh_1,\gamma_Q)\otimes \clh_2}=0,
\qquad \text{for all } \bfl \in \cll \text{ and } i\in \NN.
\end{equation}
Fix an orthonormal basis $\{v_i\}_{i\in\NN}$ of $\clh_2$ and define the scalar-coordinate functions
\begin{equation}
\psi_i(x) \doteq \big\langle \widetilde \psi(x), v_i \big\rangle_{\clh_2},
\qquad x\in \clh_1.
\end{equation}
Then $\psi_i \in L^2(\clh_1,\gamma_Q:\RR)$ for every $i\in\NN$ and, using
$\Gamma_{\bfl i}[x]=H_{\bfl}[x]\otimes v_i$ and Fubini's theorem,
\begin{align*}
0
= \big\langle \widetilde \psi , \Gamma_{\bfl i} \big\rangle_{L^2(\clh_1,\gamma_Q)\otimes \clh_2}
&= \int_{\clh_1} \big\langle \widetilde \psi(x),  H_{\bfl}[x] v_i \big\rangle_{\clh_2}  \gamma_Q(dx) \\
&= \int_{\clh_1} \psi_i(x)  H_{\bfl}[x]  \gamma_Q(dx),
\qquad \text{for all } \bfl\in\cll,\ i\in\NN.
\end{align*}
Since $\{H_{\bfl}\}_{\bfl\in\cll}$ is complete in $L^2(\clh_1,\gamma_Q:\RR)$ (see, e.g.,
Theorem~1.1.1 in \cite{Nua06book} or Theorem~9.7 in \cite{DaPr06}),
it follows that $\psi_i=0$ $\gamma_Q$-a.s.\ for every $i\in\NN$.
Hence $\widetilde \psi(x)=0_{\clh_2}$ $\gamma_Q$-a.s., proving completeness.
\end{proof}

\begin{defn} \label{def:Hermiterank}
    Let $f \in L^2(\clh_1,\gamma_Q:\clh_2)$ and let $c_{i,\bfl}$, $i \in \NN, \bfl \in \cll$ be the coefficients given by \eqref{eq:coeff-hermite-general} in the representation \eqref{eq:Hermite_expansion}. Define the \textit{Hermite rank} of the operator $f$ to be 
    \begin{equation}
    \rk(f) = \min\left\{ q \in \NN  ~\bigg|~  \text{there is a } \bfl \in \cll, \text{with}   \sum_{k=1}^\infty l_k = q \text{ and }  c_{i, \bfl} \neq 0   \text{ for some } i \in \NN  \right\}.
    \end{equation}
\end{defn}

\subsection{Wiener chaos expansion}
The following lemma formalizes the chaos decomposition for the partial sums, by using the Hermite expansion developed in Section \ref{se:Hermite}. Henceforth, we   denote the usual multiple Wiener-Itô integral by $I_p$ and its $\clh_2$-valued counterpart by $\cli_p$.

\begin{lemma}\label{cor:Hilbert_expansion_Chaos_decomposition}
The partial sums $S_n$ in \eqref{eq:partial_sum_G} admit the Wiener-Itô chaos decomposition
    \begin{equation} \label{eq:Hilbert_expansion_Chaos_decomposition}
    S_n = \sum_{p=q}^\infty \cli_p (h_{n,p}) = 
    \sum_{p=q}^\infty 
    ( I_{p} \otimes \operatorname{Id}_{\clh_{2}} ) [h_{n,p}]
    \hspace{0.2cm}
    \text{ with }
    \hspace{0.2cm}
    h_{n,p} = \sum_{i=1}^\infty \left( h_{n,p,i} \otimes v_i \right),
    \end{equation}
    where $I_p$ is the $p$-th multiple Wiener–Itô integral, $q \doteq \rk(G)$, and
    \begin{equation} \label{eq:hpnitilde}
    h_{n,p,i}
    =
    \frac{1}{\sqrt{n}} \sum_{k=1}^n \sum_{\bm{j} \in \NN^p} b_{i,\bm{j}} (\veps_{j_1 k} \otimes \dots \otimes \veps_{j_p k}) \in \mfh^{\odot p},
    \end{equation}
where, for each $p \ge q$, $\{ b_{i,\bm{j}} \doteq  b_i(j_{1}, \dots, j_{p}) ~|~ j_{1}, \dots, j_{p} \geq 1\}$ with $\sum_{\bm{j} \in \NN^p} (b_{i,\bm{j}})^2 <\infty $ and, for all $i \in \NN$, $b_{i,\bm{j}}$ is a symmetric array of real numbers in $\bm{j}$.
\end{lemma}

\begin{proof}
Using the Hermite expansion from Lemma \ref{thm:hermite-expansion-general}, we can rewrite, with further explanations given below, the series of partial sums as
\begin{align}
    S_n 
= 
    \frac{1}{\sqrt{n}} \sum_{k=1}^n G[X_k] 
&=
    \frac{1}{\sqrt{n}} \sum_{k=1}^n \sum_{i=1}^\infty 
    \sum_{\bl \in \cll} c_{i,\bl} \prod_{j=1}^{M_{\bfl}} H_{l_j}( \langle Q^{-\frac{1}{2}} e_{j} , X_k \rangle_{\clh_1} ) \otimes v_i  
    \nonumber
\\
&=
    \frac{1}{\sqrt{n}} \sum_{k=1}^n \sum_{i=1}^\infty  \sum_{\bl \in \cll} c_{i,\bl} \prod_{j=1}^{M_{\bfl}} H_{l_j}( X(\varepsilon_{jk}) ) \otimes v_i  
\label{eq:Sn-1-1}\\
&=
    \frac{1}{\sqrt{n}} \sum_{k=1}^n \sum_{i=1}^\infty \sum_{p = q}^\infty  \sum_{\bl \in \cll_p} c_{i,\bl} \prod_{j=1}^{M_{\bfl}} H_{l_j}( X(\varepsilon_{jk}) ) \otimes v_i  
\label{eq:Sn-1-2} \\
&=
    \sum_{i=1}^\infty \sum_{p=q}^\infty I_p(h_{n,p,i}) \otimes v_i ,
\label{eq:Sn-1-3}
\end{align}
where \eqref{eq:Sn-1-1} follows from Proposition 7.2.3 of \cite{nourdin2012normal} by setting $W_{e_j}(X_k) = X(\varepsilon_{jk})$, where $\{\veps_{jk}\}_{j \in \NN, k \in \ZZ}$ is an underlying isonormal process. Line \eqref{eq:Sn-1-2} follows by defining $\cll_m \doteq \{\boldsymbol{l} \in \NN_0^\infty: \sum_{k=1}^\infty l_k = m \}$ and recalling that $\rk(G) = q$. Line \eqref{eq:Sn-1-3} follows from the discussion in Chapter 8 of \cite{PecTaqbook11} and the linearity of $I_p$. Finally, \eqref{eq:Sn-1-3} gives
\begin{equation}
    S_n = \sum_{p=q}^\infty (I_p \otimes \operatorname{Id}_{\clh_2}) \left[ \sum_{i=1}^\infty \left( h_{n,p,i} \otimes v_i\right)   \right]
    = \sum_{p=q}^\infty 
    ( I_{p} \otimes \operatorname{Id}_{\clh_{2}} ) [h_{n,p}], 
\label{eq:Sn-1-4}
\end{equation}
by recalling the identities in \eqref{eq:Hilbert_expansion_Chaos_decomposition}--\eqref{eq:hpnitilde} and the operation 
$(\mathcal{A} \otimes \mathcal{B})[a \otimes b] = (\mathcal{A}[a] \otimes \mathcal{B}[b])$. The relation \eqref{eq:Sn-1-4} concludes the proof.
\end{proof}

\begin{remark} 
The form of the coefficients $b_{i,\bm{j}}$ from Lemma \ref{cor:Hilbert_expansion_Chaos_decomposition} can be given explicitly in terms of the Hermite coefficients as follows. 
    For $\bm{j}=(j_1,\dots,j_p)\in\NN^p$, define its associated
multiplicity sequence for $m \in \NN$ $\boldsymbol{\ell}(\bm{j})\in\cll_p$ and the coefficients $b_{i,\bm{j}}$ by
\[
    \ell_m(\bm{j})
    \doteq 
    \#\{r\in\{1,\dots,p\}: j_r=m\}, \quad 
       b_{i,\bm{j}}
    =
    b_i(j_1,\dots,j_p)
    \doteq
    \frac{\prod_{m=1}^{\infty} \ell_m(\bm{j})!}{p!}
    c_{i,\boldsymbol{\ell}(\bm{j})}.
\]
In particular, $b_{i,\bm{j}}$ is symmetric in $\bm{j}$.
\end{remark}

In the sequel, for $p\ge q$, we denote by $G^p$ the $p$-th chaotic component of $G$ whenever $G \in  L^2(\clh_1,\gamma_Q :\clh_2)$, namely
\begin{equation} \label{eq:def-Gp}
        G^p[x]
    \doteq
    \sum_{i=1}^{\infty}
    \sum_{\bl\in\cll_p}
    c_{i,\bl} H_{\bl}[x]\, v_i,
    \qquad x\in\clh_1.
\end{equation}

\section{Proofs of main results} \label{se:proofs}

\subsection{Proof of Theorem \ref{thm:main}}

The proof of the Breuer-Major theorem of \cite{nourdin2011quantitative} for $\clh_1 = \RR^d$ crucially exploits the finite dimensionality of the underlying Gaussian process $\{X_k\}_{k \in \ZZ}$ and is no longer available. 

For the proof of Theorem \ref{thm:main}, we use the results of \cite{duker2025fourth}. For completeness we recall Theorem 3.4 in \cite{duker2025fourth} here: Recall the chaos decomposition \eqref{eq:Hilbert_expansion_Chaos_decomposition}--\eqref{eq:hpnitilde} of $S_n$ with respective covariance operators $\clt_{S_n}$. Let $\{Z_r: r \ge 1\}$ be a family of centered $\clh_2$-valued normal random variables with respective covariance operators $\clt_{Z_r}$ and suppose:
\begin{enumerate}[label=(\roman*), ref=(\roman*), align=left]
    \item For $r \ge 1$, $\| \clt_{\mathcal{I}_r(h_{n,r})} - \clt_{Z_r}\|_{\mathcal{S}_1(\clh_2)} \to 0$. \label{item1:thm-infinite-expansions-covar-oper}
    \item \MD{Letting $\clt_Z \doteq \sum_{r=1}^\infty \clt_{Z_r}$, we have $\|\clt_Z\|_{\mathcal{S}_1(\clh_2)}<\infty$.}
    \label{item1.1:thm-infinite-expansions-trace-sec-mom}
    \item For $i\in \NN$ and $r \ge 2$,
    \begin{equation}
    \lim_{n \to \infty} \|  h_{n,r,i}  \otimes_m  h_{n,r,i}  \|_{\mfh^{\otimes 2(r - m)}} = 0,
    \end{equation}
    for all $m=1,\dots,r-1$.  \label{item:cond2}
    \item It holds that \label{item:cond1}
    \begin{equation}
        \lim_{N \to \infty} \sup_{n \ge 1}  \sum_{r=N+1}^\infty \sum_{i=1}^\infty r! \|h_{n,r,i}\|_{\mfh^{\otimes r}}^2 = 0.
    \end{equation}
\end{enumerate}
\MD{Then, $S_n \overset{d}{\to} Z$, where $Z$ is a centered Gaussian element of $\clh_2$ with covariance operator $\covop_Z$.}

We introduce the following identities based on the chaos decomposition
\eqref{eq:Hilbert_expansion_Chaos_decomposition}--\eqref{eq:hpnitilde}:
\begin{equation} \label{eq:finite-secon-moment-b}
    p! \sum_{i=1}^\infty \sum_{\bfl \in \NN^p} b_{i,\bfl}^2
    =
    \E \|G^p[X_1]\|_{\clh_2}^2 < \infty,
    \qquad
    \sum_{p=q}^\infty p! \sum_{i=1}^\infty \sum_{\bfl \in \NN^p} b_{i,\bfl}^2
    =
    \E \|G[X_1]\|_{\clh_2}^2 < \infty,
\end{equation}
where $G^p$, $q \le p < \infty$, denotes the term of order $p$ appearing in the chaotic expansion
\eqref{eq:Hilbert_expansion_Chaos_decomposition}--\eqref{eq:hpnitilde} of
$S_1 = G[X_1]$, defined in \eqref{eq:def-Gp}.

We show that the derived chaos decomposition
\eqref{eq:Hilbert_expansion_Chaos_decomposition} for $S_n$ satisfies
conditions \ref{item1:thm-infinite-expansions-covar-oper}--\ref{item:cond1}.

\noindent
\textit{Condition \ref{item1:thm-infinite-expansions-covar-oper}:}
We first derive the covariance operator $\clt_{\mathcal{I}_p(h_{n,p})}$ of an individual chaos in terms of \eqref{eq:hpnitilde} as
\begin{align}
    \clt_{\cli_p(h_{n,p})}
    &=
    \E\left[
        \cli_p(h_{n,p}) \otimes \cli_p(h_{n,p})
    \right]
    \nonumber \\
    &=
    \sum_{i,j=1}^{\infty}
    \E\left[
        I_p(h_{n,p,i}) I_p(h_{n,p,j})
    \right]
    (v_i\otimes v_j)
    \nonumber \\
    &=
    p!
    \sum_{i,j=1}^{\infty}
    \left\langle
        h_{n,p,i}, h_{n,p,j}
    \right\rangle_{\mfh^{\otimes p}}
    (v_i\otimes v_j)
    \nonumber \\
    &=
    \frac{p!}{n}
    \sum_{i,j=1}^{\infty}
    \sum_{k,l=1}^{n}
    \sum_{\bfr,\bfs\in\NN^p}
    b_{i,\bfr} b_{j,\bfs}
    \prod_{a=1}^{p}
    \left\langle
        \veps_{r_a k}, \veps_{s_a l}
    \right\rangle_{\mfh}
    (v_i\otimes v_j)
    \nonumber \\
    &=
    \sum_{i,j=1}^{\infty}
    p!
    \sum_{|v|<n}
    \left(1-\frac{|v|}{n}\right)
    \sum_{\bfr,\bfs\in\NN^p}
    b_{i,\bfr} b_{j,\bfs}
    \prod_{a=1}^{p}
    \rho_{r_a s_a}(v)
    (v_i\otimes v_j).
\end{align}
The candidate limiting covariance operator for a single chaos is then given by the presumed limit in $n$, i.e., 
\begin{equation}
     \clt_{Z_p}
    \doteq
    \sum_{i,j =1}^\infty
    p! \sum_{v \in \ZZ}
    \sum_{\bfr,\bfs \in \NN^p}
    b_{i,\bfr} b_{j,\bfs}
    \prod_{k=1}^p \rho_{r_k s_k}(v)\,
    (v_i \otimes v_j).
\end{equation}
We wish to control their distance with respect to the trace-class norm
\begin{align}
    &\| \clt_{\mathcal{I}_p(h_{n,p})} - \clt_{Z_p}\|_{\mathcal{S}_1(\clh_2)}
    \nonumber
    \\
    &=
    \Bigg\|
    \sum_{i,j =1}^\infty
    p! \sum_{|v| < n} \left( 1 - \frac{|v|}{n} \right)
    \sum_{\bfr,\bfs \in \NN^p}
    b_{i,\bfr} b_{j,\bfs}
    \prod_{k=1}^p \rho_{r_k s_k}(v)\,
    (v_i \otimes v_j)
    \nonumber
    \\
    &\hspace{1cm}
    -
    \sum_{i,j =1}^\infty
    p! \sum_{v \in \ZZ}
    \sum_{\bfr,\bfs \in \NN^p}
    b_{i,\bfr} b_{j,\bfs}
    \prod_{k=1}^p \rho_{r_k s_k}(v)\,
    (v_i \otimes v_j)
    \Bigg\|_{\mathcal{S}_1(\clh_2)}
    \nonumber
    \\
    &=
    \Bigg\|
    \sum_{i,j =1}^\infty
    p! \sum_{|v| \geq n}
    \sum_{\bfr,\bfs \in \NN^p}
    b_{i,\bfr} b_{j,\bfs}
    \prod_{k=1}^p \rho_{r_k s_k}(v)\,
    (v_i \otimes v_j)
    \nonumber
    \\
    &\hspace{1cm}
    +
    \sum_{i,j =1}^\infty
    p! \sum_{|v| < n} \frac{|v|}{n}
    \sum_{\bfr,\bfs \in \NN^p}
    b_{i,\bfr} b_{j,\bfs}
    \prod_{k=1}^p \rho_{r_k s_k}(v)\,
    (v_i \otimes v_j)
    \Bigg\|_{\mathcal{S}_1(\clh_2)}.
    \label{eq:twosummands-condition-1}
\end{align}
We control the two summands in \eqref{eq:twosummands-condition-1} separately.

We first show that
\begin{align} \label{eq:traceoflimp}
    \| \clt_{Z_p} \|_{\mathcal{S}_1(\clh_2)}
    < \infty.
\end{align}
Note that $\rho_{rs}(v) = \rho_{sr}(-v)$ by \eqref{eq:cor_scores} and recall the correlation operator $R(v): \clh_1 \to \clh_1$ from \eqref{eq:R(v)-op}.
Henceforth, we write $\mathds{1}_{\{\mathfrak{C}=1\}}$ to denote the case in which Assumption \ref{thm:ass-2.2} holds and respectively $\mathds{1}_{\{\mathfrak{C}=2\}}$ for the case when Assumption \ref{thm:ass-2.1} holds. We introduce
\begin{align}
        \theta_\mfc &\doteq \begin{cases}
        \sum_{v \in \ZZ}
    \| R(v) \|_{\op(\clh_1)}^{q} & \text{if }\mfc =  1, \\
     \sum_{v \in \ZZ}
    \| \Gamma(v) \|_{\mathcal{S}_1(\clh_1)}^{q} & \text{if }\mfc =  2, \label{eq:de:theta_theta-1}
    \end{cases}  \\
    K_\mfc &\doteq \begin{cases}
        \inf_{k \in \NN}\{ \| R(v) \|_{\op(\clh_1)} \leq 1 \text{ for all } |v| \geq k \}  & \text{if }\mfc = 1, \\
        \inf_{k \in \NN}\{ \| \Gamma(v) \|_{\mathcal{S}_1(\clh_1)} \leq 1 \text{ for all } |v| \geq k \} & \text{if }\mfc =  2.
            \end{cases}  \label{eq:de:theta_theta-2}
\end{align}
It holds that $\theta_\mfc, K_\mfc < \infty$.
In the subsequent derivations, the subscript one is associated with Assumption \ref{thm:ass-2.2} and the subscript two with Assumption \ref{thm:ass-2.1}.

For each $i\ge1$, define further
\begin{equation} \label{eq:def-Cp-operator}
   \mathfrak{b}_{p,i}
\doteq
\sum_{\bfr \in \NN^p} b_{i,\bfr}
(e_{r_1}\otimes \cdots \otimes e_{r_p})
\in \clh_1^{\otimes p}, \quad \text{and }  C_p:\clh_2 \to \clh_1^{\otimes p},
    \; \text{so that }
    C_p v_i \doteq \mathfrak{b}_{p,i},
\end{equation}
where $\{v_i\}_{i\ge1}$ denotes the orthonormal basis of $\clh_2$ appearing in
\eqref{eq:hpnitilde}. Then $C_p$ is in $\mathcal{S}_2(\clh_2,\clh_1^{\otimes p})$, since
\begin{equation}
p! \|C_p \|_{\mathcal S_2(\clh_2,\clh_1^{\otimes p})}^2
=
p! \sum_{i=1}^\infty \|\mathfrak{b}_{p,i}\|_{\clh_1^{\otimes p}}^2
=
p! \sum_{i=1}^\infty \sum_{\bfr \in \NN^p} b_{i,\bfr}^2
=
\E \| G^p[X_1] \|^2_{\clh_2}
\leq
\E \| G[X_1] \|^2_{\clh_2}
< \infty.
\end{equation}

Moreover, for $\bfr=(r_1,\dots,r_p)$ and $\bfs=(s_1,\dots,s_p)$, we have
\[
\big\langle
R(v)^{\otimes p}
\big(
e_{r_1}\otimes\cdots\otimes e_{r_p}
\big),
e_{s_1}\otimes\cdots\otimes e_{s_p}
\big\rangle_{\clh_1^{\otimes p}}
=
\prod_{k=1}^p \rho_{r_k s_k}(v).
\]
Hence, we can recast $\clt_{Z_p}$ as follows
\begin{equation} \label{eq:rewrite-as-operators}
\begin{aligned}
   \clt_{Z_p} &= \sum_{v \in \ZZ} \sum_{i,j =1}^\infty
    p! \sum_{\bfr,\bfs \in \NN^p}
    b_{i,\bfr} b_{j,\bfs}
    \prod_{k=1}^p \rho_{r_k s_k}(v)\,
    (v_i \otimes v_j) \\
    &= \sum_{v \in \ZZ}
    \sum_{i,j =1}^\infty
    p!\,
    \langle R(v)^{\otimes p}\mathfrak{b}_{p,i}, \mathfrak{b}_{p,j}\rangle_{\clh_1^{\otimes p}}\,
    (v_i \otimes v_j) \\
    &=  p!\, \sum_{v \in \ZZ}
    C_p^* R(v)^{\otimes p} C_p.
\end{aligned}
\end{equation}

Turning back to \eqref{eq:traceoflimp}, i.e., to showing that $\clt_{Z_p}$ is trace-class, we apply the triangle inequality together with \eqref{eq:rewrite-as-operators} to obtain
\begin{align}
    \| \clt_{Z_p} \|_{\mathcal{S}_1(\clh_2)}
    &\leq
    \sum_{v \in \ZZ}
    \left\|
    \sum_{i,j =1}^\infty
    p! \sum_{\bfr,\bfs \in \NN^p}
    b_{i,\bfr} b_{j,\bfs}
    \prod_{k=1}^p \rho_{r_k s_k}(v)\,
    (v_i \otimes v_j)
    \right\|_{\mathcal{S}_1(\clh_2)}
    \nonumber
    \\
    &=
    p!
    \sum_{v \in \ZZ}
    \left\|
    C_p^* R(v)^{\otimes p} C_p
    \right\|_{\mathcal{S}_1(\clh_2)}
    \label{al:operator-notation-separate-sum-1}
    \\
    &\leq
    p!
    \sum_{|v| \leq K_{\mathfrak{C}}-1}
    \left\|
    C_p^* R(v)^{\otimes p} C_p
    \right\|_{\mathcal{S}_1(\clh_2)}
    +
    p!
    \sum_{|v| \geq K_{\mathfrak{C}}}
    \left\|
    C_p^* R(v)^{\otimes p} C_p
    \right\|_{\mathcal{S}_1(\clh_2)},
    \label{al:operator-notation-separate-sum}
\end{align}
for $\mathfrak{C}=1,2$, where we use the notation introduced in
\eqref{eq:de:theta_theta-1} and \eqref{eq:de:theta_theta-2}.

We consider the two summands in \eqref{al:operator-notation-separate-sum} separately, and introduce the block operator
\begin{equation} \label{eq:two-operators}
    \mathcal{R}^{(p)}(v):
\clh_1^{\otimes p}\oplus \clh_1^{\otimes p}
\to
\clh_1^{\otimes p}\oplus \clh_1^{\otimes p}, \quad \mathcal{R}^{(p)}(v)
    \doteq 
    \begin{pmatrix}
        R(0)^{\otimes p} & R(v)^{\otimes p} \\
        R(-v)^{\otimes p} & R(0)^{\otimes p}
    \end{pmatrix}.
\end{equation}
Note that $\mathcal{R}^{(p)}(v)$ is self-adjoint and nonnegative; see Lemma \ref{lem:Positivity of the block operators}.

For the first summand in \eqref{al:operator-notation-separate-sum}, the derivations under Assumptions \ref{thm:ass-2.2} and \ref{thm:ass-2.1} align. We get, for $\mfc=1,2$,
\begin{align}
    &p!
    \sum_{|v| \leq K_\mfc-1}
    \left\|
    C_p^* R(v)^{\otimes p} C_p
    \right\|_{\mathcal{S}_1(\clh_2)}
    \label{eq:R1argument1}
    \\
    &=
    p!
    \sum_{|v| \leq K_\mfc-1}
    \left\|
    \begin{pmatrix}
    I_{\clh_2} & 0 \\
    0 & 0
    \end{pmatrix}
    (C_p^* \oplus C_p^*)
    \mathcal R^{(p)}(v)
    (C_p \oplus C_p)
    \begin{pmatrix}
    0 & 0 \\
    I_{\clh_2} & 0
    \end{pmatrix}
    \right\|_{\mathcal{S}_1(\clh_2 \oplus \clh_2)}
    \label{eq:R1argument2}
    \\
    &\leq
    p!
    \sum_{|v| \leq K_\mfc-1}
    \left\|
    (C_p^* \oplus C_p^*)
    \mathcal R^{(p)}(v)
    (C_p \oplus C_p)
    \right\|_{\mathcal{S}_1(\clh_2 \oplus \clh_2)}
    \label{eq:R1argument3}
    \\
    &=
    p!
    \sum_{|v| \leq K_\mfc-1}
    \tr_{\clh_2 \otimes \clh_2}\!\left(
    (C_p^* \oplus C_p^*)
    \mathcal R^{(p)}(v)
    (C_p \oplus C_p)
    \right)
    \label{eq:R1argument4}
    \\
    &\leq
    p!\, 4K_\mfc\,
    \tr_{\clh_2}\!\big(C_p^* R(0)^{\otimes p} C_p\big)
    \label{eq:R1argument5}
    \\
    &=
    p!\, 4K_\mfc \sum_{i=1}^\infty
    \big\langle \mathfrak{b}_{p,i}, R(0)^{\otimes p} \mathfrak{b}_{p,i} \big\rangle_{\clh_1^{\otimes p}}
    \label{eq:R1argument6}
    \\
    &\leq
    4K_\mfc \, \E \| G^p[X_1] \|^2_{\clh_2}.
    \label{eq:R1argument7}
\end{align}
Here \eqref{eq:R1argument3} follows because multiplication by the two norm-one block projections does not increase the trace-class norm. Since $\mathcal R^{(p)}(v)$ is nonnegative, the operator $(C_p^* \oplus C_p^*) \mathcal R^{(p)}(v) (C_p \oplus C_p)$ is self-adjoint and nonnegative, resulting in \eqref{eq:R1argument4}. Moreover, the trace of a nonnegative block operator is the sum of the traces of its diagonal blocks, and the number of integers $v$ satisfying $|v|\le K_{\mathfrak{C}}-1$ is at most $2K_{\mathfrak{C}}$, which gives \eqref{eq:R1argument5}. Finally, \eqref{eq:R1argument7} can be seen by applying Cauchy-Schwarz as follows
\begin{equation}
\begin{split}
    \left| \big\langle \mathfrak{b}_{p,i}, R(0)^{\otimes p} \mathfrak{b}_{p,i} \big\rangle_{\clh_1^{\otimes p}} \right|  &= \left|
    \sum_{\bfr,\bfs \in \NN^p}
    b_{i,\bfr} b_{i,\bfs}
    \prod_{k=1}^p \rho_{r_k s_k}(0)
    \right|
    \\
    &=
    \left|
    \E \left\langle
    \sum_{\bfr \in \NN^p}
    b_{i,\bfr}
    (\varepsilon_{r_1 0} \otimes \dots \otimes \varepsilon_{r_p 0}),
    \sum_{\bfs \in \NN^p}
    b_{i,\bfs}
    (\varepsilon_{s_1 0} \otimes \dots \otimes \varepsilon_{s_p 0})
    \right\rangle_{\mfh^{\otimes p}}
    \right|
    \\
    &\le
    \E\left\|
    \sum_{\bfr \in \NN^p}
    b_{i,\bfr}
    (\varepsilon_{r_1 0} \otimes \dots \otimes \varepsilon_{r_p 0})
    \right\|_{\mfh^{\otimes p}}^2
    =
    \sum_{\bfr \in \NN^p} b_{i,\bfr}^2
    =
    \|\mathfrak{b}_{p,i}\|_{\clh_1^{\otimes p}}^2.
\end{split}
\end{equation}

For the second summand in \eqref{al:operator-notation-separate-sum}, we distinguish two cases corresponding to Assumptions \ref{thm:ass-2.2} and \ref{thm:ass-2.1}.
\textit{Case 1:} We write
\begin{align}
    p! \sum_{|v| \geq K_1}
    \left\|
    C_p^* R(v)^{\otimes p} C_p
    \right\|_{\mathcal{S}_1(\clh_2)}
    &\leq
    p! \sum_{|v| \geq K_1}
    \| R(v)^{\otimes p} C_p \|_{\mathcal{S}_2(\clh_2,\clh_1^{\otimes p})}
    \| C_p \|_{\mathcal{S}_2(\clh_2,\clh_1^{\otimes p})}
    \label{eq:R2argument4}
    \\
    &\leq
    p! \sum_{|v| \geq K_1}
    \| R(v)^{\otimes p} \|_{\op(\clh_1^{\otimes p})}
    \| C_p \|_{\mathcal{S}_2(\clh_2,\clh_1^{\otimes p})}^2
    \label{eq:R2argument4-1}
    \\
    &\leq
    p! \sum_{i=1}^\infty \| \mathfrak{b}_{p,i} \|_{\clh_1^{\otimes p}}^2
    \sum_{|v| \ge K_1}
    \| R(v) \|_{\op(\clh_1)}^{p}
    \label{eq:R2argument5}
    \\
    &=
    \E \| G^p[X_1] \|^2_{\clh_2}
    \sum_{|v| \ge K_1}
    \| R(v) \|_{\op(\clh_1)}^{p}
    \label{eq:R2argument6}
    \\
    &\le
    \E \| G^p[X_1] \|^2_{\clh_2}
    \sum_{v \in \ZZ} \| R(v) \|_{\op(\clh_1)}^{q},
    \label{eq:R2argument7}
\end{align}
where \eqref{eq:R2argument4} is due to Problem 28(c) (Section 6) in \cite{reed1972methods}. The inequality \eqref{eq:R2argument4-1} can be inferred from Proposition 18.6 (d) in \cite{conway2000course} and \eqref{eq:R2argument5} is due to the definition of the operator $C_p$ in \eqref{eq:def-Cp-operator} as well as 
$\|R(v)^{\otimes p}\|_{\op(\clh_1^{\otimes p})}
\le
\|R(v)\|_{\op(\clh_1)}^{p}$. In \eqref{eq:R2argument6}, we use the relation \eqref{eq:finite-secon-moment-b}. Finally \eqref{eq:R2argument7} follows from the definition of $K_1$ given in \eqref{eq:de:theta_theta-1} and $q \le p$.

\textit{Case 2:}
For the second scenario, recalling $Q^{-1/2}_p$ from \eqref{eq:def-Qp} and Remark \ref{re:remark-on-operator-identity-R-Gamma}, we get
\begin{align}
    p! \sum_{|v| \geq K_2}
    \left\|
    C_p^* R(v)^{\otimes p} C_p
    \right\|_{\mathcal{S}_1(\clh_2)}
    &=
    p! \sum_{|v| \geq K_2}
    \left\|
    (Q^{-1/2}_p C_p)^* \Gamma(v)^{\otimes p} (Q^{-1/2}_p C_p)
    \right\|_{\mathcal{S}_1(\clh_2)}
    \nonumber
    \\
    &\leq
    p!\,\|Q^{-1/2}_p C_p\|_{\op(\clh_2,\clh_1^{\otimes p})}^2
    \sum_{|v| \ge K_2}
    \| \Gamma(v)^{\otimes p} \|_{\mathcal{S}_1(\clh_1^{\otimes p})}
    \label{eq:R2argument9}
    \\
    &\leq
    p!\,\|Q^{-1/2}_p C_p\|_{\op(\clh_2,\clh_1^{\otimes p})}^2
    \sum_{v \in \ZZ}
    \| \Gamma(v) \|_{\mathcal{S}_1(\clh_1)}^{q},
    \label{eq:R2argument10}
\end{align}
where \eqref{eq:R2argument9} follows by Theorem 18.11 (g) in \cite{conway2000course}. In \eqref{eq:R2argument9}, we use the definition of $K_2$ given in \eqref{eq:de:theta_theta-2}.

Then, by combining \eqref{eq:R1argument7} with \eqref{eq:R2argument7} in Case 1, or with
\eqref{eq:R2argument10} in Case 2, we obtain
\begin{equation} \label{eq:cond-1-bound-cont00}
    \| \clt_{Z_p} \|_{\mathcal{S}_1(\clh_2)}
    \le  \left( 
    4 K_1 \E \| G^p[X_1] \|^2_{\clh_2} + M_{1,p}   \right) \mathds{1}_{\{\mathfrak{C}=1\}} +\left( 
    4 K_2 \E \| G^p[X_1] \|^2_{\clh_2} + M_{2,p}  \right) \mathds{1}_{\{\mathfrak{C}=2\}}
    < \infty,
\end{equation}
where the constants $K_\mfc$, $\mfc=1,2$, are defined in
\eqref{eq:de:theta_theta-1} and \eqref{eq:de:theta_theta-2}, and
\begin{equation} \label{eq:def-Mi-theta-i}
    M_{\mfc,p} =
    \begin{cases}
        \E \| G^p[X_1] \|^2_{\clh_2}
        \displaystyle\sum_{v \in \ZZ } \vartheta_1(v)^{q},
        & \mathfrak{C} = 1, \\[0.4cm]
        p!\,\| Q^{-1/2}_p C_p \|^2_{\op(\clh_2,\clh_1^{\otimes p})}
        \displaystyle\sum_{v \in \ZZ }
        \vartheta_2(v)^{q},
        & \mfc = 2
    \end{cases}, 
    \quad \vartheta_\mfc(v) = \begin{cases}
        \| R(v) \|_{\op(\clh_1)} & \mfc = 1, \\
         \| \Gamma(v) \|_{\mathcal{S}_1(\clh_1)} & \mfc = 2.
    \end{cases}
\end{equation}

Turning back to the two summands in \eqref{eq:twosummands-condition-1}, we show separately that they converge to zero. For the first one, assuming that $n \ge K_\mfc$,
\begin{multline}
        \Bigg\|
    \sum_{i,j =1}^\infty
    p! \sum_{|v| \geq n}
    \sum_{\bfr,\bfs \in \NN^p}
    b_{i,\bfr} b_{j,\bfs}
    \prod_{k=1}^p \rho_{r_k s_k}(v)\,
    (v_i \otimes v_j)
   \Bigg\|_{\mathcal{S}_1(\clh_2)}
    = \left\| \sum_{|v| \ge n} p! C_p^* R(v)^{\otimes p} C_p \right\|_{\mathcal{S}_1(\clh_2)} \\
    \le \E \| G^p[X_1] \|^2_{\clh_2}
    \sum_{|v| \geq n} \| R(v) \|_{\op(\clh_1)}^{q} \mathds{1}_{\{\mfc=1\}} +     p!\,\|Q^{-1/2}_p C_p\|_{\op(\clh_2,\clh_1^{\otimes p})}^2
    \sum_{|v| \geq n}
    \| \Gamma(v) \|_{\mathcal{S}_1(\clh_1)}^{q} \mathds{1}_{\{\mfc=2\}}, \label{eq:cond-1-first0-term}
\end{multline}
where for the first line we used \eqref{eq:rewrite-as-operators} and for the second line we have used calculations identical to \eqref{eq:R2argument7} and \eqref{eq:R2argument10} for Assumptions \ref{thm:ass-2.2} and \ref{thm:ass-2.1} respectively. The convergence to zero follows by the summability of the series. 

For the second summand in \eqref{eq:twosummands-condition-1}, note that
\begin{align}
    &
    \Bigg\|
    \sum_{i,j =1}^\infty
    p! \sum_{|v| < n} \frac{|v|}{n}
    \sum_{\bfr,\bfs \in \NN^p}
    b_{i,\bfr} b_{j,\bfs}
    \prod_{k=1}^p \rho_{r_k s_k}(v)\,
    (v_i \otimes v_j)
    \Bigg\|_{\mathcal{S}_1(\clh_2)}
    \nonumber
    \\
    &\leq
    p!
    \sum_{|v| \leq K_\mfc-1}
    \frac{|v|}{n}
    \left\|
    C_p^* R(v)^{\otimes p} C_p
    \right\|_{\mathcal{S}_1(\clh_2)}
    +
    p!
    \sum_{n \geq |v| \geq K_\mfc}
    \frac{|v|}{n}
    \left\|
    C_p^* R(v)^{\otimes p} C_p
    \right\|_{\mathcal{S}_1(\clh_2)}
    \nonumber
    \\
    &\leq
    4  K_\mfc^2 \frac{1}{n}
    \E \| G^p[X_1] \|^2_{\clh_2}
    +
    \E \| G^p[X_1] \|^2_{\clh_2}
    \sum_{ n \ge |v| \ge K_1}
    \frac{|v|}{n}\,
    \| R(v) \|_{\op(\clh_1)}^{q}\,
    \mathds{1}_{\{\mfc=1\}}
    \nonumber
    \\
    &\hspace{2cm}+
    p!\,\| Q^{-1/2}_p C_p \|^2_{\op(\clh_2,\clh_1^{\otimes p})}
    \sum_{n \ge |v| \ge K_2}
    \frac{|v|}{n}\,
    \| \Gamma(v) \|_{\mathcal{S}_1(\clh_1)}^{q}\,
    \mathds{1}_{\{\mfc=2\}}
    \;\longrightarrow\; 0. \label{eq:cond-1-second0-term}
\end{align}
Indeed, the first term is of order $n^{-1}$, while in the remaining term we use dominated convergence theorem.

\noindent
\textit{Condition \ref{item1.1:thm-infinite-expansions-trace-sec-mom}:}
\MD{By condition \ref{item1:thm-infinite-expansions-covar-oper}, each $\clt_{Z_p}$ is the trace-class limit of nonnegative, self-adjoint covariance operators and hence is itself nonnegative and self-adjoint. To define $\clt_Z\doteq\sum_{p=q}^\infty\clt_{Z_p}$ and verify condition \ref{item1.1:thm-infinite-expansions-trace-sec-mom}, it remains to establish summability in trace-class norm. The following estimate gives this summability:}
\begin{align}
    &\| \clt_Z \|_{\mathcal{S}_1(\clh_2)}
    \le
    \sum_{p=q}^\infty
    \| \clt_{Z_p} \|_{\mathcal{S}_1(\clh_2)} \\
    &\le \left( 4 K_1 \E \| G[X_1] \|^2_{\clh_2} +  \E \| G[X_1] \|^2_{\clh_2} \sum_{v \in \ZZ} \|R(v)\|^q_{\operatorname{op}(\clh_1)} \right) \mathds{1}_{\{\mfc=1\}} \\
    &\quad+ \left( 
    4 K_2 \E \| G[X_1] \|^2_{\clh_2} +  \displaystyle\sum_{v \in \ZZ }
        \| \Gamma(v) \|_{\mathcal{S}_1(\clh_1)}^{q} \sum_{p=q}^\infty p!\,\| Q^{-1/2}_p C_p \|^2_{\op(\clh_2,\clh_1^{\otimes p})}  \right) \mathds{1}_{\{\mfc=2\}}
    < \infty,
    \label{eq:trace-class-total-lim-cov}
\end{align}
by \eqref{eq:cond-1-bound-cont00} and \eqref{eq:finite-secon-moment-b}. \MD{Thus the series defining $\clt_Z$ converges in trace-class norm to a nonnegative, self-adjoint operator, which is the covariance of a centered Gaussian element of $\clh_2$. Its identification with \eqref{eq:def-T} is given in \eqref{eq:cov-Z-repres} below. This verifies condition \ref{item1.1:thm-infinite-expansions-trace-sec-mom}.} 

\noindent
\textit{Condition \ref{item:cond2}:}
Recall from \eqref{eq:Hilbert_expansion_Chaos_decomposition} that
$h_{n,p}$ is written in terms of $h_{n,p,i}$ in \eqref{eq:hpnitilde}.
Then, for $l=1,\dots,p-1$, we have the following identity for the contractions:
\begin{align}
&
    h_{n,p,i} \otimes_l h_{n,p,i}
    \nonumber
\\
&=
    \frac{1}{n}
    \sum_{k_1,k_2=1}^n
    \sum_{\bfr,\bfs \in \NN^p}
    b_{i,\bfr} b_{i,\bfs}
    \big(
    \varepsilon_{r_1 k_1} \otimes \dots \otimes \varepsilon_{r_p k_1}
    \big)
    \otimes_l
    \big(
    \varepsilon_{s_1 k_2} \otimes \dots \otimes \varepsilon_{s_p k_2}
    \big)
    \nonumber
\\
&=
    \frac{1}{n}
    \sum_{k_1,k_2=1}^n
    \sum_{\bfr,\bfs \in \NN^p}
    b_{i,\bfr} b_{i,\bfs}
    \prod_{m=1}^{l}
    \langle \varepsilon_{r_m k_1},
    \varepsilon_{s_m k_2} \rangle_{\mfh}
    \nonumber
\\
&\hspace{1.5cm}
    \times
    \big(
    \varepsilon_{r_{l+1} k_1} \otimes \dots \otimes \varepsilon_{r_p k_1}
    \otimes
    \varepsilon_{s_{l+1} k_2} \otimes \dots \otimes \varepsilon_{s_p k_2}
    \big)
    \nonumber
\\
&=
    \frac{1}{n}
    \sum_{k_1,k_2=1}^n
    \sum_{\bfr,\bfs \in \NN^p}
    b_{i,\bfr} b_{i,\bfs}
    \prod_{m=1}^{l}
    \rho_{r_m s_m}(k_1-k_2)
    \nonumber
\\
&\hspace{1.5cm}
    \times
    \big(
    \varepsilon_{r_{l+1} k_1} \otimes \dots \otimes \varepsilon_{r_p k_1}
    \otimes
    \varepsilon_{s_{l+1} k_2} \otimes \dots \otimes \varepsilon_{s_p k_2}
    \big).
    \label{eq:contr-cond-4-2}
\end{align}
Taking the norm in $\mfh^{\otimes 2(p-l)}$ and denoting, for $a=1,2$, the $m$-th elements of
$\bfr_a \in \NN^p$ and $\bfs_a \in \NN^p$ by $r_{a,m}$ and $s_{a,m}$, respectively, gives
\begin{align}
&
    \| h_{n,p,i} \otimes_l h_{n,p,i} \|_{\mfh^{\otimes 2(p-l)}}^2
    \nonumber
\\
&=
    \frac{1}{n^2}
    \sum_{k_1,k_2,k_3,k_4=1}^n
    \sum_{\substack{\bfr_1,\bfr_2 \in \NN^p\\
                    \bfs_1,\bfs_2 \in \NN^p}}
    b_{i,\bfr_1} b_{i,\bfs_1}
    b_{i,\bfr_2} b_{i,\bfs_2}
    \prod_{m=1}^{l}
    \rho_{r_{1,m} s_{1,m}}(k_1-k_2)
    \prod_{m=1}^{l}
    \rho_{r_{2,m} s_{2,m}}(k_3-k_4)
    \nonumber
\\
&\hspace{2.5cm}
    \times
    \prod_{m=l+1}^{p}
    \rho_{r_{1,m} r_{2,m}}(k_1-k_3)
    \prod_{m=l+1}^{p}
    \rho_{s_{1,m} s_{2,m}}(k_2-k_4).
    \label{eq:contr-cond-4-3a}
\end{align}

Introduce, for each $i\ge1$ and $l=1,\dots,p-1$, the
operator $B_{p,l}^{(i)}
    :
    \clh_1^{\otimes l}
    \to
    \clh_1^{\otimes(p-l)}$, which, for $    \boldsymbol{\alpha}
    =
    (\alpha_1,\dots,\alpha_l)\in\NN^l,
    \boldsymbol{\beta}
    =
    (\beta_1,\dots,\beta_{p-l})\in\NN^{p-l}$, acts on the tensor basis elements by
\[
    B_{p,l}^{(i)}
    \big(
    e_{\alpha_1} \otimes \dots \otimes e_{\alpha_l}
    \big)
    \doteq 
    \sum_{\boldsymbol{\beta}\in\NN^{p-l}}
    b_{i,(\boldsymbol{\alpha},\boldsymbol{\beta})}
    (e_{\beta_1} \otimes \dots \otimes e_{\beta_{p-l}}).
\]
Here \((\boldsymbol{\alpha},\boldsymbol{\beta})\in\NN^p\) denotes
concatenation. Define also
\[
    M_{i,l}(v)
    \doteq
    B_{p,l}^{(i)}
    R(v)^{\otimes l}
    \big(B_{p,l}^{(i)}\big)^*
    \in
    \mathcal L\big(\clh_1^{\otimes(p-l)}\big).
\]
We also write $e_{\boldsymbol{\beta}} \doteq e_{\beta_1} \otimes \dots \otimes e_{\beta_{p-l}}$ for $\boldsymbol{\beta} \in \NN^{p-l}$.
Then, \eqref{eq:contr-cond-4-3a} can be rewritten as
\begin{align}
&
    \| h_{n,p,i} \otimes_l h_{n,p,i} \|_{\mfh^{\otimes 2(p-l)}}^2
    \nonumber
\\
&=
    \frac1{n^2}
    \sum_{k_1,k_2,k_3,k_4=1}^n
    \sum_{\boldsymbol{\beta}_1,\boldsymbol{\delta}_1, \boldsymbol{\beta}_2,\boldsymbol{\delta}_2\in\NN^{p-l}}
    \left\langle
    e_{\boldsymbol{\beta}_1},
    M_{i,l}(k_1-k_2)e_{\boldsymbol{\delta}_1}
    \right\rangle_{\clh_1^{\otimes(p-l)}}
    \left\langle
    e_{\boldsymbol{\delta}_1},
    R(k_2-k_4)^{\otimes(p-l)}
    e_{\boldsymbol{\delta}_2}
    \right\rangle_{\clh_1^{\otimes(p-l)}}
    \nonumber
\\
&\hspace{1.5cm}
    \times
    \left\langle
    e_{\boldsymbol{\delta}_2},
    M_{i,l}(k_3-k_4)^*
    e_{\boldsymbol{\beta}_2}
    \right\rangle_{\clh_1^{\otimes(p-l)}}
    \left\langle
    e_{\boldsymbol{\beta}_2},
    R(k_1-k_3)^{\otimes(p-l)}
    e_{\boldsymbol{\beta}_1}
    \right\rangle_{\clh_1^{\otimes(p-l)}}
    \nonumber
\\
&=
    \frac1{n^2}
    \sum_{k_1,k_2,k_3,k_4=1}^n
    \sum_{\boldsymbol{\beta}_1\in\NN^{p-l}}
    \left\langle
        e_{\boldsymbol{\beta}_1},
        M_{i,l}(k_1-k_2)
        R(k_2-k_4)^{\otimes(p-l)}
    \right.
    \nonumber
\\
&\hspace{5.2cm}
    \left.
        \times
        M_{i,l}(k_3-k_4)^*
        R(k_1-k_3)^{\otimes(p-l)}
        e_{\boldsymbol{\beta}_1}
    \right\rangle_{\clh_1^{\otimes(p-l)}}
    \nonumber
\\
&=
    \frac1{n^2}
    \sum_{k_1,k_2,k_3,k_4=1}^n
    \tr_{\clh_1^{\otimes(p-l)}}\!\left[
    M_{i,l}(k_1-k_2)
    R(k_2-k_4)^{\otimes(p-l)}
    M_{i,l}(k_3-k_4)^*
    R(k_1-k_3)^{\otimes(p-l)}
    \right].
    \label{eq:contr-cond-4-3b}
\end{align}
We distinguish again the two cases corresponding to Assumptions
\ref{thm:ass-2.2} and \ref{thm:ass-2.1}.

\textit{Case 1:}
Recall the function $\vartheta_1(v)$ from \eqref{eq:def-Mi-theta-i}.
Then, \eqref{eq:contr-cond-4-3b} yields
\begin{align}
&
\tr_{\clh_1^{\otimes(p-l)}}\!\left[
M_{i,l}(k_1-k_2)
R(k_2-k_4)^{\otimes(p-l)}
M_{i,l}(k_3-k_4)^*
R(k_1-k_3)^{\otimes(p-l)}
\right]
\nonumber
\\
&\le
\left\|
M_{i,l}(k_1-k_2)
R(k_2-k_4)^{\otimes(p-l)}
\right\|_{\mathcal S_2(\clh_1^{\otimes(p-l)})}
\nonumber
\\
&\hspace{1.5cm}
\times
\left\|
M_{i,l}(k_3-k_4)^*
R(k_1-k_3)^{\otimes(p-l)}
\right\|_{\mathcal S_2(\clh_1^{\otimes(p-l)})}
\label{al:Cond3-al1}
\\
&\le
\|M_{i,l}(k_1-k_2)\|_{\mathcal S_2(\clh_1^{\otimes(p-l)})}
\|M_{i,l}(k_3-k_4)\|_{\mathcal S_2(\clh_1^{\otimes(p-l)})}
\nonumber
\\
&\hspace{1.5cm}
\times
\|R(k_2-k_4)\|_{\op(\clh_1)}^{p-l}
\|R(k_1-k_3)\|_{\op(\clh_1)}^{p-l}
\label{al:Cond3-al2}
\\
&\le
\|B_{p,l}^{(i)}\|_{\mathcal S_2(\clh_1^{\otimes l},\clh_1^{\otimes(p-l)})}^{4}
\|R(k_1-k_2)\|_{\op(\clh_1)}^{l}
\|R(k_3-k_4)\|_{\op(\clh_1)}^{l}
\nonumber
\\
&\hspace{1.5cm}
\times
\|R(k_1-k_3)\|_{\op(\clh_1)}^{p-l}
\|R(k_2-k_4)\|_{\op(\clh_1)}^{p-l}
\label{al:Cond3-al3}
\\
&\le
\|\mathfrak{b}_{p,i}\|_{\clh_1^{\otimes p}}^{4}
\vartheta_1(k_1-k_2)^l
\vartheta_1(k_2-k_4)^{p-l}
\Big(
\vartheta_1(k_3-k_4)^p
+
\vartheta_1(k_1-k_3)^p
\Big).
\label{al:Cond3-al5}
\end{align}
Here \eqref{al:Cond3-al1} follows by 
Problem 28(c) (Section 6) in \cite{reed1972methods}, \eqref{al:Cond3-al2} and \eqref{al:Cond3-al3} are due to
the ideal property of Schatten classes and from the definition of
$M_{i,l}(v)$.
The relation \eqref{al:Cond3-al5} is due to
\[
\|B_{p,l}^{(i)}\|_{\mathcal S_2(\clh_1^{\otimes l},\clh_1^{\otimes(p-l)})}^2
=
\sum_{\boldsymbol{\alpha}\in\NN^l}
\sum_{\boldsymbol{\beta}\in\NN^{p-l}}
|b_{i,(\boldsymbol{\alpha},\boldsymbol{\beta})}|^2
=
\sum_{\bfr\in\NN^p}|b_{i,\bfr}|^2
=
\|\mathfrak{b}_{p,i}\|_{\clh_1^{\otimes p}}^2.
\]
Finally, \eqref{al:Cond3-al5} uses Young's inequality as follows
\begin{equation} \label{eq:theta-1-young}
    \vartheta_1(k_3-k_4)^l
\vartheta_1(k_1-k_3)^{p-l}
=
\Big(\vartheta_1(k_3-k_4)^p\Big)^{l/p}
\Big(\vartheta_1(k_1-k_3)^p\Big)^{(p-l)/p}
\le
\vartheta_1(k_3-k_4)^p
+
\vartheta_1(k_1-k_3)^p.
\end{equation}
Therefore,
\begin{equation} \label{eq:cond-iii-bound-1}
    \|h_{n,p,i}\otimes_l h_{n,p,i}\|_{\mfh^{\otimes 2(p-l)}}^2
\le
2\|\mathfrak{b}_{p,i}\|_{\clh_1^{\otimes p}}^4
\frac1n
\Big(\sum_{v\in\ZZ}\|R(v)\|_{\op(\clh_1)}^p\Big)
\Big(\sum_{|v|\le n}\|R(v)\|_{\op(\clh_1)}^l\Big)
\Big(\sum_{|v|\le n}\|R(v)\|_{\op(\clh_1)}^{p-l}\Big).
\end{equation}

\textit{Case 2:}
Under Assumption \ref{thm:ass-2.1} recall $\vartheta_2(v)$ from
\eqref{eq:def-Mi-theta-i}. Recalling also the decomposition in
\eqref{eq:def-Qp}, we first introduce the operators $A_{i,l} :\clh_1^{\otimes l}
    \to
    \clh_1^{\otimes(p-l)}$ and $\widetilde M_{i,l}(v)$ given by 
\begin{equation}
        A_{i,l}
    \doteq
    Q_{p-l}^{-1/2}
    B_{p,l}^{(i)}
    Q_l^{-1/2}, \quad    \widetilde M_{i,l}(v)
    \doteq
    Q_{p-l}^{-1/2}
    M_{i,l}(v)
    Q_{p-l}^{-1/2}.
\end{equation}
From the
definition of $B_{p,l}^{(i)}$, we have
\begin{equation} \label{eq:Ai-l-bound}
\begin{aligned}
\|A_{i,l}\|_{\op(\clh_1^{\otimes l},
    \clh_1^{\otimes(p-l)})}
    \le
    \|A_{i,l}\|_{\mathcal S_2(\clh_1^{\otimes l},
    \clh_1^{\otimes(p-l)})}
    =
    \|Q_p^{-1/2}\mathfrak b_{p,i}\|_{\clh_1^{\otimes p}}\\
    =
    \|Q_p^{-1/2} C_p v_i\|_{\clh_1^{\otimes p}}
    \le
    \|Q_p^{-1/2} C_p\|_{\op(\clh_2,\clh_1^{\otimes p})} < \infty,    
\end{aligned}
\end{equation}
where the finiteness at the end follows by Assumption \ref{thm:ass-2.1}.
Since, by \eqref{eq:def-Qp}, $R(v)^{\otimes l}
    =
    Q_l^{-1/2}
    \Gamma(v)^{\otimes l}
    Q_l^{-1/2}$,
    we obtain $\widetilde M_{i,l}(v)
    =
    A_{i,l}
    \Gamma(v)^{\otimes l}
    A_{i,l}^* $. Hence, 
\begin{align}
    \|\widetilde M_{i,l}(v)\|_{\op(\clh_1^{\otimes(p-l)})}
    &= \| A_{i,l}
    \Gamma(v)^{\otimes l}
    A_{i,l}^* \|_{\op(\clh_1^{\otimes(p-l)})}  \\
    &\le
    \|A_{i,l}\|_{\op(\clh_1)}^2
    \|\Gamma(v)^{\otimes l}\|_{\op(\clh_1^{\otimes l})}
    \nonumber \\
    &\le
    \|A_{i,l}\|_{\op(\clh_1)}^2
    \|\Gamma(v)\|_{\mathcal S_1(\clh_1)}^l .
    \label{eq:Mtilde-case2-bound}
\end{align}

We return to the trace term in \eqref{eq:contr-cond-4-3b}. We estimate it as follows
\begin{align}
&
\tr_{\clh_1^{\otimes(p-l)}}\!\left[
M_{i,l}(k_1-k_2)
R(k_2-k_4)^{\otimes(p-l)}
M_{i,l}(k_3-k_4)^*
R(k_1-k_3)^{\otimes(p-l)}
\right]
\nonumber
\\
&=
\tr_{\clh_1^{\otimes(p-l)}}\!\left[
\widetilde M_{i,l}(k_1-k_2)
\Gamma(k_2-k_4)^{\otimes(p-l)}
\widetilde M_{i,l}(k_3-k_4)^*
\Gamma(k_1-k_3)^{\otimes(p-l)}
\right] \label{eq:8484} \\
&\le
\|\widetilde M_{i,l}(k_1-k_2)\|_{\op(\clh_1^{\otimes(p-l)})}
\|\Gamma(k_2-k_4)^{\otimes(p-l)}\|_{\mathcal S_1(\clh_1^{\otimes (p-l)})} 
\\&\hspace{2cm}
\|\widetilde M_{i,l}(k_3-k_4)\|_{\op(\clh_1^{\otimes(p-l)})}
\|\Gamma(k_1-k_3)^{\otimes(p-l)}\|_{\mathcal S_1(\clh_1^{\otimes (p-l)})} \label{al:Cond3-al-Case2-3} \\
&\le \|Q_p^{-1/2} C_p\|_{\op(\clh_2,\clh_1^{\otimes p})}^{4}
\vartheta_2(k_1-k_2)^l
\vartheta_2(k_3-k_4)^l
\vartheta_2(k_2-k_4)^{p-l}
\vartheta_2(k_1-k_3)^{p-l}.
\label{eq:case2-trace-rewritten} \\
&\le \|Q_p^{-1/2} C_p\|_{\op(\clh_2,\clh_1^{\otimes p})}^{4}
\vartheta_2(k_1-k_2)^l
\vartheta_2(k_2-k_4)^{p-l}
\Big(
\vartheta_2(k_3-k_4)^p
+
\vartheta_2(k_1-k_3)^p
\Big).
\label{al:Cond3-al-Case2-5}
\end{align}
For \eqref{eq:8484}, we used
\eqref{eq:def-Qp}, together with the cyclicity of the trace; for \eqref{al:Cond3-al-Case2-3}, we used the ideal property of the trace-class; for \eqref{eq:case2-trace-rewritten}, we used that $    \|\Gamma(v)^{\otimes m}\|_{\mathcal S_1(\clh_1^{\otimes m})} = \|\Gamma(v)\|_{\mathcal S_1(\clh_1)}^m$. 
Finally, \eqref{al:Cond3-al-Case2-5} follows from Young's inequality as in \eqref{eq:theta-1-young}, by replacing $\vartheta_1$ with $\vartheta_2$.
    
Then, substituting \eqref{al:Cond3-al-Case2-5} into
\eqref{eq:contr-cond-4-3b}, we obtain
\begin{multline} \label{eq:cond-iii-ass-2}
    \|h_{n,p,i}\otimes_l h_{n,p,i}\|_{\mfh^{\otimes 2(p-l)}}^2
\\
\le
2
\|Q_p^{-1/2}C_p\|_{\op(\clh_2,\clh_1^{\otimes p})}^{4}
\frac1n
\Big(\sum_{v\in\ZZ}\|\Gamma(v)\|_{\mathcal S_1(\clh_1)}^p\Big)
\Big(\sum_{|v|\le n}\|\Gamma(v)\|_{\mathcal S_1(\clh_1)}^l\Big)
\Big(\sum_{|v|\le n}\|\Gamma(v)\|_{\mathcal S_1(\clh_1)}^{p-l}\Big).
\end{multline}

Combining \eqref{eq:cond-iii-bound-1} and \eqref{eq:cond-iii-ass-2}, it follows that for all $l = 1, \dots, p-1$ and all $p \ge 2$, 
\begin{align}
  \|h_{n,p,i}\otimes_l h_{n,p,i}\|_{\mfh^{\otimes 2(p-l)}}^2  
  &\le
2\|\mathfrak{b}_{p,i}\|_{\clh_1^{\otimes p}}^4
\Big(\sum_{v\in\ZZ}\|R(v)\|_{\op(\clh_1)}^p\Big)
\Big( n^{-1 + \frac{l}{p}} \sum_{|v|\le n}\|R(v)\|_{\op(\clh_1)}^l\Big) \\
&\hspace{3cm} \times
\Big(n^{-1 + \frac{p-l}{p}} \sum_{|v|\le n}\|R(v)\|_{\op(\clh_1)}^{p-l}\Big) \mathds{1}_{\{\mfc=1\}} \\
&\quad + 2 \|Q_p^{-1/2}C_p\|_{\op(\clh_2,\clh_1^{\otimes p})}^{4}
\Big(   \sum_{v\in\ZZ}\|\Gamma(v)\|_{\mathcal S_1(\clh_1)}^p\Big)
\Big( n^{-1 + \frac{l}{p}} \sum_{|v|\le n}\|\Gamma(v)\|_{\mathcal S_1(\clh_1)}^l\Big) \\
&\hspace{3cm} \times
\Big(n^{-1 + \frac{p-l}{p}}\sum_{|v|\le n}\|\Gamma(v)\|_{\mathcal S_1(\clh_1)}^{p-l}\Big) \mathds{1}_{\{\mfc=2\}}  \to 0 \label{eq:hnpi-hnpi} ,
\end{align}
as $n \to \infty$.

\textit{Condition \ref{item:cond1}:} 
Recall $\theta_{\mathfrak{C}}$ from \eqref{eq:de:theta_theta-1}. We can estimate
\begin{align}
\sum_{p=q}^\infty  \| \clt_{\mathcal{I}_p(h_{n,p})} \|_{\mathcal{S}_1(\mathcal{H}_2)} &= 
\sum_{p=q}^\infty   \sum_{i=1}^\infty p! \|h_{n,p,i}\|_{\mfh^{\otimes p}}^2 \nonumber \\
&\le 
\sum_{p=q}^\infty \left( \E \| G^p[X_1] \|^2_{\clh_2} 4 K_\mfc + 2 M_{\mfc,p} \right) 
\nonumber\\
&\le  
\E \|G[X_1]\|_{\clh_2}^2 ( 4 K_1 + \theta_1 ) \ind_{\{\mfc=1\}} \nonumber \\
&\quad+ 
\left( 4 K_2 \E \|G[X_1]\|_{\clh_2}^2 + 2 \theta_2 \sum_{p=q}^\infty p! \| Q_p^{-1/2} C_p \|^2_{\op(\clh_1^{\otimes p},\clh_2)}   \right)
\ind_{\{\mfc=2\}}. \label{eq:cond-1-bound-cont}
\end{align}

\MD{The per-chaos majorants in the first inequality of \eqref{eq:cond-1-bound-cont} are independent of $n$ and summable over $p$. Applying the same inequality only to orders $p>N$, for $N\ge q$, gives
\begin{equation}\label{eq:uniform-chaos-tail-revised}
\sup_{n\ge1}\sum_{p=N+1}^\infty\sum_{i=1}^\infty
p!\|h_{n,p,i}\|_{\mfh^{\otimes p}}^2
\le
\sum_{p=N+1}^\infty
\left(4K_\mfc\E\|G^p[X_1]\|_{\clh_2}^2+2M_{\mfc,p}\right)
\longrightarrow0.
\end{equation}
This is the uniform tail condition \ref{item:cond1}.}

Since conditions \ref{item1:thm-infinite-expansions-covar-oper}--\ref{item:cond1} are satisfied, 
we can infer that $S_n$ converges in distribution to a centered Gaussian random variable $Z$ with covariance operator $\sum_{p = q}^\infty \covop_{Z_p}$. 
We conclude that this coincides with the representation of $\covop_Z$ in \eqref{eq:def-T}. By Lemma  \ref{cor:Hilbert_expansion_Chaos_decomposition},
\begin{align}
  \sum_{p=q}^\infty \clt_{Z_p}\label{eq:cov-Z-repres} &= 
\sum_{p=q}^\infty
\sum_{i,j =1}^\infty
    p! \sum_{v \in \ZZ} \sum_{\bfr,\bfs \in \NN^p} b_{i,\bfr} \MD{b_{j,\bfs}} \MD{\prod_{a=1}^p \rho_{r_a s_a}(v)} (v_i \otimes v_j)
  \\&=   
    \lim_{n \to \infty } \sum_{p=q}^{\infty} 
    \E \left(  \sum_{i=1}^\infty 
    ( I_{p} \otimes \operatorname{Id}_{\clh_{2}} ) \left[ h_{n,p,i} \otimes v_i \right] 
    \otimes \sum_{i=1}^\infty 
    ( I_{p} \otimes \operatorname{Id}_{\clh_{2}} ) \left[ h_{n,p,i} \otimes v_i \right] \right) \nonumber
\\&= 
    \lim_{n \to \infty } \E \left( S_n \otimes S_n \right) 
=   
    \lim _{n \to \infty }
    \frac{1}{n} \sum_{k_{1},k_{2}=1}^{n}  \E (G[X_{k_{1}}]
    \otimes G[X_{k_{2}}] ) \nonumber
\\&= 
    \lim_{n \to \infty } 
    \sum_{k=1}^{n} \left( 1 - \frac{k}{n} \right) 
    \left( 
    \E (G[X_{1}] \otimes G[X_{k+1}]) + \E (G[X_{k+1}] \otimes G[X_{1}]) \right) + \E (G[X_{1}] \otimes G[X_{1}]) \nonumber \\
    &= \covop_Z. \nonumber
\end{align}
In the computations above, the interchange of series and limits in $n$ is justified through conditions \ref{item1:thm-infinite-expansions-covar-oper}--\ref{item:cond1}.

\subsection{Proof of Theorem \ref{th:cont_case}}

First, note that
\begin{equation*}
    V_n(t) 
    = \frac{1}{\sqrt{n}} \sum_{k=1}^{\lfloor nt \rfloor} G[X_k] 
    = \frac{1}{\sqrt{n}} \sum_{k=1}^n \1_{\big[\frac{k}{n},1 \big]}(t) G[X_k].
\end{equation*}
Remark that the quantity inside the sum depends on $n$, and so Theorem \ref{thm:main} is not readily available. Using Lemma \ref{cor:Hilbert_expansion_Chaos_decomposition}, we can infer that, in law,
\begin{equation} \label{eq:V_n_chaoes}
V_n(t)
=
\sum_{p=q}^\infty 
\mathcal{I}_p
\left[  h_{n,p,t} \right]
=
\sum_{p=q}^\infty 
(I_p \otimes \operatorname{Id}_{\clh_2})
\left[  h_{n,p,t} \right]
\hspace{0.2cm}
\text{ with }
\hspace{0.2cm}
h_{n,p,t} \doteq \sum_{i=1}^\infty \left( \widehat h_{n,p,i,t} \otimes v_i \right),
\end{equation}
where
\begin{equation} \label{eq:hathpnt}
    \widehat h_{n,p,i,t} 
    \doteq
    \frac{1}{\sqrt{n}} \sum_{k=1}^n \sum_{\bfl \in \NN^p} \1_{\big[\frac{k}{n},1\big]}(t) b_{i,\bfl}
    \left(
    \varepsilon_{l_1 k} \otimes \dots \otimes \varepsilon_{l_p k} \right)
\end{equation}
for coefficients $\{b_{i,\bfl}\}_{i \in \NN, \bfl \in \NN^p}$ as in Lemma \ref{cor:Hilbert_expansion_Chaos_decomposition}. Moreover, recall $h_{n,p,i}$ defined in \eqref{eq:hpnitilde}.
For $k = 1,\dots, n$, we view $\1_{\big[\frac{k}{n},1\big]}(\cdot)$ as an element of $L^2([0,1])$.

We start with some preliminary estimates. First, recalling the non-negative, self-adjoint operator $\clt_B$ from \eqref{eq:BrownianMotion_int}, we have
\begin{equation} \label{eq:covopB-ineq}
    \|\clt_B \|_{\mathcal{S}_1(L^2([0,1]))} = \tr(\clt_B) = \int_{[0,1]} \kappa(t,t) dt  = 1/2.
\end{equation}
Moreover, consider two Hilbert spaces $\clh_1,\clh_2$ with respective bases $\{e_j\}_{j \in \NN}, \{\widetilde{e}_j\}_{j \in \NN}$, and operators $\mathcal{Q}_i \in \mathcal{S}_1(\clh_i), i = 1,2$. Then, the operator $\mathcal{Q} = \mathcal{Q}_1 \otimes \mathcal{Q}_2 \in \mathcal{S}_1(\clh_1 \otimes \clh_2)$ and $\| \mathcal{Q} \|_{\mathcal{S}_1(\clh_1 \otimes \clh_2)} = \| \mathcal{Q}_1 \|_{\mathcal{S}_1(\clh_1)} \| \mathcal{Q}_2 \|_{\mathcal{S}_1(\clh_2)}$. If, moreover, $Q_i \ge 0$, then  
\begin{equation} \label{eq:Q-decomp-hs}
    \| \mathcal{Q} \|_{\mathcal{S}_1(\clh_1 \otimes \clh_2)} = \| \mathcal{Q}_1 \|_{\mathcal{S}_1(\clh_1)} \| \mathcal{Q}_2 \|_{\mathcal{S}_1(\clh_2)}  = 
    \sum_{j_1,j_2 =1}^\infty \langle \mathcal{Q}_1[e_{j_1}], e_{j_1} \rangle_{\clh_1} \langle \mathcal{Q}_2[\widetilde{e}_{j_2}], \widetilde{e}_{j_2} \rangle_{\clh_2} = \tr(Q_1 \otimes Q_2). 
\end{equation}

Following the proof idea of Theorem \ref{thm:main}, we apply Theorem 3.4 in \cite{duker2025fourth}, that is, we verify conditions \ref{item1:thm-infinite-expansions-covar-oper}--\ref{item:cond1}.

\noindent
\textit{Condition \ref{item1:thm-infinite-expansions-covar-oper}:}
We first derive the covariance operator $  \clt_{\mathcal{I}_p(  h_{n,p, \cdot} )}$ of an individual chaos in terms of \eqref{eq:hpnitilde} as
\begin{align}
    \clt_{\mathcal{I}_p(  h_{n,p, \cdot} )} = 
    \sum_{i,j =1}^\infty
    p! \frac{1}{n} \sum_{k_1, k_2=1}^n 
    \left(\1_{\big[\frac{k_1}{n},1\big]} \otimes \1_{\big[\frac{k_2}{n},1\big]}\right)
    \sum_{\bfr,\bfs \in \NN^p} b_{i,\bfr} b_{j,\bfs} \prod_{k=1}^p \rho_{r_k s_k}(k_1 - k_2) (v_i \otimes v_j).
\end{align}
The candidate limiting covariance operator for a single chaos is then given by
$\clt_{W_p} = \clt_{B} \otimes \clt_{Z_p}$ with $\clt_{Z_p}$ as in \eqref{eq:traceoflimp}.
We need to control its distance to $  \clt_{\mathcal{I}_p(  h_{n,p, \cdot} )}$ in trace-class norm
\begin{align}
    &\|  \clt_{\mathcal{I}_p(  h_{n,p, \cdot} )}
    - \clt_{W_p}\|_{\mathcal{S}_1(L^2([0,1]) \otimes \clh_2)}
    \\&\leq
    \|   \clt_{\mathcal{I}_p(  h_{n,p, \cdot} )}
    - 
    \clt_{B} \otimes   \clt_{\mathcal{I}_p(  h_{n,p} )}
    \|_{\mathcal{S}_1(L^2([0,1]) \otimes \clh_2)}
    +
    \|
    \clt_{B} \otimes   \clt_{\mathcal{I}_p(  h_{n,p} )}
    -
    \clt_{B} \otimes \clt_{Z_p}\|_{\mathcal{S}_1(L^2([0,1]) \otimes \clh_2)}
    \\    &\leq
    \| \clt_{\mathcal{I}_p(h_{n,p,\cdot})} 
    - 
    \clt_{B} \otimes   \clt_{\mathcal{I}_p(  h_{n,p} )}
    \|_{\mathcal{S}_1(L^2([0,1]) \otimes \clh_2)}
    +
    \| \clt_{B} \|_{\mathcal{S}_1(L^2([0,1]))}
    \|
      \clt_{\mathcal{I}_p(  h_{n,p} )}
    -
    \clt_{Z_p}\|_{\mathcal{S}_1(\clh_2)},
    \label{eq:seq-eq1}
\end{align}
where the first summand in \eqref{eq:seq-eq1} goes to zero by Lemma \ref{le:ana_to_Le53Bou20} and the second summand by \eqref{eq:covopB-ineq} and since $\|
     \clt_{\mathcal{I}_p(  h_{n,p} )}
    -
    \clt_{Z_p}\|_{\mathcal{S}_1(\clh_2)} \to 0$ matches the verified condition \ref{item1:thm-infinite-expansions-covar-oper} in the proof of Theorem \ref{thm:main}.

\noindent
\textit{Condition \ref{item1.1:thm-infinite-expansions-trace-sec-mom}:} Letting $\clt_W \doteq  \sum_{r=q}^\infty \clt_{W_r}$, we have 
\begin{align}
    \| \covop_{W} \|_{\mathcal{S}_1(L^2([0,1]) \otimes \clh_2)}
    &=
    \| \covop_{B} \otimes \covop_{Z} \|_{\mathcal{S}_1(L^2([0,1]) \otimes \clh_2)}
    \leq
    \| \covop_{B} \|_{\mathcal{S}_1(L^2([0,1]))} \| \covop_{Z} \|_{\mathcal{S}_1(\clh_2)}
    < \infty, \label{eq:356-seq}
\end{align}
by \eqref{eq:covopB-ineq} and condition \ref{item1.1:thm-infinite-expansions-trace-sec-mom} in the proof of Theorem \ref{thm:main}. \MD{Since $\clt_B$ and $\clt_Z$ are nonnegative and self-adjoint, so is $\clt_W=\clt_B\otimes\clt_Z$. The trace-class estimate therefore verifies condition \ref{item1.1:thm-infinite-expansions-trace-sec-mom}.}

\noindent
\textit{Condition \ref{item:cond2}:} 
For $i\in \NN$ and $p \ge 2$,
    \begin{equation}
    \| \widehat h_{n,p,i,\cdot} \otimes_s \widehat h_{n,p,i, \cdot} \|_{\mathfrak{H}^{\otimes  2(p-s)}\otimes L^2([0,1])^{\otimes 2}}
    \leq
    \| h_{n,p,i} \otimes_s h_{n,p,i} \|_{\mathfrak{H}^{\otimes  2(p-s)}} \to 0,
    \label{al:lemma_cont_1_al2-seq} 
    \end{equation}
    for all $s=1,\dots,p-1$ and
where the inequality in \eqref{al:lemma_cont_1_al2-seq} follows since
\begin{equation} \label{eq:innerproductindicators-seq}
    \left\langle \1_{\big[\frac{k_1}{n},1\big]}, \1_{\big[\frac{k_2}{n},1\big]} \right\rangle_{L^2([0,1])} \leq 1,
\end{equation} 
and the convergence is due to condition \ref{item:cond2} in the proof of Theorem \ref{thm:main}.

\textit{Condition \ref{item:cond1}:} 
Note that
\begin{equation}
\lim_{N \to \infty} \sup_{n \ge 1}  \sum_{p=N+1}^\infty \sum_{i=1}^\infty p! \|\widehat{h}_{n,p,i,\cdot}\|_{\mfh^{\otimes p} \otimes L^2([0,1])}^2 
\leq
\lim_{N \to \infty} \sup_{n \ge 1}  \sum_{p=N+1}^\infty \sum_{i=1}^\infty p! \|h_{n,p,i}\|_{\mfh^{\otimes p}}^2
= 0.
\end{equation}
by \eqref{eq:innerproductindicators-seq} and the convergence follows by the proof of condition \ref{item:cond1} in the proof of Theorem \ref{thm:main}.
This concludes the proof.

\section{Quantitative version of Theorem \ref{thm:main}} \label{app:quant_version}

We state and prove here a quantitative version of Theorem \ref{thm:main}.
For that, we derive an explicit upper bound for the $d_2$ distance \eqref{eq:d2metric} between the partial sums \eqref{eq:partial_sum_G} and its limiting variable $Z$ derived in Theorem \ref{thm:main}.
To be more precise, we aim to derive an inequality of the form
\begin{align} \label{eq:Berry_quant}
    | \E(h[S_n]) - \E (h[Z]) |  \leq \mathcal{R}_n, \hspace{0.2cm} n \in \NN,
\end{align}
where $h \in \clc^2_b(\clh_2)$ with $\| h \|_{\clc^2_b(\clh_2)} \le 1$ and $\mathcal{R}_n \to 0$ as $n \to \infty$. An upper bound \eqref{eq:Berry_quant} quantifies the error that one makes when replacing the partial sum $S_n$ by its limiting variable $Z$.

Prior to stating our theorem, we recall the notations $\theta_\mfc, K_\mfc, \theta_\mfc(v), \mathfrak{b}_{p,i}$ and $M_{\mfc,p}$ from \eqref{eq:de:theta_theta-1}, \eqref{eq:de:theta_theta-2}, \eqref{eq:def-Cp-operator}, \eqref{eq:def-Mi-theta-i}. For $p\in\NN$, set
\begin{equation}\label{eq:remainder-terms-def-B}
    B_{\mfc,p}
    \doteq
    \mathds{1}_{\{p\ge q\}}
    \left(
    4K_\mfc \E \|G^p[X_1]\|_{\clh_2}^2
    + 2M_{\mfc,p}
    \right).
\end{equation}
Moreover, for $p\ge 2$ and $r=1,\dots,p-1$, define
\begin{align}
    \mathcal A_{n,p,r}^{(i)}
    &\doteq
    2\|\mathfrak{b}_{p,i}\|_{\clh_1^{\otimes p}}^4
    \Big(\sum_{v\in\ZZ}\|R(v)\|_{\op(\clh_1)}^p\Big)
    \Big( n^{-1 + \frac{r}{p}}
    \sum_{|v|\le n}\|R(v)\|_{\op(\clh_1)}^r\Big)
    \Big(n^{-1 + \frac{p-r}{p}}
    \sum_{|v|\le n}\|R(v)\|_{\op(\clh_1)}^{p-r}\Big)
    \mathds{1}_{\{\mfc=1\}}
    \nonumber \\
    &\quad+
    2\|Q_p^{-1/2}C_p\|_{\op(\clh_2,\clh_1^{\otimes p})}^{4}
    \Big(\sum_{v\in\ZZ}\|\Gamma(v)\|_{\mathcal S_1(\clh_1)}^p\Big)
    \Big(n^{-1 + \frac{r}{p}}
    \sum_{|v|\le n}\|\Gamma(v)\|_{\mathcal S_1(\clh_1)}^r\Big)
    \nonumber \\
    &\hspace{3cm}\times
    \Big(n^{-1 + \frac{p-r}{p}}
    \sum_{|v|\le n}\|\Gamma(v)\|_{\mathcal S_1(\clh_1)}^{p-r}\Big)
    \mathds{1}_{\{\mfc=2\}},
    \label{eq:remainder-terms-def-A}
\end{align}
and set $\mathcal A_{n,p,r}^{(i)}=0$ whenever $p<q$. For $r\ge q$ and $m,n\in\NN$, set
\begin{equation}\label{eq:explicit-coordinate-tail}
\begin{split}
    \mathfrak D_{n,r,m}
    &\doteq
    r!
    \sum_{|v|<n}
    \left(1-\frac{|v|}{n}\right)
    \sum_{j=m+1}^{\infty}
    \left|
    \left\langle
    \mathfrak b_{r,j},R(v)^{\otimes r}\mathfrak b_{r,j}
    \right\rangle_{\clh_1^{\otimes r}}
    \right|.
\end{split}
\end{equation}
Here $\mathfrak b_{r,j}=C_rv_j$, and the series defining
$\mathfrak D_{n,r,m}$ is finite for every fixed $n,r,m$ because
$C_r^*R(v)^{\otimes r}C_r$ is trace-class on $\clh_2$.

\MD{For $p\ge1$, define $c_p(r)$ for $1\le r\le p-1$, $c(a,b)$ for $1\le a<b$, and $c(a,b,\chi)$ for ordered pairs of distinct positive integers $a,b$ and $1\le\chi\le a\wedge b-1$, as follows:}
\begin{equation} \label{eq:remainder-terms-def-quant-constants}
\begin{split}
    c_{p}(r)
    &\doteq
    p(r-1)! \binom{p-1}{r-1}^2 \sqrt{(2p-2r)!},
    \qquad r=1,\dots,p-1,
    \\
    c(a,b)
    &\doteq
    a!^{2} \binom{b-1}{a-1}^{2} (b-a)!,
    \\
    c(a,b,\chi)
    &\doteq
    \frac{a^{2}}{2}
    (\chi-1)!^{2}
    \binom{a-1}{\chi-1}^{2}
    \binom{b-1}{\chi-1}^{2}
    (a+b-2\chi)! .
\end{split}
\end{equation}
Next define, for $p\ge 2$,
\begin{equation}\label{eq:remainder-terms-def-gamma-p}
    \widetilde{\gamma}_{n,i,j}^{(p)}
    \doteq
    \frac{1}{2}
    \sum_{\chi=1}^{p-1}
    c_p(\chi)^2
    \left(
    \mathcal A_{n,p,p-\chi}^{(i)}
    +
    \mathcal A_{n,p,p-\chi}^{(j)}
    \right),
\end{equation}
and, for $l_1\neq l_2$,
\begin{equation}\label{eq:remainder-terms-def-gamma-l1-l2}
\begin{aligned}
    \widetilde{\gamma}_{n,i,j}^{(l_1,l_2)}
    &\doteq
    c(l_1,l_2)
    \frac{B_{\mfc,l_1}}{l_1!}
    \left(\mathcal A_{n,l_2,l_2-l_1}^{(j)}\right)^{1/2}
    \mathds{1}_{\{l_1<l_2\}}
    \\
    &\quad+
    \MD{\frac{l_1^2}{l_2^2}}c(l_2,l_1)
    \frac{B_{\mfc,l_2}}{l_2!}
    \left(\mathcal A_{n,l_1,l_1-l_2}^{(i)}\right)^{1/2}
    \mathds{1}_{\{l_2<l_1\}}
    \\
    &\quad+
    \sum_{\chi=1}^{l_1\wedge l_2-1}
    \MD{c(l_1,l_2,\chi)}
    \left(
    \mathcal A_{n,l_1,l_1-\chi}^{(i)}
    +
    \mathcal A_{n,l_2,l_2-\chi}^{(j)}
    \right).
\end{aligned}
\end{equation}
The factor $l_1^2/l_2^2$ and the ordered coefficient $c(l_1,l_2,\chi)$ agree with equations (3.61)--(3.62) of \cite{duker2025fourth}. 


We define the five error quantities
\begin{equation} \label{eq:remainder-terms-def-quant-1}
\begin{split}
   \widetilde{\clr}_{1,M}
    &\doteq
    \sqrt{
    \sum_{p=M+1}^{\infty}
    \left(
    4K_\mfc \E \|G^p[X_1]\|_{\clh_2}^2
    +
    2M_{\mfc,p}
    \right)},
    \\
    \widetilde{\clr}_{2,n,M}
    &\doteq
   \frac{1}{2}\sum_{p=q}^M
    \Bigg\{
    \E \|G^p[X_1]\|_{\clh_2}^2
    \Bigg(
        \frac{4K_\mfc^2}{n}
        +
        \sum_{K_1\le |v|<n}
        \frac{|v|}{n}
        \theta_1(v)^q
        \mathds{1}_{\{\mfc=1\}}
        +
        \sum_{|v|\ge n}
        \theta_1(v)^q
        \mathds{1}_{\{\mfc=1\}}
    \Bigg)
    \nonumber \\
&\hspace{1.5cm}
    +
    p!\,\|Q_p^{-1/2}C_p\|_{\op(\clh_2,\clh_1^{\otimes p})}^2
    \Bigg(
        \sum_{K_2\le |v|<n}
        \frac{|v|}{n}
        \theta_2(v)^q
        \mathds{1}_{\{\mfc=2\}}
        +
        \sum_{|v|\ge n}
       \theta_2(v)^q
        \mathds{1}_{\{\mfc=2\}}
    \Bigg)
    \Bigg\}
    \\
    \widetilde{\clr}_{3,n,m,M}
    &\doteq
    \sqrt{m}\,M
    \sqrt{
    \sum_{i,j=1}^m
    \left(
    \sum_{p=2}^M
    \widetilde{\gamma}_{n,i,j}^{(p)}
    +
    \sum_{\substack{l_1,l_2=1\\ l_1\neq l_2}}^M
    \widetilde{\gamma}_{n,i,j}^{(l_1,l_2)}
    \right)
    },
    \\
    \widetilde{\clr}_{4,n,m,M}
    &\doteq
    \frac{M+2}{2}
    \sum_{r=q}^M
    \mathfrak D_{n,r,m},
    \\
    \widetilde{\clr}_{5,n,m,M}
    &\doteq
    \left( \sum_{p=q}^\infty B_{\mathfrak{C},p} \right)^{1/2}
    \left(
    (M+3)
    \sum_{r=q}^M
    \mathfrak D_{n,r,m}
    \right)^{1/2}.
\end{split}
\end{equation}

{Under the assumptions of Theorem \ref{thm:main},
$\widetilde{\clr}_{1,M}\to0$ as $M\to\infty$, and
$\widetilde{\clr}_{2,n,M},\widetilde{\clr}_{3,n,m,M}\to0$
as $n\to\infty$ for fixed $m,M$.
For the coordinate tails, the estimates used in condition \ref{item1:thm-infinite-expansions-covar-oper} give, for each fixed $r\ge q$,
\[
\sum_{v\in\ZZ}\|C_r^*R(v)^{\otimes r}C_r\|_{\mathcal S_1(\clh_2)}<\infty.
\]
Since the absolute diagonal sum of a trace-class operator is bounded by its trace norm, dominated convergence yields
\begin{equation}\label{eq:uniform-coordinate-tail-revised}
\sup_{n\ge1}\mathfrak D_{n,r,m}
\le r!\sum_{v\in\ZZ}\sum_{j=m+1}^\infty
\big|\langle C_r^*R(v)^{\otimes r}C_rv_j,v_j\rangle_{\clh_2}\big|
\longrightarrow0
\quad(m\to\infty).
\end{equation}
Thus $\widetilde{\clr}_{4,n,m,M}$ and $\widetilde{\clr}_{5,n,m,M}$ vanish as $m\to\infty$, uniformly in $n$, for every fixed $M$. Theorem \ref{th:quantitave} therefore recovers Theorem \ref{thm:main} by taking first $n\to\infty$, then $m\to\infty$, and finally $M\to\infty$.}

\begin{theorem}[Quantitative Breuer--Major theorem for Hilbert space-valued random variables]
\label{th:quantitave}
Suppose that the assumptions of Theorem \ref{thm:main} hold. Then, for all
$n\in\NN$,
\begin{equation}
    d_2(S_n,Z)
    \le
    \inf_{M\ge q}
    \left[
    \widetilde{\clr}_{1,M}
    +
    \frac{1}{2}\widetilde{\clr}_{1,M}^2
    +
    \widetilde{\clr}_{2,n,M}
    +
    \inf_{m\ge1}
    \left(
    \widetilde{\clr}_{3,n,m,M}
    +
    \widetilde{\clr}_{4,n,m,M}
    +
    \widetilde{\clr}_{5,n,m,M}
    \right)
    \right].
\end{equation}
\end{theorem}

\begin{proof}
From \cite[Theorem 3.5]{duker2025fourth}, it follows that
\begin{equation} \label{thm:eq-quantitative-bound}
    d_2(S_n,Z)
    \le
    \inf_{M\ge1}
    \left[
    \clr_{1,M}
    +
    \clr_{6,M}
    +
    \clr_{2,n,M}
    +
    \inf_{m\ge1}
    \left(
    \clr_{3,n,m,M}
    +
    \clr_{4,n,m,M}
    +
    \clr_{5,n,m,M}
    \right)
    \right],
\end{equation}
where the bounds are given in \MD{equations (3.60)--(3.62) of \cite{duker2025fourth}}.
We rewrite these quantities in terms of the quantities introduced above.

First, using the tail estimate from \eqref{eq:cond-1-bound-cont}, we obtain
\begin{equation}
    \clr_{1,M}
    \doteq
    \sqrt{
    \sup_{n\ge1}
    \sum_{i=1}^{\infty}
    \sum_{p=M+1}^{\infty}
    p!\|h_{n,p,i}\|_{\mfh^{\otimes p}}^2
    }
    \le
    \sqrt{
    \sum_{p=M+1}^{\infty}
    \left(
    4K_\mfc \E \|G^p[X_1]\|_{\clh_2}^2
    +
    2M_{\mfc,p}
    \right)}
    =
    \widetilde{\clr}_{1,M}.
\end{equation}
Moreover, \eqref{eq:trace-class-total-lim-cov} gives
\begin{equation}
\begin{split}
    \clr_{6,M}
    \doteq
    \frac{1}{2}
    \sum_{p=M+1}^{\infty}
    \sum_{i=1}^{\infty}
    \langle \clt_{Z_p}v_i,v_i\rangle_{\clh_2}
    =
    \frac{1}{2}
    \sum_{p=M+1}^{\infty}
    \|\clt_{Z_p}\|_{\mathcal S_1(\clh_2)} 
    \\
    \le
    \frac{1}{2}
    \sum_{p=M+1}^{\infty}
    \left(
    4K_\mfc \E \|G^p[X_1]\|_{\clh_2}^2
    +
    M_{\mfc,p}
    \right) \le
    \frac{1}{2}\widetilde{\clr}_{1,M}^2 .
\end{split}
\end{equation}

Next, using \eqref{eq:twosummands-condition-1},
\eqref{eq:cond-1-first0-term}, and \eqref{eq:cond-1-second0-term}, we have
\begin{align}
&\clr_{2,n,M}
\doteq
    \frac{1}{2}
    \sum_{p=1}^M
    \|\clt_{\cli_p(h_{n,p})}-\clt_{Z_p}\|_{\mathcal S_1(\clh_2)}
    \nonumber \\
&\le
    \frac{1}{2}
    \sum_{p=q}^M
    \Bigg[
    \Bigg\|
    \sum_{i,j=1}^{\infty}
    p!
    \sum_{|v|\ge n}
    \sum_{\bfr,\bfs\in\NN^p}
    b_{i,\bfr}b_{j,\bfs}
    \prod_{k=1}^p\rho_{r_ks_k}(v)
    (v_i\otimes v_j)
    \Bigg\|_{\mathcal S_1(\clh_2)}
    \nonumber \\
&\hspace{2.5cm}
    +
    \Bigg\|
    \sum_{i,j=1}^{\infty}
    p!
    \sum_{|v|<n}
    \frac{|v|}{n}
    \sum_{\bfr,\bfs\in\NN^p}
    b_{i,\bfr}b_{j,\bfs}
    \prod_{k=1}^p\rho_{r_ks_k}(v)
    (v_i\otimes v_j)
    \Bigg\|_{\mathcal S_1(\clh_2)}
    \Bigg]
    \nonumber \\
&\le
    \frac{1}{2}\sum_{p=q}^M
    \Bigg\{
    \E \|G^p[X_1]\|_{\clh_2}^2
    \Bigg(
        \frac{4K_\mfc^2}{n}
        +
        \sum_{K_1\le |v|<n}
        \frac{|v|}{n}
        \|R(v)\|_{\op(\clh_1)}^q
        \mathds{1}_{\{\mfc=1\}}
        +
        \sum_{|v|\ge n}
        \|R(v)\|_{\op(\clh_1)}^q
        \mathds{1}_{\{\mfc=1\}}
    \Bigg)
    \nonumber \\
&\hspace{1.5cm}
    +
    p!\,\|Q_p^{-1/2}C_p\|_{\op(\clh_2,\clh_1^{\otimes p})}^2
    \Bigg(
        \sum_{K_2\le |v|<n}
        \frac{|v|}{n}
        \|\Gamma(v)\|_{\mathcal S_1(\clh_1)}^q
        \mathds{1}_{\{\mfc=2\}}
        +
        \sum_{|v|\ge n}
        \|\Gamma(v)\|_{\mathcal S_1(\clh_1)}^q
        \mathds{1}_{\{\mfc=2\}}
    \Bigg)
    \Bigg\}
    \\
    &=
    \widetilde{\clr}_{2,n,M}.
    \label{allastlinethirdsummand}
\end{align}

We turn to $\clr_{3,n,m,M}$. For every $i\in\NN$, from \eqref{eq:hnpi-hnpi},
\begin{equation}\label{eq:cond-iii-bound-A}
        \|h_{n,p,i}\|_{\mfh^{\otimes p}}^2
    \le
    \frac{B_{\mfc,p}}{p!}, \quad  p \ge q, \quad \|h_{n,p,i}\otimes_r h_{n,p,i}\|_{\mfh^{\otimes 2(p-r)}}^2
    \le
    \mathcal A_{n,p,r}^{(i)} , p \ge 2, \;  r = 1, \dots, p -1.
\end{equation}
Therefore,
\begin{align}
    \gamma_{n,i,j}^{(p)}
    &\doteq
    \frac{1}{2}
    \sum_{\chi=1}^{p-1}
    c_p(\chi)^2
    \left(
    \|h_{n,p,i}\otimes_{p-\chi}h_{n,p,i}\|_{\mfh^{\otimes 2\chi}}^2
    +
    \|h_{n,p,j}\otimes_{p-\chi}h_{n,p,j}\|_{\mfh^{\otimes 2\chi}}^2
    \right)
    \nonumber \\
    &\le
    \frac{1}{2}
    \sum_{\chi=1}^{p-1}
    c_p(\chi)^2
    \left(
    \mathcal A_{n,p,p-\chi}^{(i)}
    +
    \mathcal A_{n,p,p-\chi}^{(j)}
    \right)
    =
    \widetilde{\gamma}_{n,i,j}^{(p)} .
\end{align}
Moreover, from \eqref{eq:cond-iii-bound-A}
\begin{equation}\label{eq:remainder-terms-def-quant-7}
\begin{aligned}
    \gamma_{n,i,j}^{(l_1,l_2)}
    &\doteq
    c(l_1,l_2)
    \|h_{n,l_1,i}\|_{\mfh^{\otimes l_1}}^2
    \|h_{n,l_2,j}\otimes_{l_2-l_1}h_{n,l_2,j}\|_{\mfh^{\otimes 2l_1}}
    \mathds{1}_{\{l_1<l_2\}}
    \\
    &\quad+
    \MD{\frac{l_1^2}{l_2^2}}c(l_2,l_1)
    \|h_{n,l_2,j}\|_{\mfh^{\otimes l_2}}^2
    \|h_{n,l_1,i}\otimes_{l_1-l_2}h_{n,l_1,i}\|_{\mfh^{\otimes 2l_2}}
    \mathds{1}_{\{l_2<l_1\}}
    \\
    &\quad+
    \sum_{\chi=1}^{l_1\wedge l_2-1}
    \MD{c(l_1,l_2,\chi)}
    \left(
    \|h_{n,l_1,i}\otimes_{l_1-\chi}h_{n,l_1,i}\|_{\mfh^{\otimes 2\chi}}^2
    +
    \|h_{n,l_2,j}\otimes_{l_2-\chi}h_{n,l_2,j}\|_{\mfh^{\otimes 2\chi}}^2
    \right)
    \\
        &\le     c(l_1,l_2)
    \frac{B_{\mfc,l_1}}{l_1!}
    \left(\mathcal A_{n,l_2,l_2-l_1}^{(j)}\right)^{1/2}
    \mathds{1}_{\{l_1<l_2\}}
    +
    \MD{\frac{l_1^2}{l_2^2}}c(l_2,l_1)
    \frac{B_{\mfc,l_2}}{l_2!}
    \left(\mathcal A_{n,l_1,l_1-l_2}^{(i)}\right)^{1/2}
    \mathds{1}_{\{l_2<l_1\}} 
    \\
    &\quad+\sum_{\chi=1}^{l_1\wedge l_2-1}
    \MD{c(l_1,l_2,\chi)}
    \left(
    \mathcal A_{n,l_1,l_1-\chi}^{(i)}
    +
    \mathcal A_{n,l_2,l_2-\chi}^{(j)}
    \right) \\
    &\le
    \widetilde{\gamma}_{n,i,j}^{(l_1,l_2)} .
\end{aligned}
\end{equation}
As a result, it follows that
\begin{equation}
\begin{split}
    \clr_{3,n,m,M}
    &\doteq
    \sqrt{m}\,M
    \sqrt{
    \sum_{i,j=1}^m
    \left(
    \sum_{p=2}^M \gamma_{n,i,j}^{(p)}
    +
    \sum_{\substack{l_1,l_2=1\\l_1\neq l_2}}^M
    \gamma_{n,i,j}^{(l_1,l_2)}
    \right)
    } \\
    &\le
    \sqrt{m}\,M
    \sqrt{
    \sum_{i,j=1}^m
    \left(
    \sum_{p=2}^M \widetilde{\gamma}_{n,i,j}^{(p)}
    +
    \sum_{\substack{l_1,l_2=1\\l_1\neq l_2}}^M
    \widetilde{\gamma}_{n,i,j}^{(l_1,l_2)}
    \right)
    }
    =
    \widetilde{\clr}_{3,n,m,M}.
\end{split}
\end{equation}
Next, by using stationarity and that $h_{n,r}=0$ for $r<q$, 
\begin{equation} \label{eq:45667}
\begin{split}
    \clr_{4,n,m,M}
    &\doteq
    \frac{M+2}{2}
    \sum_{r=q}^M
    \sum_{j=m+1}^{\infty}
    \langle \clt_{\cli_r(h_{n,r})}v_j,v_j\rangle_{\clh_2}
    \\
    &=  \frac{M+2}{2}
    \sum_{r=q}^M
    \sum_{j=m+1}^{\infty}
    r!\|h_{n,r,j}\|_{\mfh^{\otimes r}}^2 \\
    &= \frac{M+2}{2}
    \sum_{r=q}^M
    \sum_{j=m+1}^{\infty}
    r!
    \sum_{|v|<n}
    \left(1-\frac{|v|}{n}\right)
    \left\langle
    \mathfrak b_{r,j},R(v)^{\otimes r}\mathfrak b_{r,j}
    \right\rangle_{\clh_1^{\otimes r}}   \\
    &\le \frac{M+2}{2} \sum_{r=q}^M \mathfrak{D}_{n,r,m} =
    \widetilde{\clr}_{4,n,m,M}.
\end{split}
\end{equation}

Finally, by the same tail estimate used above for $\clr_{1,M}$, applied from chaos order $q$ onward,
\begin{equation}\label{eq:explicit-trace-bound-Sn}
    \|\clt_{S_n}\|_{\mathcal S_1(\clh_2)}
    \le
    \sum_{p=q}^{\infty}
    \left(
    4K_\mfc \E\|G^p[X_1]\|_{\clh_2}^2
    +
    2M_{\mfc,p}
    \right) = \sum_{p=q}^\infty B_{\mathfrak{C},p}.
\end{equation}
Combining \eqref{eq:45667} and
\eqref{eq:explicit-trace-bound-Sn}, we get
\begin{equation}
\begin{split}
    \clr_{5,n,m,M}
    &\doteq
    \|\clt_{S_n}\|_{\mathcal S_1(\clh_2)}^{1/2}
    \left(
    (M+3)
    \sum_{r=1}^M
    \sum_{j=m+1}^{\infty}
    \langle \clt_{\cli_r(h_{n,r})}v_j,v_j\rangle_{\clh_2}
    \right)^{1/2}
    \\ &\le \left( \sum_{p=q}^\infty B_{\mathfrak{C},p} \right)^{1/2}
    \left(
    (M+3)
    \sum_{r=q}^M
    \mathfrak D_{n,r,m}
    \right)^{1/2}
    =
    \widetilde{\clr}_{5,n,m,M}.
\end{split}
\end{equation}
\MD{Here the factor $M+3$ comes from the joint bound for the two off-diagonal operator blocks in equations (3.53) and (3.60) of \cite{duker2025fourth}.}
\end{proof}

\section{Examples and applications} \label{sec:applications}

We provide in this section some examples for possible operators $G$. Suppose that $X_k$ is a Gaussian and stationary process taking values in a Hilbert space $\clh$, with covariance (and hence nuclear) operator $Q[\cdot] \doteq \E \langle X_1, \cdot \rangle_{\clh} X_1$. Denote by $\{e_j\}_{j \in \NN}$ the basis of $\clh$.

\subsection{Sample covariance operator} \label{subsec:sample-covar}

A natural estimator for the covariance operator $Q$, which has been studied in, e.g., Section 4.1 of \cite{Bosq2000:Linear} (see also \cite{mas2006sufficient}) is given by
\begin{equation*}
\Gamma_n \in \cll(\clh,\clh), \quad 
    \Gamma_{n}[\cdot] \doteq \frac{1}{n} \sum_{k=1}^n \langle X_k, \cdot \rangle_{\clh} X_k.
\end{equation*}
We suppose that $\Gamma_{n}[\cdot]$ is a random element in $\mathcal{S}_2(\clh)$, namely that $\|\Gamma_n\|^2_{\mathcal{S}_2(\clh)} = \sum_{k=1}^\infty \| \Gamma_n[v_k]\|^2_{\clh} < \infty$. Then, $S_n$ defined in \eqref{eq:partial_sum_G} can be rewritten as {$S_n = \sqrt{n}(\Gamma_n-Q)$, with $G[x] = \langle x, \cdot \rangle_{\clh} x-Q$} , $\clh_{1} = \clh$ and $\clh_{2} = \mathcal{S}_2(\clh)$.   {The } condition $\E \| G [X_1] \|^2_{\mathcal{S}_2(\clh)} < \infty$ in Theorem \ref{thm:main}   {is automatically satisfied for Gaussian \(X_1\) with trace-class covariance \(Q\). Indeed, for the rank-one operator \(x\otimes x=\langle x,\cdot\rangle_\clh x\),
} \begin{equation} \label{eq:EGeigenvalueassumptions}
\E \|  X_1  {\otimes X_1 } \|^2_{\mathcal{S}_2(\clh)} 
=
\E  \| X  {_1\|_\clh^4
=
2\|Q\|^2_{\mathcal S_2(\clh)}} + \tr(Q)^2
 <\infty .
 \end{equation}
  Consequently \(\E\|G[X_1]\|_{\mathcal S_2(\clh)}^2<\infty\) as well.

We argue that the Hermite rank in Definition \ref{def:Hermiterank} of the map $G[x] \doteq \langle x, \cdot \rangle_{\clh} x - Q$ is equal to two.
We aim to write the sample covariance operator using the Hermite expansion \eqref{eq:Hermite_expansion}. Recall that $X_k \otimes X_k$ is an unbiased estimator, i.e., $\E [X_k \otimes X_k] = Q$. Equivalently, the map \(G[x]=x\otimes x-Q\) is centered. Then, we can calculate explicitly the Hermite coefficients for $r,s \in \NN, \bfl \in \cll$ as follows:
\begin{align}
c_{r,s,\boldsymbol{l}} 
&\doteq \frac{1}{\prod_{j=1}^\infty l_j!}
\langle G[X_1] - \E G[X_1], H_{\bm{l}}(X_1) \otimes (\langle e_{r}, \cdot \rangle_{\clh} e_s) \rangle_{L^{2}(\Omega:\RR) \otimes \mathcal{S}_2(\clh)}
\nonumber
\\&= \frac{1}{\prod_{j=1}^\infty l_j!}
 \left\langle \langle X_1, e_r \rangle_{\clh} \langle X_1, e_s \rangle_\clh - \langle Q[e_r], e_s \rangle_\clh, \prod_{j=1}^\infty H_{l_j}(W_{e_j}(X_1) ) \right\rangle_{L^2(\Om:\RR)}
\label{eq:Hcof_cov_op_1}
\\&= \frac{1}{\prod_{j=1}^\infty l_j!} \left[
\E \left(
\langle X_1, e_r \rangle_{\clh} \langle X_1 , e_s \rangle_{\clh}
\prod_{j=1}^\infty H_{l_j} (W_{u_j}(X_1) )
\right) - \langle Q[e_r], e_s \rangle_{\clh} \E \left( 
\prod_{j=1}^\infty H_{l_j}(W_{e_j}(X_1) ) \right) \right]
\label{eq:Hcof_cov_op_2}
\\&= 
\frac{\lambda_r^{\frac{1}{2}} \lambda_s^{\frac{1}{2}}}{\prod_{j=1}^\infty l_j!}
\E \left[ H_1(W_{e_r}(X_1) )  H_1(W_{e_s}(X_1))
\prod_{j=1}^\infty H_{l_j}(W_{e_j}(X_1) ) \right],
\label{eq:Hcof_cov_op_3}
\end{align}
where \eqref{eq:Hcof_cov_op_1} follows from the definition of $G$ and by taking the inner product with respect to $\mathcal{S}_2(\clh)$, \eqref{eq:Hcof_cov_op_3} follows upon noticing that the second summand in \eqref{eq:Hcof_cov_op_2} is zero and by \eqref{eq:white-noise-def}. Note that, for $r \neq s$, 
\begin{equation*}
\begin{split}
 c_{r,s,\boldsymbol{l}} &=   
 \begin{cases}
    \frac{\lambda_r^{\frac{1}{2}} \lambda_s^{\frac{1}{2}} }{\prod_{j=1}^\infty l_j!}
\E \left[ (H_1(W_{e_r}(X_1)))^2   (H_1(W_{e_s}(X_1)))^2\right] & \text{if } l_r = l_s = 1, \bfl \in \cll_2, \\
0 & \text{otherwise},
\end{cases} \\
&= \begin{cases}
    \lambda_r^{\frac{1}{2}} \lambda_s^{\frac{1}{2}} & \text{if } l_r = l_s = 1, \bfl \in \cll_2, \\
0 & \text{otherwise},
\end{cases}
\end{split}
\end{equation*}
with $\cll_2 = \{\boldsymbol{l} \in \NN_0^\infty: \sum_{k=1}^\infty l_k = 2 \}$, while for $r = s$
\begin{equation*}
\begin{split}
c_{r,r,\boldsymbol{l}} &=
\begin{cases}
    \frac{\lambda_r}{\prod_{j=1}^\infty l_j!}
\E \left[ (H_2(W_{u_r}(X_1)) + 1)   H_2(W_{u_r}(X_1))
\right] & \text{if } l_r = 2, \bfl \in \cll_2, \\
0 & \text{otherwise},
\end{cases} \\
&= \begin{cases}
    \lambda_r  & \text{if } l_r =  2, \bfl \in \cll_2, \\
0 & \text{otherwise}.
\end{cases}
\end{split}
\end{equation*}
Altogether, this says that
\begin{equation*}
    c_{r,s,\bfl} = 
    \begin{cases}
    \lambda_r^{\frac{1}{2}} \lambda_s^{\frac{1}{2}} & \text{if } l_r = l_s = 1, \bfl \in \cll_2, \\
    \lambda_r  & \text{if } l_r =  2, \bfl \in \cll_2, \\
0 & \text{otherwise}.
\end{cases}
\end{equation*}

We now compute \(C_2\). By the definition of \(C_p\), 
\begin{equation}
    C_p = 0 , \quad \text{for }p \neq 2, \quad C_2 ( e_i \otimes e_j) = \begin{cases}
        \MD{\lambda_i} e_i \otimes e_i & \text{if } i = j , \\
        \MD{\frac{\sqrt{\lambda_i \lambda_j}}{2}(e_i\otimes e_j+e_j\otimes e_i)} & \text{if } i \neq j.
    \end{cases}
\end{equation}
Hence
\begin{equation*}
    Q_2^{-1/2}C_2 (e_i \otimes e_j)= \begin{cases}
        e_i\otimes e_j & \text{if }i = j, \\
        \frac{1}{2} \left( e_i \otimes e_j + e_j \otimes e_i \right) & \text{if } i \neq j,
    \end{cases}
\end{equation*}
so that $\sum_{p = 1}^\infty p! \|Q_p^{-1/2}C_p\|^2_{\op} = 2 \| Q_2^{-1/2}C_2 \|^2_{\op} \le 2$.
Therefore condition \eqref{thm:ass-2-2:different} holds for the sample covariance operator for any trace-class covariance operator $Q$.

Then, we can infer that for every Gaussian process $\{X_k\}_{k \in \ZZ}$ satisfying condition \eqref{thm:ass-2.1-eq-Gamma} or \eqref{thm:ass-2.0-eq-Gamma} for $q=2$, its sample covariance operator satisfies
\begin{equation} \label{eq:conv_cov_operator}
    \sqrt{n} \left( \Gamma_n - Q \right)
    \xrightarrow{d} Z,
\end{equation}
where $Z$ is a $\mathcal{S}_2(\clh)$-valued centered Gaussian element with covariance operator given by \eqref{eq:def-T}, with $\clh_2 = \mathcal{S}_2(\clh)$ and $G[x] = \langle x, \cdot \rangle_{\clh}  x - Q[\cdot]$.

Analogous (but simpler) calculations can be done for the sample mean $\frac{1}{n} \sum_{k=1}^n X_k$ (here $G = \operatorname{Id}_\clh$ and $\clh_1 = \clh_2 = \clh$), giving rise to the generalized Hermite coefficients $c_{i,\bfl} = \lambda_i^{\frac{1}{2}} \delta_{l_i = 1} \delta_{\bfl \in \cll_1}$, for $i \in \NN$. This says that the generalized Hermite rank is $q=1$, and so one can apply Theorem \ref{thm:main} whenever $\{X_k\}_{k \in \ZZ}$ satisfies the assumptions of Theorem \ref{thm:main} for $q = 1$.

\subsection{Eigenvalue estimation}

Suppose that the eigenvectors $\{v_j\}_{j \in \NN}$ of the covariance operator $Q$ are known and consider the problem of estimating their corresponding positive eigenvalues $\{\lambda_j\}_{j \in \NN}$. The following consistent estimators $\hat \lambda_{jn}$ for $\lambda_j$ were considered in Section 4.2 of \cite{Bosq2000:Linear}, 
\begin{equation*}
    \hat \lambda_{jn} = \frac{1}{n} \sum_{k=1}^n \langle X_k, v_j \rangle_{\clh}^2, \quad \E \hat \lambda_{jn} = \lambda_j, \quad j,n \in \NN,
\end{equation*}
giving rise to the setting of Theorem \ref{thm:main} with $\clh_1 = \clh, \clh_2 = \RR, S_n = \sqrt{n}(\hat \lambda_{jn} - \lambda_j)$ and $G_j[x] = \langle x,v_j \rangle_{\clh}^2-\lambda_j$. Fix some $j \in \NN$ and note that $\E |G_j[X_1]|^2 = \EE \left( \langle X_1,v_j \rangle^2_{\clh} - \lambda_j \right)^2  = 2 \lambda_j^2< \infty$. Moreover, by arguing as in the previous section, we have that the coefficients of $G_j$ are $c_{\bfl} = \lambda_j \delta_{l_j 2}$, and so $G_j$ has Hermite rank $ q=2$. Theorem \ref{thm:main} implies that if condition \eqref{thm:ass-2.1-eq-Gamma} is met with $q=2$, then for all $j\in \NN$,
\begin{equation*}
   \sqrt{n}( \hat \lambda_{jn} - \lambda_j ) \xrightarrow{d} \mathcal{N}(0,\sigma^2_j),
\end{equation*}
where $\sigma^2_j \doteq  \E G_j^2[X_1] + 2 \sum_{k=1}^\infty \E G_j[X_1] G_j[X_{k+1}]$. When $\{X_k\}$ is a Gaussian ARH(1) process, this result recovers Theorem 4.10 in \cite{Bosq2000:Linear}.

In fact, we can strengthen this result by considering the simultaneous estimation of all eigenvalues. Since  $\lambda \doteq  (\lambda_1,\lambda_2,\dots) \in \ell^1(\NN)$, it follows that $\lambda  \in \ell^2(\NN), \;  \hat \lambda_n = (\hat \lambda_{1n},\hat \lambda_{2n},\dots) \in \ell^2(\NN)$, and $\{\psi_j\}$ denote the canonical basis of $\ell^2(\NN)$. Then, $G[x] \in \ell^2(\NN)$, where $G_j[x] = \langle x,v_j \rangle_{\clh}^2-\lambda_j$ and it holds that $\E\| G[X_1] \|^2_{\ell^2(\NN)} < \infty$. Here, $G$ has Hermite coefficients $c_{j,\bfl} = \lambda_j \delta_{l_j 2}$ and Hermite rank $q=2$. 

Moreover, by arguing as in the sample covariance operator, we can see that condition \eqref{thm:ass-2-2:different} holds for the eigenvalue estimator as well. Then, by Theorem \ref{thm:main},
\begin{equation*}
    \sqrt{n} (\hat \lambda_n - \lambda) \xrightarrow{d} S,
\end{equation*}
where $S$ is a Gaussian element of $\ell^2(\NN)$ with covariance operator $\covop_S$, such that 
\begin{equation*}
    \langle \covop_S(\psi_i), \psi_j \rangle_{\ell^2(\NN)} = \sum_{v \in \ZZ} \EE \left( G_i [X_0]  G_j [X_v] \right) .
\end{equation*}

\subsection{Shallow neural operators with Gaussian initializations}

In a recent work, \cite{kovachki2023neural} introduced a framework for learning operators, termed \textit{neural operators}. Numerous quantitative CLTs have been investigated in the context of neural networks, e.g., \cite{favaro2023quantitative,kovachki2023neural,bordino2023non}. In this section, we study the limiting distribution of a single layer neural operator with random initializations of the parameters, as the width of the network goes to infinity. The single layer update for these approximation schemes takes the form of some output $u \in L^2(D:\RR^m)$, where $D$ is a bounded domain $D \subset \RR^l$; more precisely, in a simplified form, 
\begin{equation*}
    u(x) = \sigma\left(\int_{D} \kappa(x,y)v(y)d\mu(y) \right), \quad x \in D,
\end{equation*}
where $\kappa: D \times D \to \RR^{m\times m}$ is a suitable kernel, $v \in  L^2(D:\RR^m)$ is the input function, $\mu$ is a suitable measure, and $\sigma:\RR^m \to \RR^m$ is an \textit{activation function}. Here, we take $\mu$ to be the Lebesgue measure and write $\clh_D \doteq L^2(D:\RR^m)$.
We assume that $\sigma$ is Lipschitz and denote by $N_\sigma:\clh_D\to\clh_D$ the map
\begin{equation*}
    N_\sigma(f)(x) \doteq \sigma(f(x)), \quad f\in \clh_D,\ x\in D.
\end{equation*}
There are several options for kernels, but here we select the so-called \textit{low-rank neural operators}; see Section 4.2 in \cite{kovachki2023neural}.
Low-rank neural operators are formally defined through
\begin{equation*}
    \kappa(x,y) = \sum_{j=1}^r \Phi^{(j)}(x) (\Psi^{(j)}(y))', \quad \text{implying}   \int_D \kappa(\cdot,y) v(y) d\mu(y) = \sum_{j=1}^r \langle \Psi^{(j)}, v \rangle_{\clh_D }\Phi^{(j)}(\cdot),
\end{equation*}
where $r \in \NN$ is a constant termed the \textit{rank} of the kernel and, for $i = 1,\dots,r$, $\Phi^{(i)},\Psi^{(i)}$ are $\clh_D$-valued centered Gaussian  random variables. Recall the Cartesian inner product $\langle x,y \rangle_{\clh_D^{2r}} = \sum_{i=1}^{2r} \langle x_i, y_i \rangle_{\clh_D}$, where $x_i$ denotes the $i$-th marginal of $x \in \clh_D^{2r}$. Let $\{\Phi_k^{(i)},\Psi_k^{(i)}\}_{i = 1,\dots,r,k \in \ZZ}$ be a centered Gaussian sequence such that, letting for $k \in \ZZ$
\begin{equation} \label{def:Xk-neural}
    X_k \doteq \left(\Phi_k^{(1)}, \dots, \Phi_k^{(r)}, \Psi_k^{(1)}, \dots, \Psi_k^{(r)}\right) \in \clh_D^{2r} \doteq \clh_1,
\end{equation}
the sequence $\{X_k\}_{k \in \ZZ}$ is stationary with nondegenerate covariance operator $Q$. 

Let $\nu$ be a Borel probability measure on $\clh_D$ satisfying
\begin{equation}
    \int_{\clh_D}\|v\|_{\clh_D}^2 \nu(dv)<\infty,
\end{equation}
and set $\clh_2 \doteq L^2(\clh_D,\nu:\clh_D)$.
For $x=(\phi_1,\dots,\phi_r,\psi_1,\dots,\psi_r)\in \clh_1$, define the finite-rank operator
\begin{equation} \label{eq:K_x-neural}
    K_x v \doteq \sum_{j=1}^r \langle \psi_j, v \rangle_{\clh_D}\phi_j, \quad v\in \clh_D.
\end{equation}
For $X_k$ as in \eqref{def:Xk-neural}, define 
\begin{equation*}
    G_\sigma[X_k][v] \doteq  N_\sigma \left( \sum_{j=1}^r \langle \Psi_k^{(j)}, v \rangle_{\clh_D} \Phi_k^{(j)}   \right), \quad v\in \clh_D.
\end{equation*}
Thus, $G_\sigma[X_k]$ is viewed as an element of $\clh_2=L^2(\clh_D,\nu:\clh_D)$. It is left to show that $G$ satisfies Assumption \ref{thm:ass-2.2}. We have, $G_\sigma \in L^2(\clh_1,\gamma_Q:\clh_2)$ since, with further explanations given below, for a constant $C_\sigma>0$ depending only on $\sigma$ and $D$,
\begin{align}
\E\|G_\sigma[X_1]\|_{\clh_2}^2
    &\leq C_\sigma
    \left( 1+
    \int_{\clh_D}\|v\|_{\clh_D}^2\nu(dv)\,
    \E\|K_{X_1}\|_{\mathcal S_2(\clh_D)}^2
    \right)
    \label{al:neural-1}
    \\
    &\leq C_\sigma
    \left(
    1+
    \int_{\clh_D}\|v\|_{\clh_D}^2\nu(dv)\,
    \E
        \left[
            \left(
                \sum_{j=1}^r
                \|\Phi_1^{(j)}\|_{\clh_D}
                \|\Psi_1^{(j)}\|_{\clh_D}
            \right)^2
        \right]
    \right)
    \label{al:neural-2}
    \\&\leq C_\sigma
    \left(
        1+
        \int_{\clh_D}\|v\|_{\clh_D}^2\nu(dv)\,
        \E
        \left[ 
        \sum_{j_1,j_2=1}^r
        \|\Phi_1^{(j_1)}\|^2_{\clh_D}
        \|\Psi_1^{(j_1)}\|^2_{\clh_D}
        \right]
    \right)
    \label{al:neural-3}
    \\&\leq C_\sigma
    \left(
        1+
        \int_{\clh_D}\|v\|_{\clh_D}^2\nu(dv)\, 
        \sum_{j_1,j_2=1}^r
        \left(
        \E
        \left[
        \|\Phi_1^{(j_1)}\|^4_{\clh_D}
        \right]
        \E
        \left[
        \|\Psi_1^{(j_1)}\|^4_{\clh_D}
        \right]
        \right)^{\frac{1}{2}}
    \right)
    \label{al:neural-4}
    <\infty.
\end{align}
Recall the finite-rank operator $K_x$ from \eqref{eq:K_x-neural}. Then, the inequality \eqref{al:neural-1} can be inferred, since, for
$x=(\phi_1,\dots,\phi_r,\psi_1,\dots,\psi_r)\in\clh_1$,
\begin{align}
    \|G_\sigma[x]\|_{\clh_2}^2
    &= \int_{\clh_D} \|N_\sigma(K_xv)\|_{\clh_D}^2 \nu(dv) \\
    &\leq \int_{\clh_D}
    \left(2a_\sigma^2 |D| + 2b_\sigma^2 \|K_xv\|_{\clh_D}^2\right)
    \nu(dv) 
    \label{al:neural-5}
    \\
    &\leq 2a_\sigma^2 |D|
    + 2b_\sigma^2 \|K_x\|_{\mathrm{op}(\clh_D)}^2
    \int_{\clh_D}\|v\|_{\clh_D}^2\nu(dv) \\
    &\leq 2a_\sigma^2 |D|
    + 2b_\sigma^2 \|K_x\|_{\mathcal S_2(\clh_D)}^2
    \int_{\clh_D}\|v\|_{\clh_D}^2\nu(dv) \\
    &\leq C_\sigma
    \left(
        1+\|K_x\|_{\mathcal S_2(\clh_D)}^2
        \int_{\clh_D}\|v\|_{\clh_D}^2\nu(dv)
    \right),
\end{align}
where \eqref{al:neural-5} follows since 
$\sigma$ is Lipschitz and $|D|$ denotes the Lebesgue measure of $D$. To be more precise, there exist constants $a_\sigma,b_\sigma>0$
such that
\begin{equation}
    \|\sigma(z)\|_{\RR^m} \leq a_\sigma+b_\sigma \|z\|_{\RR^m},
    \quad z\in \RR^m.
\end{equation}
Since $D$ is bounded, this also implies that, for every $f\in\clh_D$,
\begin{align*}
    \|N_\sigma(f)\|_{\clh_D}^2
    &= \int_D \|\sigma(f(x)) \|_{\RR^m}^2 dx 
    \leq \int_D \left(a_\sigma+b_\sigma \|f(x)\|_{\RR^m}\right)^2 dx \leq 2a_\sigma^2 |D| + 2b_\sigma^2 \|f\|_{\clh_D}^2.
\end{align*}

Inequality \eqref{al:neural-2} follows by the triangle inequality in $\mathcal S_2(\clh_D)$,
\begin{align*}
    \|K_x\|_{\mathcal S_2(\clh_D)}
    \leq \sum_{j=1}^r \|\phi_j\otimes\psi_j\|_{\mathcal S_2(\clh_D)} = \sum_{j=1}^r \|\phi_j\|_{\clh_D}\|\psi_j\|_{\clh_D}.
\end{align*}
Finally, \eqref{al:neural-3} and \eqref{al:neural-4} are applications of H\"older's inequality.
Finiteness in \eqref{al:neural-4} follows since $\int_{\clh}\|v\|_{\clh_D}^2\nu(dv)<\infty$ and
$X_1$ is Gaussian in $\clh_1$ and therefore admits finite fourth moments. Thus
$G_\sigma\in L^2(\clh_1,\gamma_Q:\clh_2)$.

For simplicity, suppose that $\E G_\sigma[X_1] = 0$. We suppose further that $G_\sigma$ has Hermite rank $q_\sigma$. Then, $G_\sigma\in L^2(\clh_1,\gamma_Q:\clh_2)$ together with
\[
    \sum_{v\in\ZZ}\|R(v)\|_{\mathrm{op}(\clh_1)}^{q_\sigma}<\infty,
\]
\MD{implies, by Theorem \ref{thm:main}, that for the one-layer, $n$-width neural operator,} 
\begin{equation*}
    S_n = \frac{1}{\sqrt{n}} \sum_{k=1}^n G_{\sigma}[X_k] \xrightarrow{d} Z_\sigma,
\end{equation*}
\MD{where $Z_\sigma$ is a possibly degenerate, centered $\clh_2$-valued Gaussian random variable with covariance operator given in \eqref{eq:def-T}.}

\section{Auxiliary results} \label{app:aux}

\begin{lemma} \label{le:ana_to_Le53Bou20}
Suppose the assumptions of Theorem \ref{th:cont_case} hold.
Then, for each fixed $p \ge q$
\begin{equation*}
\| 
  \clt_{\mathcal{I}_p(  h_{n,p, \cdot} )}
- 
\clt_{B} \otimes \clt_{\mathcal{I}_p(h_{n,p})} \|_{\mathcal{S}_1(L^2([0,1]) \otimes \clh_2)} \to 0, \quad \text{as } n \to \infty.
\end{equation*}
\end{lemma}

\begin{proof}
Note first that
\begin{align}
    &  \clt_{\mathcal{I}_p(  h_{n,p, \cdot} )}
    - 
    \clt_{B} \otimes \clt_{\mathcal{I}_p(h_{n,p})}
    \\&=
    \sum_{i,j =1}^\infty
    p! \frac{1}{n} \sum_{k_1, k_2=1}^n 
    \left(\1_{\big[\frac{k_1}{n},1\big]} \otimes \1_{\big[\frac{k_2}{n},1\big]} \right)
    \sum_{\bfr,\bfs \in \NN^p} b_{i,\bfr} b_{j,\bfs} \prod_{k=1}^p \rho_{r_k s_k}(k_1 - k_2) (v_i \otimes v_j)
    \nonumber
    \\&\hspace{1cm}-
    \clt_{B} \otimes \sum_{i,j =1}^\infty
    p! \frac{1}{n} \sum_{k_1, k_2=1}^n \sum_{\bfr,\bfs \in \NN^p} b_{i,\bfr} b_{j,\bfs} \prod_{k=1}^p \rho_{r_k s_k}(k_1 - k_2) (v_i \otimes v_j).
\label{eq:twosummands-condition-1-seq}
\end{align}
For each $k\in\{-(n-1),\dots,n-1\}$, the number of pairs
$(k_1,k_2)\in\{1,\dots,n\}^2$ such that $k_1-k_2=k$ is $n-|k|$.
Hence, grouping the sum in \eqref{eq:twosummands-condition-1-seq} according to the value of $k_1-k_2$, and recalling \eqref{eq:rewrite-as-operators}, we can recast the operator difference as
\begin{align}
  \clt_{\mathcal{I}_p(  h_{n,p, \cdot} )}
    - 
    \clt_{B} \otimes \clt_{\mathcal{I}_p(h_{n,p})}
&=
p!\sum_{k=0}^{n-1}
\left(
\frac{1}{n}\sum_{j=1}^{n-k}
\1_{[\frac{j+k}{n},1]}\otimes \1_{[\frac{j}{n},1]}
-
\left(1-\frac{k}{n}\right)\clt_B
\right)\otimes B_k
\nonumber
\\
&\qquad
+
p!\sum_{k=1}^{n-1}
\left(
\frac{1}{n}\sum_{j=1}^{n-k}
\1_{[\frac{j}{n},1]}\otimes \1_{[\frac{j+k}{n},1]}
-
\left(1-\frac{k}{n}\right)\clt_B
\right)\otimes B_{-k}
\label{eq:Tnp-grouped-lags}
\\
&=
p!\sum_{|k|<n} \mathcal{A}_{kn}\otimes B_k,
\end{align}
where $\clt_B$ is the covariance operator given in \eqref{eq:BrownianMotion_int} and the operators $\mathcal{A}_{kn}, B_k$ are defined by
\begin{equation} \label{eq:def-Bk}
\begin{split}
        B_k
&\doteq
C^*_p R(k)^{\otimes p}C_p, \\
\mathcal{A}_{kn} &\doteq \begin{cases}
    \frac{1}{n}\sum_{j=1}^{n-k}
\1_{[\frac{j+k}{n},1]}\otimes \1_{[\frac{j}{n},1]}
-
\left(1-\frac{k}{n}\right)\clt_B, & k=0,\dots,n-1 \\
\left( \frac{1}{n}\sum_{j=1}^{n-|k|}
\1_{[\frac{j+|k|}{n},1]}\otimes \1_{[\frac{j}{n},1]}
-
\left(1-\frac{|k|}{n}\right)\clt_B \right)^*, &  k=-(n-1),\dots,-1.
\end{cases}
\end{split}
\end{equation}

Therefore, by the triangle inequality,
\begin{align}
\|  \clt_{\mathcal{I}_p(  h_{n,p, \cdot} )}
    - 
    \clt_{B} \otimes \clt_{\mathcal{I}_p(h_{n,p})}\|_{\mathcal{S}_1(L^2([0,1])\otimes \mathcal{H}_2)}
&\le
p!\sum_{|k|<n}
\|\mathcal{A}_{kn}\otimes B_k\|_{\mathcal{S}_1(L^2([0,1])\otimes \mathcal{H}_2)}
\nonumber
\\
&=
p!\sum_{|k|<n}
\|\mathcal{A}_{kn}\|_{\mathcal{S}_1(L^2([0,1]))}
\|B_k\|_{\mathcal{S}_1(\mathcal{H}_2)}
\label{al:Tnkbound-1}
\\
&\le
p!\sum_{|k|<n}
\left(
\frac{1}{\sqrt{n}}
+
\sqrt{\frac{|k|}{2n}}
\right)
\|B_k\|_{\mathcal{S}_1(\mathcal{H}_2)}.
\label{al:Tnkbound-2}
\end{align}
Here \eqref{al:Tnkbound-1} follows from the tensor product formula for trace norms presented above \eqref{eq:Q-decomp-hs}, and the inequality \eqref{al:Tnkbound-2} follows by Lemma
\ref{le:Tnp-trace-class-bound}, together with
$\|\mathcal{A}_{kn}\|_{\mathcal{S}_1(L^2([0,1]))}
=
\|\mathcal{A}^*_{|k|n}\|_{\mathcal{S}_1(L^2([0,1]))}$ for $k<0$.

It remains to prove the convergence to $0$ under Assumptions
\ref{thm:ass-2.2} or \ref{thm:ass-2.1}.
Then, for $\ell \ge K_\mfc$, where $\mfc = 1,2$ is as in \eqref{eq:de:theta_theta-1} corresponding to Assumptions
\ref{thm:ass-2.2} and \ref{thm:ass-2.1} respectively,
\begin{align*}
\|  \clt_{\mathcal{I}_p(  h_{n,p, \cdot} )}
    - 
    \clt_{B} \otimes \clt_{\mathcal{I}_p(h_{n,p})}\|_{\mathcal{S}_1(L^2([0,1])\otimes \mathcal{H}_2)}
&\le
p!\sum_{|k|\le \ell}
\left(
\frac{1}{\sqrt{n}}
+
\sqrt{\frac{|k|}{2n}}
\right)
\bigl\|
C^*_p R(k)^{\otimes p}C_p
\bigr\|_{\mathcal{S}_1(\mathcal{H}_2)}
\\
&\qquad
+
p!\sum_{\ell < |k| < n}
\left(
\frac{1}{\sqrt{n}}
+
\sqrt{\frac{|k|}{2n}}
\right)
\bigl\|
C^*_p R(k)^{\otimes p}C_p
\bigr\|_{\mathcal{S}_1(\mathcal{H}_2)}
\\
&\le
p!\left(
\frac{1}{\sqrt{n}}
+
\sqrt{\frac{\ell}{2n}}
\right)
\sum_{|k|\le \ell}
\bigl\|
C^*_p R(k)^{\otimes p}C_p
\bigr\|_{\mathcal{S}_1(\mathcal{H}_2)}
\\
&\qquad
+
p!\left(
1+\frac{1}{\sqrt{2}}
\right)
\sum_{|k|>\ell}
\bigl\|
C^*_p R(k)^{\otimes p}C_p
\bigr\|_{\mathcal{S}_1(\mathcal{H}_2)},
\end{align*}
where the last line follows since $\frac{1}{\sqrt{n}}+\sqrt{\frac{|k|}{2n}}
\le
1+\frac{1}{\sqrt{2}}$.
Letting first $n\to\infty$ and then $\ell \to\infty$, and using
\eqref{eq:R1argument7}, \eqref{eq:R2argument7}, and \eqref{eq:R2argument10},
together with Assumptions \ref{thm:ass-2.2} or \ref{thm:ass-2.1}, respectively, we conclude that $\|  \clt_{\mathcal{I}_p(  h_{n,p, \cdot} )}
    - 
    \clt_{B} \otimes \clt_{\mathcal{I}_p(h_{n,p})}\|_{\mathcal{S}_1(L^2([0,1])\otimes \mathcal{H}_2)}
\to 0$ as $n \to \infty$.
\end{proof}

The following lemma was used in the proof of Lemma \ref{le:ana_to_Le53Bou20}.

\begin{lemma} \label{le:Tnp-trace-class-bound}
Recall $\mathcal{A}_{kn}$ from \eqref{eq:def-Bk}.
Then, for every $|k|<n$,
\begin{equation} \label{eq:Akn-trace-bound}
\|\mathcal{A}_{kn}\|_{\mathcal{S}_1(L^2([0,1]))}
\le
\frac{1}{\sqrt n}
+
\sqrt{\frac{|k|}{2n}}.
\end{equation}
\end{lemma}

\begin{proof}
We estimate $\|\mathcal{A}_{kn}\|_{\mathcal{S}_1(L^2([0,1]))}$ for $k=0,\dots,n-1$; the same bound for negative $k$ then follows since the trace norm is invariant under taking adjoints.

Note that the covariance operator $\mathcal{T}_B$ in \eqref{eq:BrownianMotion_int}
has kernel $\kappa(s,t)=s\wedge t$ and admits the representation
\begin{equation} \label{eq:C0-rank-one-integral}
\mathcal{T}_B
=
\int_0^1 g_u\otimes g_u \,du,
\qquad g_u \doteq \mathbf{1}_{[u,1]} .
\end{equation}

For $a\in[0,1]$, define the right-shift operator $S_a$ on $L^2([0,1])$, for $u\in[0,1-a]$, by
\begin{equation} \label{eq:shift-on-gu}
(S_af)(s)
\doteq
\1_{[a,1]}(s)f(s-a), \quad  \text{such that } \;S_ag_u = g_{u+a}.
\end{equation}
Then $S_a$ is a contraction on $L^2([0,1])$. 

Next, define the empirical counterpart of \eqref{eq:C0-rank-one-integral} as
\begin{equation} \label{eq:def-Dn}
\mathcal{D}_n
\doteq
\frac{1}{n}\sum_{j=1}^n g_{j/n}\otimes g_{j/n}, \quad \text{with kernel } \mathcal{D}_n(s,t) \doteq
\frac{1}{n}\sum_{j=1}^n \1_{\{j/n\le s\wedge t\}} = \frac{\lfloor n(s\wedge t)\rfloor}{n}.
\end{equation}
Fix $k\in\{0,\dots,n-1\}$ and write $a=\frac{k}{n}$. By \eqref{eq:def-Bk}, \eqref{eq:shift-on-gu}, and \eqref{eq:def-Dn},
\begin{align}
\mathcal{A}_{kn}
&=
\frac{1}{n}\sum_{j=1}^{n-k} g_{(j+k)/n}\otimes g_{j/n}
-
(1-a)\mathcal{T}_B
\nonumber
\\
&=
S_a\left(
\frac{1}{n}\sum_{j=1}^{n-k} g_{j/n}\otimes g_{j/n}
\right)
-
(1-a)\mathcal{T}_B
\label{eq:Akn-first-rewrite}
\\
&=
S_a \mathcal{D}_n
-
(1-a)\mathcal{T}_B
\label{eq:Akn-first-rewrite-2}
\\
&=
S_a(\mathcal{D}_n-\mathcal{T}_B)
+
\bigl(S_a-(1-a)I\bigr)\mathcal{T}_B,
\label{eq:Akn-decomp}
\end{align}
where \eqref{eq:Akn-first-rewrite-2} follows by \eqref{eq:def-Dn}, since $S_a g_{j/n}=0$ for $j>n-k$.

We estimate the two terms in \eqref{eq:Akn-decomp} separately. For the first summand, note that
\begin{align}
\|S_a(\mathcal{D}_n-\mathcal{T}_B)\|_{\mathcal{S}_1(L^2([0,1]))}
&\leq
\| S_a \|_{\op(L^2([0,1]))}
\|\mathcal{D}_n-\mathcal{T}_B\|_{\mathcal{S}_1(L^2([0,1]))}
\nonumber
\\
&\le
\|\mathcal{D}_n-\mathcal{T}_B\|_{\mathcal{S}_1(L^2([0,1]))},
\label{eq:Akn-decomp-1-1}
\end{align}
where \eqref{eq:Akn-decomp-1-1} follows by Theorem 18.11 (g) in \cite{conway2000course} and since $S_a$ is a contraction on $L^2([0,1])$.

We now bound $\|\mathcal{D}_n-\mathcal{T}_B\|_{\mathcal{S}_1(L^2([0,1]))}$. We can decompose $\mathcal{T}_B=\mathcal{V}\mathcal{V}^*$ (resp. $\mathcal{D}_n=\mathcal{V}_n\mathcal{V}_n^*$), where $\mathcal{V}$ and $\mathcal{V}_n$ are given by
\[
(\mathcal{V}f)(s) \doteq \int_0^s f(u)du, \quad (\mathcal{V}_n f)(s)
\doteq
\int_0^{\lfloor ns\rfloor/n} f(u)du .
\]
Above, $\mathcal{V}$ is the usual Volterra operator. Hence
\[
\mathcal{D}_n-\mathcal{T}_B
=
\mathcal{V}_n(\mathcal{V}_n^*-\mathcal{V}^*)
+
(\mathcal{V}_n-\mathcal{V})\mathcal{V}^* .
\]
Therefore, by the ideal inequality for Schatten norms,
\begin{align}
\|\mathcal{D}_n-\mathcal{T}_B\|_{\mathcal{S}_1(L^2([0,1]))}
&\le
\|\mathcal{V}_n\|_{\mathcal{S}_2(L^2([0,1]))}
\|\mathcal{V}_n-\mathcal{V}\|_{\mathcal{S}_2(L^2([0,1]))}
\nonumber
\\
&\quad+
\|\mathcal{V}_n-\mathcal{V}\|_{\mathcal{S}_2(L^2([0,1]))}
\|\mathcal{V}\|_{\mathcal{S}_2(L^2([0,1]))}
\nonumber
\\
&=
\left(
\|\mathcal{V}_n\|_{\mathcal{S}_2(L^2([0,1]))}
+
\|\mathcal{V}\|_{\mathcal{S}_2(L^2([0,1]))}
\right)
\|\mathcal{V}_n-\mathcal{V}\|_{\mathcal{S}_2(L^2([0,1]))}.
\label{eq:Dn-TB-trace-bound} \\
&\le \frac{1}{\sqrt n}. \label{eq:Dn-TB-trace-bound-final}
\end{align}
In deriving \eqref{eq:Dn-TB-trace-bound-final} from \eqref{eq:Dn-TB-trace-bound}, we have reasoned as follows. Since the kernel of $\mathcal{V}$ (resp. $\mathcal{V}_n$) is $\1_{\{u\le s\}}$ (resp. $1_{\{ u \le \lfloor n s \rfloor / n\}}$), we have
\begin{equation} \label{eq:V-HS-calculation}
\|\mathcal{V}\|_{\mathcal{S}_2(L^2([0,1]))}^2
=
\int_0^1\int_0^1 \1_{\{u\le s\}}du\,ds
=
\frac12, \quad 
\|\mathcal{V}_n\|_{\mathcal{S}_2(L^2([0,1]))}^2
=
\int_0^1 \frac{\lfloor ns\rfloor}{n}ds
=
\frac{n-1}{2n}
\le
\frac12.
\end{equation}
Moreover, the kernel of $\mathcal{V}-\mathcal{V}_n$ is $\1_{\{\lfloor ns\rfloor/n<u\le s\}}$ and therefore
\[
\|\mathcal{V}_n-\mathcal{V}\|_{\mathcal{S}_2(L^2([0,1]))}^2
=
\int_0^1
\left(s-\frac{\lfloor ns\rfloor}{n}\right)ds
=
\frac{1}{2n},
\]
obtaining \eqref{eq:Dn-TB-trace-bound-final}.

For the second summand of \eqref{eq:Akn-decomp}, using again the decomposition $\mathcal{T}_B=\mathcal{V}\mathcal{V}^*$, we have
\begin{align} 
\bigl\|\bigl(S_a-(1-a)I\bigr)\mathcal{T}_B\bigr\|_{\mathcal{S}_1(L^2([0,1]))}
&=
\bigl\|\bigl(S_a-(1-a)I\bigr)\mathcal{V}\mathcal{V}^*\bigr\|_{\mathcal{S}_1(L^2([0,1]))}
\nonumber
\\
&\le
\bigl\|\bigl(S_a-(1-a)I\bigr)\mathcal{V}\bigr\|_{\mathcal{S}_2(L^2([0,1]))}
\|\mathcal{V}\|_{\mathcal{S}_2(L^2([0,1]))}
\label{eq:shift-defect-factorization}
\\
&=
\sqrt{a}(1-a)\frac{1}{\sqrt{2}}
\le
\sqrt{\frac{a}{2}},
\label{eq:shift-defect-factorization-2}
\end{align}
where \eqref{eq:shift-defect-factorization} is due to Problem 28(c) (Section 6) in \cite{reed1972methods}. The second term in \eqref{eq:shift-defect-factorization} is again handled by \eqref{eq:V-HS-calculation}.
Moreover, for the first factor in \eqref{eq:shift-defect-factorization}, the kernel of $\bigl(S_a-(1-a)I\bigr)\mathcal{V}$ is $\1_{\{u\le s-a\}}-(1-a)\1_{\{u\le s\}}$, therefore
\begin{align}
\bigl\|\bigl(S_a-(1-a)I\bigr)\mathcal{V}\bigr\|_{\mathcal{S}_2(L^2([0,1]))}^2
&=
\int_0^1\int_0^1
\left(
\1_{\{u\le s-a\}}-(1-a)\1_{\{u\le s\}}
\right)^2
 du ds
\nonumber
\\
&=
\int_a^1
\left(
\int_0^{s-a} a^2 du
+
\int_{s-a}^{s}(1-a)^2 du
\right)ds
+
\int_0^a \int_0^s (1-a)^2 du ds
\nonumber
\\
&=
\int_a^1 \left(a^2(s-a)+a(1-a)^2\right)ds
+
(1-a)^2\int_0^a s ds
\nonumber
\\
&=
a(1-a)^2.
\label{eq:shift-defect-HS}
\end{align}

Finally, using \eqref{eq:Akn-decomp}, \eqref{eq:Akn-decomp-1-1}, \eqref{eq:Dn-TB-trace-bound-final}, and \eqref{eq:shift-defect-factorization-2}, since $a = \frac{k}{n}$,
\[
\|\mathcal{A}_{kn}\|_{\mathcal{S}_1(L^2([0,1]))}
\le
\frac{1}{\sqrt n}
+
\sqrt{\frac{k}{2n}},
\qquad k=0,\dots,n-1.
\]
The same estimate holds with $k$ replaced by $|k|$, which proves \eqref{eq:Akn-trace-bound}. The final claim follows immediately from \eqref{eq:Akn-trace-bound}.
\end{proof}

The following lemma shows the nonnegativity of the operator given in \eqref{eq:two-operators}. 

\begin{lemma}\label{lem:Positivity of the block operators}
Fix \(p\in\NN\) and \(v\in\ZZ\).  
Recall the operator $R(v):\clh_1\to\clh_1$ given in \eqref{eq:R(v)-op} and its associated block operator $\mathcal R^{(p)}(v):
\clh_1^{\otimes p}\oplus \clh_1^{\otimes p}
\to
\clh_1^{\otimes p}\oplus \clh_1^{\otimes p}$ given in \eqref{eq:two-operators}.
Then \(\mathcal R^{(p)}(v)\) is nonnegative.
\end{lemma}

\begin{proof}
By the definition of \(R(v)\) in terms of the white-noise map, for
\(h,g\in\clh_1\),
\[
\begin{aligned}
|\langle R(v)h,g\rangle_{\clh_1}|^2
&=
\left|
\E\left[W_h(X_1)W_g(X_{1+v})\right]
\right|^2 \\
&\le
\E\left[|W_h(X_1)|^2\right]
\E\left[|W_g(X_{1+v})|^2\right]  \\
&=
\|h\|_{\clh_1}^2\|g\|_{\clh_1}^2,
\end{aligned}
\]
where we used Cauchy--Schwarz and stationarity. Hence
\(\|R(v)\|_{\op(\clh_1)}\le 1\). Moreover, \(R(0)=I_{\clh_1}\), and
\(R(-v)=R(v)^*\).

Therefore
\[
    \|R(v)^{\otimes p}\|_{\op(\clh_1^{\otimes p})}
    \le
    \|R(v)\|_{\op(\clh_1)}^p
    \le 1.
\]
The block operator in \eqref{eq:two-operators} may be written as
\[
\mathcal R^{(p)}(v)
=
\begin{pmatrix}
I_{\clh_1^{\otimes p}} & R(v)^{\otimes p} \\
R(-v)^{\otimes p} & I_{\clh_1^{\otimes p}}
\end{pmatrix}.
\]
Thus, for \((x,y)\in \clh_1^{\otimes p}\oplus \clh_1^{\otimes p}\),
\[
\begin{aligned}
\big\langle \mathcal R^{(p)}(v)(x,y),(x,y)\big\rangle_{\clh_1^{\otimes p}\oplus \clh_1^{\otimes p}}
&=
\|x\|_{\clh_1^{\otimes p}}^2
+
\|y\|_{\clh_1^{\otimes p}}^2
+
2\Re\langle R(v)^{\otimes p}y,x\rangle_{\clh_1^{\otimes p}} \\
&\ge
\|x\|_{\clh_1^{\otimes p}}^2
+
\|y\|_{\clh_1^{\otimes p}}^2
-
2\|x\|_{\clh_1^{\otimes p}}\|y\|_{\clh_1^{\otimes p}} \\
&=
\left(
\|x\|_{\clh_1^{\otimes p}}
-
\|y\|_{\clh_1^{\otimes p}}
\right)^2
\ge 0.
\end{aligned}
\]
Hence \(\mathcal R^{(p)}(v)\) is nonnegative.
\end{proof}

For the following lemma, let $\{X_k\}_{k \in \ZZ}$ be an $\clh$-valued stationary Gaussian stochastic process, with covariance operator $Q$ and autocorrelation function given by 
\begin{equation*}
    \rho_{rs}(l-k) = \E \left[ W_{e_r}(X_k) W_{e_s}(X_l)    \right].
\end{equation*}
The following result allows us to rewrite the normalized series with regard to an isonormal Gaussian process. The proof is very closely related to that of Proposition 7.2.3 of \cite{nourdin2012normal}.

\begin{lemma} \label{pro:isonormal}
    There exists a real separable Hilbert space $\mathfrak{H}$, as well as an isonormal Gaussian process over $\mathfrak{H}$, written $\{ X(h) : h \in \mathfrak{H} \}$, with the property that there exists a set 
$E = \{\varepsilon_{ik} : k \in \ZZ, i \in \NN \} \subset \mfh$ such that 1. $E$ generates $\mathfrak{H}$;
2. $\langle \varepsilon_{ik}, \varepsilon_{jl} \rangle_{\mathfrak{H}}
=
\rho_{ij}(l-k)$,
3. $W_{e_i}(X_k) = X(\varepsilon_{ik})$.
\end{lemma}

\begin{proof}
We start by defining $\mathfrak{H}$. First define a set $\cle$, such that $h \in \cle$ if and only if $h = \{ h_{ik}, k \in \ZZ, i \in \NN  \mid h_{ik} = 0   \text{ for all but finitely many } h_{ik} \}$. Equip the space $\cle$ with the inner product
\begin{align} \label{eq:h_inner_prod}
\langle f, g \rangle_{\mfh} 
= \sum_{\substack{k,l \in \ZZ \\ i,j \in \NN}} f_{ik} g_{jl} \rho_{ij}(l-k) 
= \sum_{\substack{k,l \in \ZZ \\ i,j \in \NN}} f_{ik} g_{jl} \E \left[ W_{e_i}(X_k)  W_{e_j}(X_l) \right],
\end{align}
whenever $f,g \in \cle$, and define $\mathfrak{H}$ as the closure of $\cle$ with respect to the inner product $\langle \cdot, \cdot \rangle_{\mfh}$.

We proceed with defining the generating set $E$. First, we take, for $k \in \ZZ, i \in \NN$, sequences of the form
\begin{equation*}
    \veps_{ik} = \{ \delta_{k l} \delta_{i j}, j \in \NN , l \in \ZZ \} = \{ \delta_{k l}  \langle u_i, u_j \rangle_{\clh_1}, j \in \NN, l \in \ZZ  \}.
\end{equation*}
Consider the set $E \doteq \{ \veps_{ik}, i \in \NN , k \in \ZZ \} \subset \mfh$. Clearly $E$ is a generating set for $\mfh$, and so Property 1 in the statement is satisfied. For $h \in \cle$, define 
\begin{equation} \label{eq:Xh}
    X(h) \doteq \sum_{k \in \ZZ,i \in \NN} h_{ik} W_{e_i}(X_k).
\end{equation}
Then, for $h \in \mfh$, select a sequence $h_n \in \cle$ converging to $h$ and take $X(h) = \lim_{n \to \infty} X(h_n)$, where the limit is understood in the $L^2(\Om)$ sense. We prove that $\{ X(h): h \in \mathfrak{H} \}$ is an isonormal process over $\mathfrak{H}$. Indeed, for $f,g \in \mfh$, from \eqref{eq:h_inner_prod} and \eqref{eq:Xh},
\begin{equation*}
\begin{split}
\E \left[ X(f) X(g) \right] &= \E \left[ \left(\sum_{k \in \ZZ,i \in \NN} f_{ik} W_{e_i}(X_k) \right) \times \left(    \sum_{l \in \ZZ,j \in \NN} g_{jl} W_{e_j}(X_l)\right)     \right] \\
&= \sum_{\substack{k,l \in \ZZ \\ i,j \in \NN}} f_{ik} g_{jl} \rho_{ij}(l-k) = \langle f,g \rangle_{\mathfrak{H}}
\end{split} 
\end{equation*}
and so, in particular, Property 2 follows. Finally, by construction, 
\begin{equation*}
    X(\veps_{ik}) = \sum_{l \in \ZZ,j \in \NN} \delta_{i j} \delta_{k l} W_{e_j}(X_l) = W_{e_i}(X_k),
\end{equation*}
such that Property 3 is also satisfied.
\end{proof}


\section*{Acknowledgements}

PZ was supported by the NSF grant DMS-2152577, from a dissertation completion fellowship from UNC's graduate school, and from the Deutsche Forschungsgemeinschaft (DFG, German
Research Foundation) under Germany’s Excellence Strategy EXC 2044/2 –390685587, Mathematics Münster:
Dynamics–Geometry–Structure.

The authors are grateful to Giovanni Peccati for pointing out an error in the literature that affected the proof presented in the original version of this article.

\small
\bibliographystyle{acm}
\bibliography{mybibliography}

\end{document}